\documentclass[]{siamltex}

\usepackage{graphicx}
\usepackage{amsmath}
\usepackage{amsfonts}
\usepackage[caption=false]{subfig}
\usepackage{multirow}

\newcommand{\pderiv}[3]{\frac{\partial^{#2} #1}{\partial #3^{#2}}}

\newcommand{\ppderiv}[6]{\frac{\partial^{#2} #1}{\partial #3^{#4} \partial #5^{#6}}}

\newcommand{\pppderiv}[8]{\frac{\partial^{#2} #1}{\partial #3^{#4} \partial #5^{#6} \partial #7^{#8}}}

\newcommand{\tpderiv}[3]{\tfrac{\partial^{#2} #1}{\partial #3^{#2}}}
\newcommand{\tppderiv}[6]{\tfrac{\partial^{#2} #1}{\partial #3^{#4} \partial #5^{#6}}}
\newcommand{\tpppderiv}[8]{\tfrac{\partial^{#2} #1}{\partial #3^{#4} \partial #5^{#6} \partial #7^{#8}}}

\newcommand{\pderivelip}[6]{\frac{\partial^{#2} {#1}} {\partial {#3}^{#4} \cdots \partial {#5}^{#6}}}

\newcommand{\abs}[1]{\lvert#1\rvert}

\newcommand{\Bigabs}[1]{\Bigl\lvert#1\Bigr\rvert}
\newcommand{\Biggabs}[1]{\Biggl\lvert#1\Biggr\rvert}

\providecommand{\norm}[1]{\lVert#1\rVert}
\newcommand{\Bignorm}[1]{\Bigl\lvert \Bigl\lvert #1 \Bigr\rvert \Bigr\rvert}
\newcommand{\Biggnorm}[1]{\Biggl\lvert \Biggl\lvert #1 \Biggr\rvert \Biggr\rvert}

\newcommand{\bigpars}[1]{\bigl(#1\bigr)}
\newcommand{\Bigpars}[1]{\Bigl(#1\Bigr)}
\newcommand{\Biggpars}[1]{\Biggl(#1\Biggr)}

\newcommand{\bigbracks}[1]{\bigl[#1\bigr]}
\newcommand{\Bigbracks}[1]{\Bigl[#1\Bigr]}
\newcommand{\Biggbracks}[1]{\Biggl[#1\Biggr]}

\providecommand{\norm}[1]{\lVert#1\rVert}

\newcommand{\innerprod}[2]{\langle #1, #2 \rangle}

\newcommand{\kronprodelip}[2]{{#1} \otimes \cdots \otimes {#2}}

\newcommand{\kronderivelip}[2]{ \kronprodelip{\pderiv{}{}{#1}}{\pderiv{}{}{#2}} }
\newcommand{\tkronderivelip}[2]{ \kronprodelip{\tpderiv{}{}{#1}} {\tpderiv{}{}{#2}}}

\newcommand{\Mmax}{M_{\text{max}}}
\newcommand{\Linh}{L_{\text{inh}}}

\newcommand{\compdom}{X_c}
\newcommand{\albpatch}{\mathcal{A}}
\newcommand{\closedcompdomwoalb}{\overline{X_c \setminus \mathcal{A}}}
\newcommand{\compdomwoalb}{X_c \setminus \mathcal{A}}
\newcommand{\lyupderiv}[1]{\pderiv{\pi}{}{x} #1 \dot{x}}
\newcommand{\tlyupderiv}[1]{\tpderiv{\pi}{}{x} #1 \dot{x}}
\newcommand{\dpidx}{\pderiv{\pi}{}{x}}
\newcommand{\tdpidx}{\tpderiv{\pi}{}{x}}

\newcommand{\dpibardx}{\pderiv{\bar{\pi}}{}{x}}

\newcommand{\tdpibardx}{\tpderiv{\bar{\pi}}{}{x}}

\newcommand{\dzdw}{\pderiv{z}{}{w}}

\newcommand{\tildepi}{\tilde{\pi}}
\newcommand{\tildekappa}{\tilde{\kappa}}
\newcommand{\tildexi}{\tilde{\xi}}
\newcommand{\tildef}{\tilde{f}}
\newcommand{\tildeg}{\tilde{g}}
\newcommand{\tildeq}{\tilde{q}}
\newcommand{\tilder}{\tilde{r}}

\newcommand{\hatpi}{\hat{\pi}}
\newcommand{\hatkappa}{\hat{\kappa}}
\newcommand{\hatxi}{\hat{\xi}}
\newcommand{\hatf}{\hat{f}}
\newcommand{\hatg}{\hat{g}}
\newcommand{\hatq}{\hat{q}}
\newcommand{\hatr}{\hat{r}}

\newcommand{\ith}[1]{{#1}^{\text{th}}}

\newcommand{\xdot}[2]{\dot{x}_{#1}^{#2}}

\title{The patchy Method for the Infinite Horizon Hamilton-Jacobi-Bellman Equation and its Accuracy \thanks{This work was partially supported by AFOSR}}

\author{Thomas Hunt\footnotemark[2]  \and Arthur J. Krener\footnotemark[2]}

\begin{document}
\maketitle

\renewcommand{\thefootnote}{\fnsymbol{footnote}}
\footnotetext[2]{Naval Postgraduate School, 833 Dyer Road
Monterey, CA 93943-5216 (twhunt@nps.edu, ajkrener@nps.edu)}

\renewcommand{\thefootnote}{\arabic{footnote}}

\begin{abstract}
We introduce a modification to the patchy method of Navasca and Krener for solving the stationary Hamilton Jacobi Bellman equation.
The numerical solution that we generate is a set of polynomials that approximate the optimal cost and optimal control on a partition of the state space.
We derive an error bound for our numerical method under the assumption that the optimal cost is a smooth strict Lyupanov function.
The error bound is valid when the number of subsets in the partition is not too large.
\end{abstract}



\section{Introduction} \label{sec:background:opt_cntrl_prob}
In this paper, we introduce a numerical scheme for solving the infinite horizon optimal control problem of minimizing the integral
\begin{subequations} \label{eqn:optml_cntrl_prblm}
\begin{equation} \label{eqn:lagrangian_integral}
\int_0^{\infty} l(x,u) dt
\end{equation}
of a Lagrangian $l(x,u)$ subject to the controlled dynamics
\begin{align}\label{eqn:dynamics}
\dot{x} = f(x,u) & & x(0) = x^0
\end{align}
\end{subequations}
where $f$ and $l$ are both smooth, and $l$ is strictly convex in the control $u \in \mathbb{R}^m$ for all states $x \in \mathbb{R}^n$.
The solution that we are interested in is an optimal cost $\pi : \mathbb{R}^n \rightarrow \mathbb{R}$, which is the minimum value of \eqref{eqn:lagrangian_integral} incurred by driving the state from an initial value of $x$ to the origin.
The optimal control that drives the state to the origin is given in feedback form by $u=\kappa(x)$.
Our solution is a set of polynomials that approximate the optimal cost and optimal control on a partition of a subset of the state space.

Our method is a modification of the original patchy method \cite{krener_navasca_patchy}, which we altered to obtain the error bound in Theorem \ref{thm:cmptd_exct_err_bnd}.
Like the original patchy method, it is an extension of the Cauchy-Kovalevskaya method \cite{pde_text2}, the fast marching method \cite{412624, sethian_text}, the patchy technique of Ancona and Bressan \cite{Ancona2007279}, and Al'brekht's power series method \cite{albrekht}.

Al'brekht's method is an algorithm for computing a series solution to the optimal cost that must be centered at the origin of the state space.
A drawback to the requirement that the series solution be centered at the origin is, even when the optimal cost and optimal control are smooth over the entire domain of interest, their power series solutions may be local in nature.
This means that outside of some region containing the origin, it may be impractical to compute the number of terms necessary to achieve a desired accuracy.
The patchy method that we present uses Al'brekht's series solution as an initialization step, and generates a new series solution centered at a point away from the origin.
Doing this enlarges the region where we have a valid approximation to the optimal cost and optimal control.

\subsection{The Hamilton-Jacobi-Bellman PDE}
It is widely known that if the optimal control problem \eqref{eqn:optml_cntrl_prblm} has a smooth optimal cost $\pi(x)$, and the optimal control can be put in feedback form $u=\kappa(x)$, then the optimal control and optimal cost satisfy the Hamilton-Jacobi-Bellman (HJB) equation
\begin{align} \label{eqn:HJB_PDE0}
\begin{split}
0 &= \min_u \Bigl( \frac{\partial \pi}{\partial x}(x) f(x,u) + l(x,u) \Bigr) \\
\kappa(x) &= \text{arg} \min_u{\Bigl(\frac{\partial \pi}{\partial x}(x) 
                               \frac{\partial f} {\partial u}(x,u) + 
                               \frac{\partial l} {\partial u}(x,u) \Bigr)}
\end{split}
\end{align}
locally around the origin.
Furthermore, if $\tfrac{\partial \pi}{\partial x}(x) f(x,u) + l(x,u)$ is strictly convex in
$u$ locally around $x=0$, $u=0$, then the HJB PDEs \eqref{eqn:HJB_PDE0} simplify to
\begin{subequations} \label{eqn:HJB_PDE}
\begin{align} 
0 &= \frac{\partial \pi}{\partial x}(x) f(x,\kappa(x)) + l(x,\kappa(x)) 
\label{eqn:HJB_PDE1} \\
0 &= \frac{\partial \pi}{\partial x}(x) 
                               \frac{\partial f} {\partial u}(x, \kappa(x)) + 
                               \frac{\partial l} {\partial u}(x,\kappa(x)).
\label{eqn:HJB_PDE2}
\end{align}
\end{subequations}

\subsection{Solutions to the Hamilton-Jacobi-Bellman PDE}
The existence of a solution of the optimal control problem \eqref{eqn:optml_cntrl_prblm} around the origin is determined by the leading order terms of the problem data $f$ and $l$.
Suppose the dynamics and Lagrangian have Taylor expansions
\begin{equation} \label{eqn:lagrangian_dynamics_series}
\begin{gathered}
\dot{x} = F x + G u + \sum_{k=2}^{\infty}f^{[k]}(x,u) \\
l(x,u) = \frac{1}{2}\bigl( x^T Q x + u^T R u \bigr) 
          + \sum_{k=3}^{\infty}l^{[k]}(x,u) \\
\end{gathered}
\end{equation}
where $^{[k]}$ denotes degree $k$ terms in the power series, $R$ is positive definite, and $Q$ is positive semidefinite. 
The optimal control problem is said to be \emph{nice} if $(F,G)$ is stabilizable and $(Q^{1/2},F)$ is controllable.

Al'brekht \cite{albrekht} showed that when the optimal control problem \eqref{eqn:HJB_PDE} is nice, the optimal cost and optimal control have series solution centered at the origin of the state space.
Lukes \cite{doi:10.1137/0307007} showed that both these series solutions converge to the optimal cost and control under suitable conditions. 
Al'brekht's method is an algorithm for computing both these series.
It works by computing the degree $j$ and $j-1$ terms as a pair in the expansions of the optimal cost and optimal control, starting at $j=2$ and terminating at some user specified degree.
The original description of the algorithm is in Al'brekht's original work \cite{albrekht}, and a generalization with notation consistent with this paper is in \cite{krener_navasca_patchy}.
The Matlab implementation of Al'brekht's method that we use in the patchy method is available by request \cite{art_code}.

\section{Derivation of the patchy method} \label{sec:drvtn}
\subsection{Optimal control problem assumptions}
\label{sec:drvtn:opt_cntrl_ass}
We will derive the patchy method for the special case of the optimal control problem where the control is scalar ($m=1$), and the dynamics and lagrangian have the form
\begin{align}
\label{eqn:drvtn_dynamics_lagrangian_form}
\dot{x} = f(x) + g(x) u
& &
l(x,u) = q(x) + \frac{1}{2} r(x) u^2
\end{align}
where
\begin{align*}
f,g : \mathbb{R}^n \rightarrow \mathbb{R}^n
& &
q,r : \mathbb{R}^n \rightarrow \mathbb{R}
& &
u \in \mathbb{R}
.
\end{align*}
The patchy method does not require these assumptions, but they do  significantly simplify its derivation.
Furthermore, we assume
\begin{enumerate}
\item
\label{ass:smooth_prob_data}
$f$, $g$, $q$, $r$ have continuous partial derivatives of to degree $d+2$ on $X_c$.
\item
\label{ass:nice_optml_cntrl_prob}
The optimal control problem is ``nice", meaning it can be expressed in the form of \eqref{eqn:lagrangian_dynamics_series}, and $(F,G)$ is stabilizable and $(F, Q^{1/2})$ is detectable.
We also assume that $Q$ is strictly positive definite.%
\item
\label{ass:smooth_pi}
$\pi$ has continuous partial derivatives up to order $d+2$ on $X_c$
\item
\label{ass:qr_pos_def}
$q \ge 0$ and $r>0$ on $X_c$
\item
\label{ass:strict_lyup_pi}
The optimal cost is a strict Lyapunov function on $\compdom$, meaning
\begin{equation}
\pderiv{\pi}{}{x}(x) \dot{x} = \pderiv{\pi}{}{x}(x) \bigpars{f(x) + g(x) \kappa(x)  } < 0
\end{equation}
for all $x$ in $X_c \setminus \{0\}$.
\end{enumerate}
Assumptions \ref{ass:smooth_prob_data} and \ref{ass:nice_optml_cntrl_prob} guarantee that assumption \ref{ass:smooth_pi} holds in some neighborhood of the origin.
The Taylor series of $\pi$ in this neighborhood can be computed by Al'brekht's method.
The fifth assumption is essential in both the derivation of the patchy method and the proof of its error bound in Theorem \ref{thm:cmptd_exct_err_bnd}.

\subsection{Overview of the patchy method} 

Although the method works in higher dimensions, we will assume in the overview that the state space is two dimensional.
This allows us to illustrate how we partition the state space.
We first compute the degree $d+1$ Taylor approximation to the optimal cost at the origin by Al'brekht's method, which we denote $\pi^0$.
We assume that all truncated series solutions to the optimal cost and optimal control are polynomials of degree $d+1$ and $d$. 
We then pick a sublevel set $\pi^0 \le c$ on which $\pi^0$ has an acceptable level of error.
We refer to this sublevel set as the \emph{Al'brekht patch} and denote it as $\albpatch$.
The patchy method expands the domain of the numerical solution to a superset of the Al'brekht patch, which we refer to as the computational domain and denote it as $\compdom$.
The patchy method picks a \emph{patch point}, denoted $x^i$, on the boundary of the Al'brekht patch and then computes an approximating polynomial to the optimal cost that is centered at the new patch point, which we denote as $\pi^i$.
The core of the patchy algorithm is a method to compute the approximate partial derivatives of the optimal cost at $x^i$, thus computing an approximation to the Taylor polynomial centered at $x^i$ of the optimal cost.
The patchy method repeats this process until the origin is surrounded with a number of patch points.

The computational domain is partitioned into a set of \emph{patches} during the course of the patchy algorithm.
Each patch point is associated with a single patch, and the optimal cost and optimal control at any state that falls inside in the patch are computed by the approximating polynomials centered at the associated patch point.
Every point in the computational domain is associated with exactly one patch, and we compute the approximate optimal cost and optimal control at the point by the series solutions centered at the associated patch point.
In two dimensions, the boundaries of the patch are the boundary of the Al'brekht patch, a level curve of the optimal cost centered at the patch point, and two \emph{lateral boundaries} that intersect the other two boundaries.
Illustrations of the Al'brekht patch and a typical patch in a two dimensional state space are included in Figure \ref{fig:albrekht_typical_patch}.

Once the Patchy method has encircled the Al'brekht patch with new patches, it generates a new ring of patches and associated optimal cost polynomials, where the new patches form a ring that encloses the old set of patches.

As we will see, under some standard assumptions on the underlying optimal control problem, as well as the assumptions that the true solution is both sufficiently smooth and a strict Lyapunov function, we can derive error estimates on the patchy method.
Unfortunately, these error estimates do not go to zero as the density of patches on the computational domain goes to infinity.
To our knowledge, the only algorithm that provably solves the nonlinear Hamilton-Jacobi-Bellman equations in multiple dimensions with higher order accuracy is the finite difference scheme due to Szpiro and Dupuis \cite{szpiro_dupuis}.
Their method is provably second order accurate on subsets of the computational domain where the true solution is smooth.

\subsection{Placing a new patch point}
Given the \emph{previous} patch point $x^i$ and its associated approximate optimal cost $\pi^i$, the patchy method must determine where in the state space to place the \emph{new} patch point $x^{i+1}$.
If the computed optimal cost $\pi^i$ is sufficiently close to the true optimal cost in a neighborhood of $x^i$, then by the assumption that the optimal cost is a strict Lyapunov function, the closed loop dynamics is not tangent to a level curve of $\pi^i$ in the neighborhood.
Thus the partial derivatives of the optimal cost can be calculated by the technique of this section at a point on an appropriately chosen level curve of $\pi^i$.
The patchy method places $x^{i+1}$ on a level curve of $\pi^i$ under the constraint that the distance between $x^i$ and $x^{i+1}$ is less than $h$, the \emph{maximum consecutive patch point distance}.
We derive the maximum consecutive patch point distance in chapters \ref{sec:lte_1} and \ref{sec:phi-1_lip}.

\subsection{Computing the optimal cost at a new patch point}
Away from the origin, the HJB equations alone do not fully specify all the partial derivatives of the optimal cost at a given patch point.
As a consequence, the patchy algorithm computes some partial derivatives of the optimal cost from the HJB equations, and computes the remaining partial derivatives from the previous optimal cost.

The computed optimal cost at the new patch point $x^{i+1}$ is calculated by \emph{inheritance}, meaning we set
\begin{equation*}
\pi^{i+1}(x^{i+1}) = \pi^i(x^{i+1})
\end{equation*}

\subsection{Computing the first order partial derivatives of the optimal cost away from the origin}
First, we compute $\mathbf{n}$, the gradient direction of $\pi^{i+1}$ at $x^{i+1}$ from $\pi^i$ by
\begin{equation}
\label{eqn:drvtn_nrmlzd_opt_cost_grdnt}
\mathbf{n} \equiv \frac{1}{\norm{\pderiv{\pi^i}{}{x}(x^{i+1})}} \pderiv{\pi^i}{}{x}(x^{i+1})
.
\end{equation}
If the level sets of the previous and new optimal costs are tangent, then the gradients of the previous and new optimal cost are collinear at the current patch point.
In this case, there is some positive scalar $z$ such that
\begin{equation}
\label{eqn:drvtn_crrnt_opt_cost_grad}
\pderiv{\pi^{i+1}}{}{x}(x^{i+1}) = z \mathbf{n}
\end{equation}
To compute $z$, we first define the scalars
\begin{align}
\label{eqn:drvtn_fn_gn_defn}
f_n \equiv \mathbf{n} \cdot f(x^{i+1}) & &\text{and} & & g_n \equiv \mathbf{n} \cdot g(x^{i+1})
.
\end{align}
We can derive a formula for the optimal control in terms of the gradient of the optimal cost by solving \eqref{eqn:HJB_PDE2} under the assumption \eqref{eqn:drvtn_dynamics_lagrangian_form}.
If we substitute this formula for the optimal control into \eqref{eqn:HJB_PDE1}, then it reduces to the scalar quadratic equation
\begin{equation}
\label{eqn:drvtn_sclr_hjb_eqn}
0 = -\frac{g_n^2}{2 r(x^{i+1})} z^2  + f_n z + q(x^{i+1})
\end{equation}
under \eqref{eqn:drvtn_crrnt_opt_cost_grad} and \eqref{eqn:drvtn_fn_gn_defn}.
If the computed and exact gradients of the optimal cost are sufficiently close, then it follows from the assumption that the exact optimal cost is a strict Lyapunov function that the quadratic equation \eqref{eqn:drvtn_sclr_hjb_eqn} has exactly one strictly positive root.
We denote this positive root as $z_+$, so the computed gradient of the new optimal cost and the computed new optimal control at the new patch point are
\begin{align} \label{eqn:drvtn_crnt_optml_cost_grdnt_optml_cntrl}
\pderiv{\pi^{i+1}}{}{x}(x^{i+1})
=
z_+ \mathbf{n} 
& &
\kappa^{i+1}(x^{i+1})
=
-\frac{1}{r(x^{i+1})} \pderiv{\pi^{i+1}}{}{x}(x^{i+1}) g(x^{i+1})
\end{align}

\subsection{Computing higher order partial derivatives of the optimal cost away from the origin}
\label{sec:hghr_ordr_prtls_off_alb}
Let $\dot{x}^{i+1}$ denote the computed \emph{optimal direction} at the new patch point calculated from the first order partial derivatives of the computed optimal cost.
If the exact and computed gradients of the optimal cost are sufficiently close at the new patch point, then it follows from the assumption that the optimal cost is a strict Lyapunov function that the computed optimal direction is nonzero.
The computed optimal direction is calculated from the formula
\begin{equation*}
\dot{x}^{i+1} = f(x^{i+1}) + g(x^{i+1}) \kappa^{i+1}(x^{i+1})
.
\end{equation*}
We set 
\begin{align*}
\hat{V}^1 = \frac{1}{\norm{\dot{x}^{i+1}}} \dot{x}^{i+1}
& &
\hat{V} \equiv \begin{bmatrix} \hat{V}^1 \cdots \hat{V}^n \end{bmatrix}
\end{align*}
where $\hat{V}^1, \hat{V}^2, \dotsc, \hat{V}^n $ form an orthonormal basis of $\mathbb{R}^n$.
This can be done with a Householder reflector \cite[ch.~5.1.2]{golub_van_loan:mat_comps}.
We will compute the higher order partial derivatives of the optimal cost under the change of variables
\begin{equation*}
x = x^{i+1} + \hat{V} \hatxi
\end{equation*}
and then recover the partial derivatives with respect to the original state space variables.
In the new variables, the HJB equations \eqref{eqn:HJB_PDE} become
\begin{subequations} \label{eqn:drvtn_trnsfrms_HJB}
\begin{align}
\label{eqn:drvtn_trnsfrms_HJBa}
0
&=
\pderiv{\hatpi}{}{\hatxi}(\hatxi)
\bigpars{\hatf(\hatxi) + \hatg(\hatxi) \hatkappa(\hatxi)}
+ \hatq(\hatxi) + \frac{1}{2}\hatr(\hatxi) \bigpars{\hatkappa(\hatxi)}^2 \\
\label{eqn:drvtn_trnsfrms_HJBb}
0
&=
\pderiv{\hatpi}{}{\hatxi}(\hatxi) \hatg(\hatxi) + \hatr(\hatxi) \hatkappa(\hatxi)
\end{align}
\end{subequations}
where
\begin{align}
\begin{aligned}
\hatpi(\hatxi) &\equiv \pi(x^{i+1} + \hat{V} \hatxi) \\
\hatf(\hatxi) &\equiv \hat{V}^T f(x^{i+1} + \hat{V} \hatxi) \\
\hatq(\hatxi) &\equiv q(x^{i+1} + \hat{V} \hatxi) 
\end{aligned}
& &
\begin{aligned}
\hatkappa(\hatxi) &\equiv \kappa(x^{i+1} + \hat{V} \hatxi) \\
\hatg(\hatxi) &\equiv \hat{V}^T g(x^{i+1} + \hat{V} \hatxi) \\
\hatr(\hatxi) &\equiv r(x^{i+1} + \hat{V} \hatxi)
\end{aligned}
\end{align}
To derive a formula for the \emph{characteristic} second order partial derivatives of the optimal cost, we evaluate the derivative of the first HJB equation in \eqref{eqn:drvtn_trnsfrms_HJBa} with respect to $\hatxi_j$ at $\hatxi=0$, and arrive at
\begin{equation} \label{eqn:scnd_drvtv_optml_cst_exct}
0
=
\norm{\dot{x}} \ppderiv{\hatpi}{2}{\hatxi_j}{}{\hatxi_1}{}(0)
+
\pderiv{\hatpi}{}{\hatxi}{}(0)
\Bigpars{
\pderiv{\hatf}{}{\hatxi_j}(0) + \pderiv{\hatg}{}{\hatxi_j}(0)\hatkappa(0)
}
+
\pderiv{\hatq}{}{\hatxi_j}(0)
+
\frac{1}{2} \pderiv{\hatr}{}{\hatxi_j}(0) \bigpars{\hatkappa(0)}^2
\end{equation}
where $\dot{x}$ denotes the optimal direction at $x^{i+1}$.
The terms involving the partial derivative of $\hatkappa$ with respect to $\hatxi_j$ drop out due to  \eqref{eqn:drvtn_trnsfrms_HJBb}.
We compute the \emph{characteristic} second order partial derivatives of $\hatpi^{i+1}$ by substituting in the computed approximations to the optimal direction and gradient of the optimal cost into \eqref{eqn:scnd_drvtv_optml_cst_exct}, and solving for the unknown second order partial derivatives of the optimal cost.
This yields the formula
\begin{multline} \label{eqn:ddpiddxi_cmptd}
\ppderiv{\hatpi^{i+1}}{2}{\hatxi_j}{}{\hatxi_1}{}(0)
=
-\frac{1}{\norm{\dot{x}^{i+1}} }
\Biggbracks{
 \pderiv{\hatpi^{i+1}}{}{\hatxi}{}(0)
\Bigpars{
\pderiv{\hatf}{}{\hatxi_j}(0) + \pderiv{\hatg}{}{\hatxi_j}(0)\hatkappa^{i+1}(0)
}
+
\pderiv{\hatq}{}{\hatxi_j}(0)
\\
+
\frac{1}{2} \pderiv{\hatr}{}{\hatxi_j}(0) \bigpars{\hatkappa^{i+1}(0)}^2
}
.
\end{multline}
The characteristic second order partial derivatives are a strict subset of all the second order partial derivatives, since the HJB equations do not specify all the second order partial derivatives.
We call the remaining second order partial derivatives \emph{non characteristic}.
We compute them by inheritance, meaning for $2 \le j_1 \le j_2 \le n$ we set
\begin{equation}
\ppderiv{\hatpi^{i+1}}{2}{\hatxi_{j_1}}{}{\hatxi_{j_2}}{}(0)
=
\ppderiv{}{2}{\hatxi_{j_1}}{}{\hatxi_{j_2}}{}
\bigbracks{\pi^i(x^{i+1} + \hat{V} \hatxi)}_{\hatxi=0}
.
\end{equation}
To get the approximate first order partial derivatives of the optimal control, we follow the same process.
We differentiate the second HJB equation\eqref{eqn:drvtn_trnsfrms_HJBb} with respect to $\hatxi_j$ at $\hatxi=0$, and after substituting in the computed partial derivatives of the optimal control and cost, we obtain the formula

\begin{equation} \label{eqn:dkappa_dxi_cmptd}
\pderiv{\hatkappa^{i+1}}{}{\hatxi_j}(0)
=
-\frac{1}{\hatr(0)}
\Biggbracks{
\ppderiv{\hatpi^{i+1}}{2}{\hatxi_j}{}{\hatxi}{}(0) \hatg(0)
+
\pderiv{\hatpi^{i+1}}{}{\hatxi}(0) \pderiv{\hatg}{}{\hatxi_j}(0)\hatkappa^{i+1}(0)
+
\pderiv{\hatr}{}{\hatxi_j}(0)\hatkappa^{i+1}(0)
}
.
\end{equation}
We compute the remaining higher order partial derivatives in an analogous fashion.
Their formulas are in \S \ref{sec:deriv_formulas}.

We now introduce new notation to describe how to  recover the partial derivatives with respect to the original state space variables from the partial derivatives with respect to the new variables.
They are computed by
\begin{align} \label{eqn:opt_cost_pderiv_rcvry_frmls}
\begin{split}
\ppderiv{\pi^{i+1}}{2}{x_{j_1}}{}{x_{j_2}}{}(x^{i+1})
&=
\pderiv{}{}{\hatxi} \otimes \pderiv{}{}{\hatxi}
\Bigbracks{\hatpi^{i+1}(\hatxi)}_{\hatxi=0}
(\hat{V}^T e_{j_1} \otimes \hat{V}^T e_{j_2})
\\
& \mspace{11 mu} \vdots 
\\
\pderivelip{\pi^{i+1}}{k}{x_{j_1}}{}{x_{j_k}}{}(x^{i+1})
&=
\kronprodelip{\pderiv{}{}{\hatxi}}{\pderiv{}{}{\hatxi}}
\Bigbracks{\hatpi^{i+1}(\hatxi)}_{\hatxi=0}
(\kronprodelip{\hat{V}^T e_{j_1}}{\hat{V}^T e_{j_k}})
\end{split}
\end{align}
where $e_j$ is the $\ith{j}$ column of the identity matrix.
Each formula in \eqref{eqn:opt_cost_pderiv_rcvry_frmls} is a standard dot product, written as the product of a row and column vector.
The symbol $\otimes$ has two meanings in \eqref{eqn:opt_cost_pderiv_rcvry_frmls}.
It is the standard Kronecker product if it appears between two matrices.
The row vector $\kronprodelip{\tpderiv{}{}{\hatxi}}{\pderiv{}{}{\hatxi}} [\hatpi^{i+1}(\hatxi)]$ is new notation for the \emph{Kronecker derivative}, which is a bookkeeping mechanism.
In the first equation in \eqref{eqn:opt_cost_pderiv_rcvry_frmls}, if the state space is two dimensional then the Kronecker derivative is defined as
\begin{equation*}
\pderiv{}{}{\hatxi} \otimes \pderiv{}{}{\hatxi}
\Bigbracks{\hatpi^{i+1}(\hatxi)}_{\hatxi=0}
=
\begin{bmatrix}
\pderiv{\hatpi}{2}{\hatxi_1}(\hatxi)
&
\ppderiv{\hatpi}{2}{\hatxi_1}{}{\hatxi_2}{}(\hatxi)
&
\ppderiv{\hatpi}{2}{\hatxi_1}{}{\hatxi_2}{}(\hatxi)
&
\pderiv{\hatpi}{2}{\hatxi_2}(\hatxi)
\end{bmatrix}_{\hat{\xi}=0}
.
\end{equation*}
We give the full definition of the Kronecker derivative in the appendix.

\subsection{Computing the optimal cost and optimal control away from a patch point}
To compute the optimal cost or optimal control at a point in the state space that is not a patch point, we need a means to determine which patch the point belongs to.
In other words, we need to know which polynomials $\pi^i$ and $\kappa^i$ to use to calculate the optimal cost and optimal control at some $\bar{x}$ in the state space.
There are two types of patches, \emph{typical} patches, and a single Al'brekht patch.
The boundary of the Al'brekht patch is a level curve of $\pi^0$ that encloses the entire patch.
The optimal cost and optimal control at a point inside the Al'brekht patch are computed from $\pi^0$ and $\kappa^0$.
A typical patch is associated with the patch point $x^{i+1} \ne 0$, and the optimal cost and optimal control at every point in the patch is computed from $\pi^{i+1}$ and $\kappa^{i+1}$.
When the state space is two dimensional, a typical patch has four intersecting boundaries that enclose the patch.
Two of the boundaries are level curves of the optimal cost, and the other two boundaries are \emph{lateral} boundaries.
The patch point $x^{i+1}$ lies on the optimal cost level curve of $\pi^i$ that defines one of the patch boundaries.
The other optimal cost level curve patch boundary is a level curve of $\pi^{i+1}$.
Each lateral boundary is a straight line segment that intersects the level curve of $\pi^i$ at a point $x_b$, and is oriented so that it points in the opposite direction of the optimal direction $f(x_b) + g(x_b) \kappa^i(x_b)$.
When the state space is two dimensional, the \emph{patch boundary point} $x_b$ is placed on a level curve of the optimal cost so that it is halfway between two adjacent patch points that lie on the same level curve.
The optimal direction at $x_b$ may be computed from the approximate optimal cost associated with either adjacent patch point.
Illustrations of the two types of patches are in Figure \ref{fig:albrekht_typical_patch}.
\begin{figure}[h]
\fbox{\subfloat[Al'brekht patch]{\includegraphics[keepaspectratio, height=1.4in]{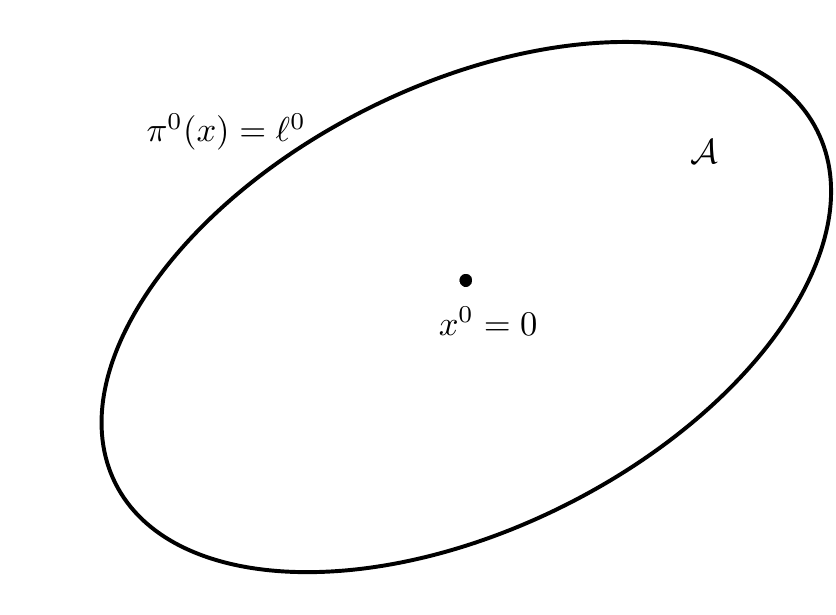}}}
\fbox{\subfloat[Typical patch]{\includegraphics[keepaspectratio, height=1.4in]{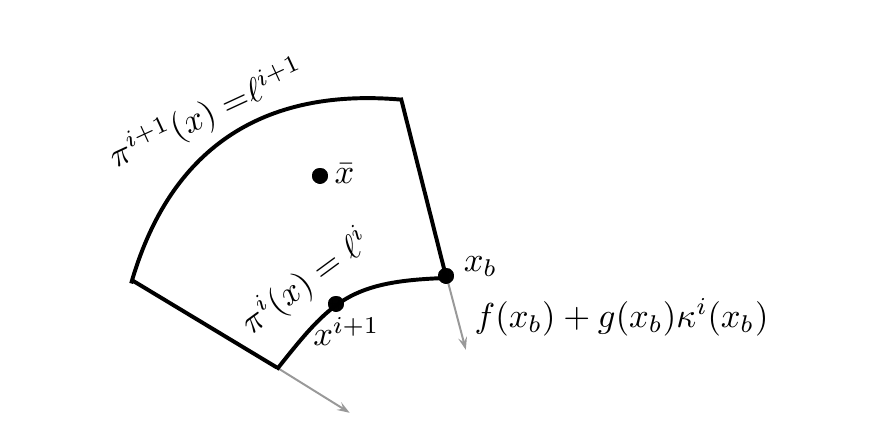}}}
\caption{Al'brekht patch and typical patch}
\label{fig:albrekht_typical_patch}
\end{figure}

\section{Analytical error bound for the patchy method}
\label{sec:anyl_err_bnd}
The difference between the exact and computed optimal cost grows as a power of the maximum consecutive patch point distance times a term that grows exponentially as the patchy method moves along a sequence of consecutive patch points.
We first develop some concepts and lemmas that we will need before making the more rigorous statement of the error bound in Theorem \ref{thm:cmptd_exct_err_bnd} at the end of this section.

\subsection{Sequence of consecutive patch points}
The error bound  is predicated on the notion of a \emph{sequence of consecutive patch points}.
\begin{definition}
\label{defn:cnsctv_ppt_sqnc}
The patch points $x^i$ and $x^{i+1}$ are said to be \emph{consecutive} if the computed optimal cost polynomial centered at $x^{i+1}$ was computed from information contained in the computed optimal cost polynomial centered at $x^i$.
The origin is always the first patch point in a sequence of consecutive patch points.
A sequence of three consecutive patch points is illustrated in Figure \ref{fig:cnsctv_ppts}.
\begin{figure}[h]
\centering
\fbox{\includegraphics[width=5.0in]{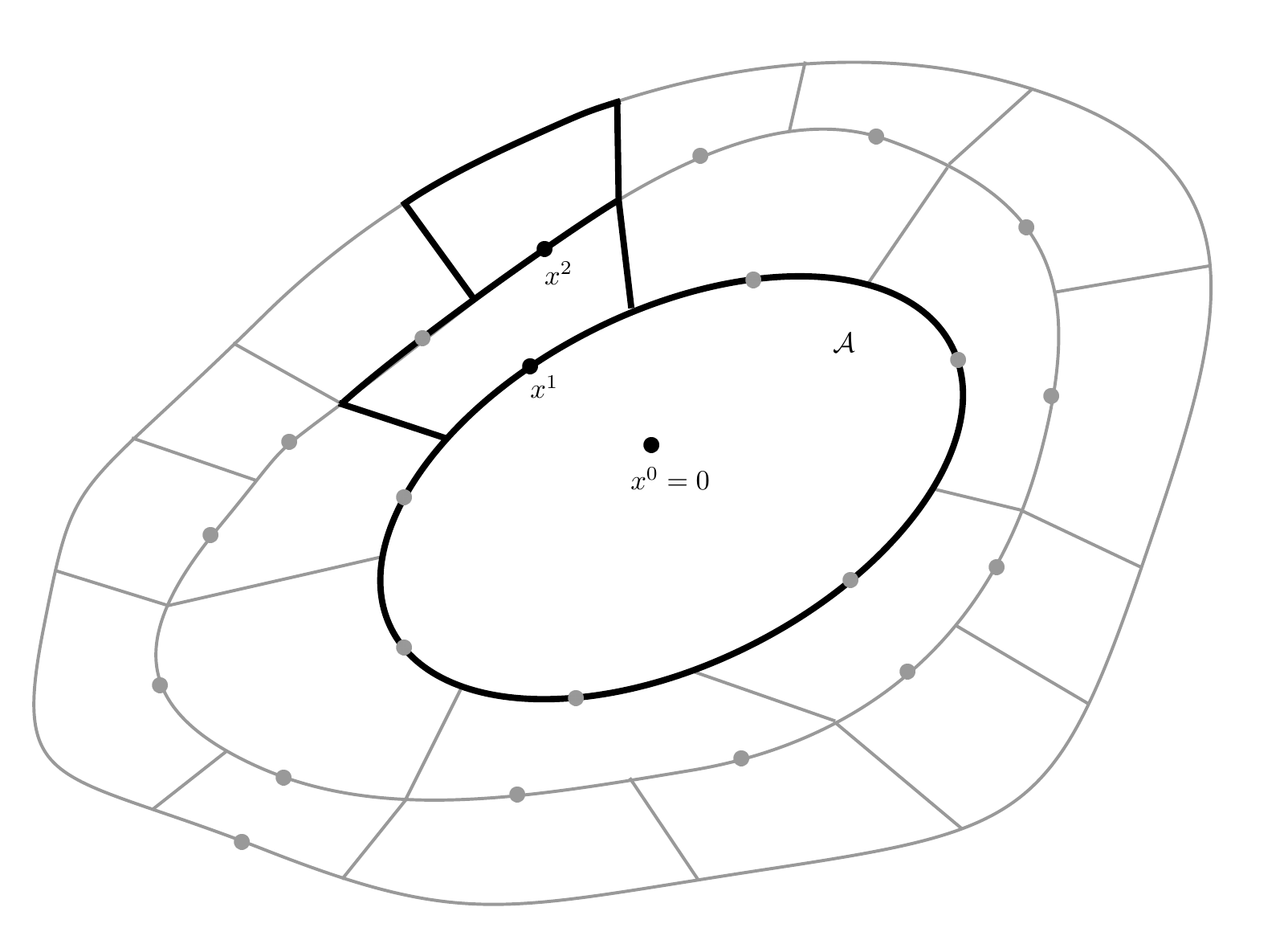}}
\caption{A sequence of three consecutive patch points $x^0, x^1, x^2$}
\label{fig:cnsctv_ppts}
\end{figure}
\end{definition}

\subsection{Notation} \label{sec:notation}
The patchy algorithm takes a truncated series solution to the optimal cost centered at the previous patch point, and generates a new series solution to the optimal cost centered at the next patch point.
Thus, it can be interpreted as a mapping from a set of polynomial coefficients associated with the previous patch point to a set of polynomial coefficients associated with the new patch point.
These coefficients are approximate partial derivatives of the optimal cost evaluated at the appropriate patch point, so we introduce notation to reflect this.

$\hat{C}_j^i$ denotes the vector of computed $\ith{j}$ partial derivatives of the optimal cost at $x^i$.
For example, if the state space is two dimensional
\begin{equation*}
\hat{C}_2^i = 
\begin{bmatrix}
\pderiv{\hat{\pi}}{2}{x_1}(x^i) &
\ppderiv{\hat{\pi}}{2}{x_1}{}{x_2}{}(x^i) &
\pderiv{\hat{\pi}}{2}{x_2}(x^i) 
\end{bmatrix}
\end{equation*}
$\hat{C}^i$ is the vector of all partial derivatives of the computed optimal cost at $x^i$ up to order $d+1$, that is
\begin{equation*}
\hat{C}^i = 
\begin{bmatrix}
\hat{C}_0^i & \hat{C}_1^i & \cdots & \hat{C}_{d+1}^i
\end{bmatrix}
\end{equation*}

We denote with $\phi_j$ the part of the patchy method that computes the $\ith{j}$ order partial derivatives of the optimal cost at the current patch point from the partial derivatives of the optimal cost evaluated at the previous patch point, meaning
\begin{equation*}
\hat{C}_j^{i+1} = \phi_j(\hat{C}^i)
\end{equation*}

$C_j^i$ denotes the vector of exact $\ith{j}$ partial derivatives of the optimal cost $\pi$ at $x^i$, and $C^i$ has a definition that is analogous to $\hat{C}^i$.

\subsection{Local truncation vector}
\begin{definition}
\label{defn:loc_trunc_vec}
$C^i$ denotes the vector of exact partial derivatives of the optimal cost at $x^i$, and $C_j^{i+1}$ denotes the exact partial derivatives of order $j$ of the optimal cost at $x^{i+1}$, where $x^i$ and $x^{i+1}$ are consecutive patch points.
The local truncation vector $\tau_j^i$ is
\begin{equation*}
\tau_j^i \equiv C^{i+1}_j - \phi_j(C^i)
\end{equation*}
which is the difference between the exact $\ith{j}$ order partial derivatives of $\pi$  at $x^{i+1}$, and the computed partial derivatives of $\pi$ at $x^{i+1}$ that are generated by feeding $\phi_j$ the exact partial derivatives of of $\pi$ centered at $x^i$.
\end{definition}

\subsection{The Lipschitz condition}
\begin{definition}
\label{defn:lip_cond}
Suppose $C^i_j$ and $\hat{C}^i_j$ are vectors of exact and approximate partial derivatives of order $j$ of the optimal cost $\pi$ at $x^i$.
The function $\phi$ is said to satisfy the Lipschitz condition if there exists a \emph{maximum multiplier} $\Mmax > 0$ and a maximum consecutive patch point distance $h>0$, such that for all $M \le \Mmax$ and all $\norm{x^{i+1} - x^i} \le h$, then
\begin{gather*}
\norm{\phi_k(C^i) - \phi_k(\hat{C}^i)}_2 \le L M \norm{x^{i+1}-x^i}_2^{d+2-k}
\text{ for } 
0 \le k \le d+1
\intertext{whenever}
\norm{C^i_j - \hat{C}^i_j}_2 \le M \norm{x^{i+1}-x^i}^{d+2-j}_2
\text{ for } 
0 \le j \le d+1
\end{gather*}
\end{definition}

The following lemma will help us establish that the patchy method satisfies the Lipschitz condition.
\begin{lemma}
\label{lem:poly_lip}
Polynomials satisfy the Lipschitz condition of Definition \ref{defn:lip_cond} with $\Mmax \le \infty$.
We may write any degree $d+1$ polynomial as
\begin{equation*}
P(C,x) \equiv \sum_{\ell =0}^{d+1} C_{\ell} \cdot m_{\ell}(x).
\end{equation*}
where $C_{\ell}$ is a coefficient vector as defined in \S \ref{sec:notation} and $m_{\ell}(x)$ is a vector of $\ith{\ell}$ degree monomials so that all the degree $\ell$ terms of the polynomial are $C_{\ell} \cdot m_{\ell}(x)$.
Thus, if $\norm{C_{\ell} - \hat{C}_{\ell}}_2 \le M \norm{x}_2^{d+2-\ell}$, then there exists $\Linh < \infty$ such that
\begin{equation*}
\Bigabs{\frac{\partial^k P}{\partial x_{j_1} \partial x_{j_2} \cdots \partial x_{j_k}}(C,x) 
- \frac{\partial^k P}{\partial x_{j_1} \partial x_{j_2} \cdots \partial x_{j_k}}(\hat{C},x) }
\le 
\Linh M \norm{x}_2^{d+2-k}
\end{equation*}

\begin{proof}
Since $P(C,x)-P(\hat{C},x)=P(C-\hat{C},x)$, we may assume without a loss of generality that $\hat{C}=0$, and $\norm{C_{\ell}}_2 \le M \norm{x}_2^{d+2-\ell}$.
By Cauchy-Schwarz, 
\begin{equation*}
\Bigabs{\frac{\partial^k P}{\partial x_{j_1} \partial x_{j_2} \cdots \partial x_{j_k}}(C,x) }
\le \sum_{\ell=k}^{d+1} \norm{C_{\ell}} \Bignorm{\frac{\partial^k m_{\ell}}{\partial x_{j_1} \partial x_{j_2} \cdots \partial x_{j_k}}(x)}.
\end{equation*}
Each entry in the $\ith{k}$ order partial derivative of $m_{\ell}$ is $\mathcal{O}(\norm{x}^{\ell-k})$, and by assumption $\norm{C_{\ell}}_2 \le M \norm{x}_2^{d+2-\ell}$, so each term in the sum is $\mathcal{O}(\norm{x}^{d+2-k})$.
\end{proof}
\end{lemma}

The next lemma establishes a bound on the difference between the Taylor polynomial coefficients of the optimal cost $\pi$, and the coefficients of the computed optimal cost $\pi^i$.

\begin{lemma}
\label{lem:opt_cost_coef_diff_bnd}
Suppose there exists a maximum consecutive patch point distance $h>0$ such that $\phi$ satisfies the Lipschitz condition with constants $L$ and $\Mmax$.
Also, suppose there exists a local truncation error constant $T$ such that for any patch points $x^i$ and $x^{i+1}$ in $\closedcompdomwoalb$ within a radius of $h$ of each other,
\begin{equation*}
\norm{\tau_j^i} \le T h^{d+2-j}
\end{equation*}
where the approximating polynomials to the optimal cost are degree $d+1$.
Let $\hat{C}_j^i$ denote the vector of $\ith{j}$ order partial derivatives of the optimal cost computed by the patchy method, where $i$ is the sequence index for a sequence of consecutive patch points.
Let $C_j^i$ denote the corresponding vector of exact partial derivatives.
The growth in the difference between the true and computed $\ith{j}$ partial derivatives of the optimal cost is bounded by
\begin{gather*}
\norm{C_j^i - \hat{C}_j^i} \le \frac{L^i - 1}{L-1} T h^{d+2-j}
\text{ for } 1 \le i \le N
\intertext{where $N$ is the smallest integer such that}
\frac{L^N - 1}{L-1} T \le \Mmax
\end{gather*}
\begin{proof}
The proof proceeds by induction on the sequence index $i$, and is similar in spirit to the construction of an error bound between the true solution of an ODE and a solution computed by a one step method \cite[ch.~7.2]{stoer_bulirsch:nmrcl_anly}.
We assume that all arithmetic is exact.
It follows from the definition of the local truncation vector that the difference between the true and computed $\ith{j}$ order partial derivatives is bounded by
\begin{equation*}
\norm{C_j^{i+1} - \hat{C}_j^{i+1}} \le \norm{\phi_j(C^i) - \phi_j(\hat{C}^i)} + \norm{ \tau_j^i} 
\end{equation*}
At the origin, the patchy algorithm computes the exact partial derivatives of the optimal cost by Al'brekht's method, so $\hat{C}_j^0 = C_j^0$ for $0 \le j \le d+1$.
Therefore the theorem holds in the base case of the induction proof.
Now, suppose that at the $\ith{i}$ patch point in the sequence of consecutive patch points,
\begin{align*}
\norm{C_j^i - \hat{C}_j^i} \le \frac{L^i - 1}{L-1} T h^{d+2-j}
& & \text{and} & &
\frac{L^i - 1}{L-1} T  \le \Mmax
\end{align*}
for $0 \le j \le d+1$.
Then
\begin{align*}
\norm{\phi_j(C^i) - \phi_j(\hat{C}^i)} \le L \frac{L^i - 1}{L-1} T h^{d+2-j}
\end{align*}
for $0 \le j \le d+1$ since $\phi$ satisfies the Lipschitz condition.
Therefore, at the next patch point in the sequence, the difference between the partial derivatives is bounded by
\begin{equation*}
\norm{C_j^{i+1} - \hat{C}_j^{i+1}}
\le
L \frac{L^i - 1}{L-1} T h^{d+2-j} + T h^{d+2-j} \\
=
\frac{L^{i+1}-1}{L-1} T h^{d+2-j}
.
\end{equation*}
Thus, the conclusion of the lemma holds.
\end{proof}
\end{lemma}

\subsection{Analytical bound on absolute error of the optimal cost}
\begin{theorem}
\label{thm:cmptd_exct_err_bnd}
Suppose there exists a maximum consecutive patch point distance $h>0$ such that $\phi$ satisfies the Lipschitz condition.
Also, suppose there exists a local truncation error constant $T$ such that for any consecutive patch points $x^i$ and $x^{i+1}$ in $\closedcompdomwoalb$ within a radius $h$ of each other, $\norm{\tau_j^i} \le T h^{d+2-j}$, where the degree of all the optimal cost approximating polynomials is $d+1$.
Finally, assume that the maximum patch diameter is proportional to the maximum consecutive patch point distance.
Then there exist constants $L$ and $K$ such that for any $x$ in the $\ith{i}$ patch in a sequence, the error between the exact and computed optimal cost is bounded by
\begin{equation*}
\abs{\pi(x) - \pi^i(x)} \le K \Bigl( \frac{L^i-1}{L-1} + 1 \Bigr) h^{d+2}
\end{equation*}
and the difference between corresponding $\ith{k}$ order partial derivatives is bounded by
\begin{equation*}
\Bigabs{
\pderivelip{\pi}{k}{x_{j_1}}{}{x_{j_k}}{}(x) 
-
\pderivelip{\pi^i}{k}{x_{j_1}}{}{x_{j_k}}{}(x)
}
\le K \Bigl( \frac{L^i-1}{L-1} + 1 \Bigr) h^{d+2-k}
.
\end{equation*}
\begin{proof}
The basic idea behind the proof is to write the true optimal cost as the sum of its Taylor polynomial and a remainder term, then apply Lemma \ref{lem:opt_cost_coef_diff_bnd} to bound the difference between the coefficients of $\pi^i$ and the coefficients of the Taylor polynomial.
By Lemma \ref{lem:opt_cost_coef_diff_bnd}, there is some Lipschitz constant $L$ such that the exact and computed $\ith{k}$ order partial derivatives of the optimal cost at $x^i$ are bounded by
\begin{equation*}
\norm{ C_k^i - \hat{C}_k^i }
\le
\frac{L^i-1}{L-1} T h^{d+2-k}
\end{equation*}
where $T$ is the local truncation constant from the theorem statement.
Let $\bar{x}$ be any point in the $\ith{i}$ patch in a sequence of consecutive patches, and let $D$ denote the maximum patch diameter for all the patches in the computational domain $\compdom$. 
By Lemma \ref{lem:poly_lip}, there exists $\Linh$ such that
\begin{equation*}
\Bigabs{
\pderivelip{}{k}{x_{j_1}}{}{x_{j_k}}{} \bigbracks{P(C^i - \hat{C}^i, x-x^i) }_{x=\bar{x}}
}
\le
\Linh \frac{L^i-1}{L-1} T D^{d+2-k}
\end{equation*}
By assumption, the optimal cost is smooth on $\compdom$, so it follows from the definition of the Taylor remainder constant $R_T$ that
\begin{equation*}
\Bigabs{
\pderivelip{}{k}{x_{j_1}}{}{x_{j_k}}{} \bigbracks{\pi(x-x^i) - P(C^i, x-x^i)}_{x=\bar{x}}
}
\le R_T D^{d+2-k}
.
\end{equation*}
Therefore
\begin{multline*}
\Bigabs{
\pderivelip{\pi}{k}{x_{j_1}}{}{x_{j_k}}{}(x)
-
\pderivelip{\pi^i}{k}{x_{j_1}}{}{x_{j_k}}{}(x)
} \\
\begin{aligned}
&\le
\Bigabs{
\pderivelip{}{k}{x_{j_1}}{}{x_{j_k}}{} \bigbracks{P(C^i - \hat{C}^i, x-x^i) }_{x=\bar{x}}
}
\\
&\phantom{\le \vert}
+ \Bigabs{
\pderivelip{}{k}{x_{j_1}}{}{x_{j_k}}{} \bigbracks{\pi(x-x^i) - P(C^i, x-x^i)}_{x=\bar{x}}
} \\
&\le
\Bigpars{\Linh \frac{L^i-1}{L-1} T + R_T} \Bigpars{\frac{D}{h}}^{d+2-k} h^{d+2-k}
.
\end{aligned}
\end{multline*}
The error bound from the theorem statement follows by setting $K =  (D/h)^{d+2} \max(T \Linh, R_T)$.
Our assumption that the patch diameter $D$ is proportional to the maximum consecutive patch point distance $h$ guarantees that $D/h$ is bounded from above by a constant.
\end{proof}
\end{theorem}

In sections \ref{sec:lte_1}-\ref{sec:lip_k}, we list sufficient conditions that guarantee that the local truncation error of the patchy algorithm has the desired order, and the patchy algorithm satisfies the Lipschitz condition, thus guaranteeing the error bound in the previous theorem.

\section{Local truncation error for first order partial derivatives of the optimal cost} \label{sec:lte_1}
We state Theorem \ref{thm:phi1_lte_bound} whose content implies the first order local truncation vector is bounded by $\norm{\tau_1^i} \le T h^{d+1}$, and defer its proof until after we have established some supporting lemmas.
In this section, we always assume that the conditions on the optimal control problem in \S \ref{sec:drvtn:opt_cntrl_ass} hold.
\begin{theorem}
\label{thm:phi1_lte_bound}
Suppose $\pi$ solves the HJB equations \eqref{eqn:HJB_PDE} under the assumptions in \S \ref{sec:drvtn:opt_cntrl_ass}.
Let $\bar{\pi}$ denote the Taylor polynomial of the optimal cost centered at the previous patch point, and let $\tpderiv{\pi^{i+1}}{}{x}(x^{i+1})$ denote the computed gradient of the optimal cost at $x^{i+1}$, computed from the coefficients of $\bar{\pi}$.
There exists a ball centered at the origin in which the Al'brekht patch must lie, a maximum consecutive patch point distance, and a constant $T$ such that
\begin{equation}
\label{eqn:lip_stmnt}
\Bignorm{\pderiv{\pi}{}{x}(x^{i+1}) - \pderiv{\pi^{i+1}}{}{x}(x^{i+1}) } %
\le
T \Bignorm{\pderiv{\pi}{}{x}(x^{i+1}) - \pderiv{\bar{\pi}}{}{x}(x^{i+1}) }.
\end{equation}
whenever the distance from $x^{i+1}$ to the previous patch point is less than the maximum consecutive patch point distance
. 
The maximum consecutive patch point distance $h_T$ must be less than the consecutive patch point distance that we derive in Lemma \ref{lem:dRdT_term1}, and  small enough to ensure $\norm{\tdpidx(x^{i+1}) - \tdpibardx(x^{i+1})} \le M_{\Delta}$, where
\begin{align*}
M_{\Delta} &< \inf_{\compdomwoalb}\Bignorm{\pderiv{\pi}{}{x}(x)}
\\
M_{\Delta} &< \frac{\inf_{\compdomwoalb} \abs{\pderiv{\pi}{}{x}(x) f(x)}}{\sup_{\compdomwoalb} \norm{f(x)}}
\\
M_{\Delta} &< \frac{\inf_{\compdomwoalb} \abs{\pderiv{\pi}{}{x}(x) g(x)}}{\sup_{\compdomwoalb} \norm{g(x)}}
\\
M_{\Delta}
&<
\frac{1}{2}
\sup_{\closedcompdomwoalb}
\Bigpars{
\Bignorm{\dpidx(x)}\bigpars{\norm{f(x)}^2 + 2 \frac{q(x)}{r(x)}\norm{g(x)}^2} 
\\
&\phantom{<}+ \norm{f(x)}^2 + \norm{g(x)}^2
}
\inf_{\closedcompdomwoalb} \Bigpars{\dpidx(x)\dot{x}}^2
\end{align*}
The assumptions in \S \ref{sec:drvtn:opt_cntrl_ass} guarantee that each right hand side of the previous four inequalities is strictly positive.
\end{theorem}

\begin{corollary}
\label{cor:phi1_lte_bound}
There exists a maximum consecutive patch point distance $h_T$ and local truncation constant $T$ such that for any two consecutive patch points $x^i$ and $x^{i+1}$ in the computational domain $\compdom$ within a distance of $h_T$ of each other,
\begin{equation}
\Bignorm{\pderiv{\pi}{}{x}(x^{i+1}) - \pderiv{\pi^{i+1}}{}{x}(x^{i+1}) } %
\le
T \norm{x^{i+1} - x^i}^{d+1}
.
\end{equation}
where $\tpderiv{\pi^{i+1}}{}{x}(x^{i+1})$ denotes the computed gradient of the optimal cost at $x^{i+1}$, computed from the coefficients of the Taylor polynomial centered at the previous patch point.
Equivalently
\begin{equation*}
\norm{\tau_1^i} = 
\Bignorm{C_1^{i+1} - \phi_1(C^i) }
\le
T \norm{x^{i+1} - x^i}^{d+1}
.
\end{equation*}
\end{corollary}

It is helpful to think of the quantities $f_n$ and $g_n$ from \eqref{eqn:drvtn_fn_gn_defn}, and the solution to the quadratic equation $z$ from \eqref{eqn:drvtn_sclr_hjb_eqn} as functions of a nonzero vector $w$, so we define
\begin{align}
\label{eqn:dpi_dx_ptchy_defns}
\begin{split}
f_n(w) &\equiv \frac{w}{\norm{w}} \cdot f(x^{i+1}) \\
g_n(w) &\equiv \frac{w}{\norm{w}} \cdot g(x^{i+1}) \\
z(w) &\equiv
\begin{cases}
-\frac{q(x^{i+1})}{f_n(w)} & g_n(w) = 0 \\
\frac{f_n(w) + \sqrt{f_n(w)^2 + 2\frac{q(x^{i+1})}{r(x^{i+1})} g_n(w)^2} }{g_n(w)^2} r(x^{i+1}) & g_n(w) \ne 0
\end{cases}
\end{split}
\end{align}

As a consequence of the assumption that $\tlyupderiv{(x)} < 0$, $f_n(w)$ and $g_n(w)$ cannot both be zero away from the origin if $w$ is sufficiently close to $\tpderiv{\pi}{}{x}(x^{i+1})$.
Therefore, $z(w)$ is well defined.
The difference between the true and computed first order partial derivatives of $\pi$ can be rewritten as
\begin{equation*} %
\pderiv{\pi}{}{x}(x^{i+1}) - \pderiv{\pi^{i+1}}{}{x}(x^{i+1}) 
= 
\pderiv{\pi}{}{x}(x^{i+1})
- \frac{z(\pderiv{\bar{\pi}}{}{x}(x^{i+1}))}{\norm{\pderiv{\bar{\pi}}{}{x}(x^{i+1})}} \pderiv{\bar{\pi}}{}{x}(x^{i+1})
\end{equation*}
and the main inequality \eqref{eqn:lip_stmnt} from the statement of Theorem \ref{thm:phi1_lte_bound} will follow from the fact that
\begin{equation} \label{eqn:z_over_w_bound} %
\frac{z(\pderiv{\bar{\pi}}{}{x}(x^{i+1}))}{\norm{\pderiv{\bar{\pi}}{}{x}(x^{i+1})}} = 
1 + \mathcal{O} \Bigpars{ \Bignorm{ \pderiv{\pi}{}{x}(x^{i+1}) - \pderiv{\bar{\pi}}{}{x}(x^{i+1}) } }
\end{equation}
as long as the distance from $x^{i+1}$ to the previous patch point is less than the maximum consecutive patch point distance.

To prove \eqref{eqn:z_over_w_bound}, we define the helper function
\begin{equation} \label{eqn:R_defn}
R(t)
\equiv
\frac{
z \Bigpars{\pderiv{\pi}{}{x}(x^{i+1}) + t \Bigpars{ \pderiv{\bar{\pi}}{}{x}(x^{i+1}) - \pderiv{\pi}{}{x}(x^{i+1})}}
}
{
\Bignorm{\pderiv{\pi}{}{x}(x^{i+1}) + t \Bigpars{ \pderiv{\bar{\pi}}{}{x}(x^{i+1}) - \pderiv{\pi}{}{x}(x^{i+1})}}
}
\end{equation}
so that we can appeal to the mean value theorem, and rewrite the left-hand side of \eqref{eqn:z_over_w_bound} as $1 + \tfrac{d R}{d t}(\xi)$ for some $0 < \xi < 1$.
If $R$ is differentiable, then its derivative is bounded by
\begin{multline} \label{eqn:dRdT_bnd}
\Bigabs{\frac{dR}{dt}(t)} 
\le
\\
\Biggpars{ 
\Bignorm{\dzdw \Bigpars{\dpidx(x) + t\Bigpars{\dpibardx(x) - \dpidx(x)}} }
+ 
\frac{z\Bigpars{\dpidx(x) + t\Bigpars{\dpibardx(x) - \dpidx(x)}}}{\Bignorm{ \dpidx(x) + t\Bigpars{\dpibardx(x) - \dpidx(x)} }}
}
\\
\times
\frac{\Bignorm{\dpibardx(x) - \dpidx(x)}}{\Bignorm{\dpidx(x) + t\Bigpars{\dpibardx(x) - \dpidx(x)}}}
\end{multline}
Our assumption that the optimal cost is a strict Lyupanov function is the key fact that guarantees that there exists a maximum consecutive patch point distance such that $\tfrac{d R}{d t}(t)$ is well defined for $0 \le t \le 1$ and \eqref{eqn:z_over_w_bound} holds.
This is the main content of Lemmas \ref{lem:dRdT_term1}, \ref{lem:z_over_dpidx_bounded}, \ref{lem:dzdw_exists}, and \ref{lem:dzdw_bounded}.

The next lemma guarantees that the denominators in the bound on $\tfrac{d R}{d t}(t)$ \eqref{eqn:dRdT_bnd} are bounded away from zero.

\begin{lemma} \label{lem:dRdT_term1} %
Let $\bar{\pi}$ denote the Taylor polynomial centered at the previous patch point.
If $\pi$ has continuous first order partial derivatives on $\compdom$ and is a strict Lyapunov function on $\compdom$, then there exists a maximum consecutive patch point distance $h'$ and constant $T'$, such that
\begin{equation*}
\frac{\Bignorm{\dpidx(x) - \dpibardx(x)}}
{\Bignorm{\dpidx(x) + t \Bigpars{\dpidx(x) - \dpibardx(x)}}}
\le 
T' \Bignorm{\pderiv{\pi}{}{x}(x) - \pderiv{\bar{\pi}}{}{x}(x)}
.
\end{equation*}
whenever $0 \le t \le 1$ and the distance from $x$ to the previous patch point is less than $h'$.

\begin{proof}
We split the proof into two cases, depending on whether the previous patch point is the origin or not.
In either case, we will use the fact that if $0 \le \norm{\tdpidx(x) - \tdpibardx(x)} < \norm{\tdpidx(x)}$, then
\begin{equation*}
\frac{\Bignorm{\dpidx(x) - \dpibardx(x)}}
{\Bignorm{\dpidx(x) + t \Bigpars{\dpidx(x) - \dpibardx(x)}}}
\le
\frac{\Bignorm{\dpidx(x) - \dpibardx(x)}}{\norm{\dpidx(x)} - \norm{\dpidx(x) - \dpibardx(x)}} 
\end{equation*}
for all $0 \le t \le 1$.

If the previous patch point is the origin, then we need establish the existence of a deleted neighborhood of radius $h'$ centered at the origin such that
\begin{equation*}
 \Biggabs{ \Bignorm{\dpidx(x)} - \Bignorm{\dpidx(x) - \pderiv{\bar{\pi}}{}{x}(x)} } 
> 0
\end{equation*}
for all $x$ in the neighborhood.
$\dpidx$ is nonzero away from the origin under the assumption that $\pi$ is a strict Lyapunov function on $\compdom$,
so we need to establish that the  difference between the gradients of $\pi$ and $\bar{\pi}$ goes to zero faster than the gradient of $\pi$ as $x$ goes to zero.
The assumption (\S \ref{eqn:drvtn_dynamics_lagrangian_form}, Assumption \ref{ass:nice_optml_cntrl_prob}) that $Q \succ 0$ guarantees that the leading order linear term of $\tpderiv{\pi}{}{x}(x)$ is nonzero for all nonzero $x$.
Thus $\norm{\tdpidx(x)-\tdpibardx(x)} / \norm{\tdpidx(x)} \rightarrow 0$ as  $x \rightarrow 0$,
so there is a radius $h'>0$ such that 
\begin{equation*}
\Bignorm{\dpidx(x)-\dpibardx(x)} <  \Bignorm{ \dpidx(x) }
\text{ for all $x$ in }
0 < \norm{x} \le h'.
\end{equation*}
Since $\tdpidx$ and $\tdpibardx(x)$ are continuous and $r \le \norm{x} \le h'$ is compact, then for any $r$ in $0 < r < h'$,
\begin{equation*} %
\sup_{r \le \norm{x} \le h'}
\Biggabs{ \Bignorm{\dpidx(x)} - \Bignorm{\dpidx(x) - \pderiv{\bar{\pi}}{}{x}(x)} }^{-1}
<
\infty
.
\end{equation*}
Thus, the lemma holds if the previous patch point is the origin, and the distance from $x^{i+1}$ to the origin is no greater than $h'$.

For this lemma to apply to the patchy method, the Al'brekht patch $\mathcal{A}$ must lie in the neighborhood of the origin of radius $h'$,
and the neighborhood of radius $r$ must lie completely inside the Al'brekht patch.

We now consider the case where the previous patch point $x^i$ is not the origin, so its distance to the origin is at least $r$.
$\norm{\tdpidx(x)}$ is bounded away from zero on $\closedcompdomwoalb$, and $\lim_{x \rightarrow x^i} \norm{\tdpidx(x) - \tdpibardx(x)} = 0$, so there exists an $h'$ such that
\begin{equation*}
\sup_{\substack{\closedcompdomwoalb \\ \norm{x-x^i} \le h'}} 
      \Biggabs{ \Bignorm{\dpidx(x)} - \Bignorm{\dpidx(x) - \pderiv{\bar{\pi}}{}{x}(x)} }^{-1}
< \infty
\end{equation*}
and we have established the validity of the lemma when the previous patch point is not the origin.
\end{proof}
\end{lemma}

Our next task is to establish the existence of $M_{\Delta}$ so that $z(\tdpidx(x) + \Delta)/\norm{\tdpidx(x) + \Delta}$ is uniformly bounded for all $\norm{\Delta} \le M_{\Delta}$ and all $x$ in $\compdom$.
This is the content of lemma \ref{lem:z_over_dpidx_bounded}.

\begin{lemma} \label{lem:z_over_dpidx_bounded}
\begin{equation*}
\sup_{\substack{\closedcompdomwoalb \\ \norm{\Delta} \le M_{\Delta}}}
     \frac{\abs{z(\pderiv{\pi}{}{x}(x) + \Delta)}}{\norm{\pderiv{\pi}{}{x}(x) + \Delta}}
< \infty
\end{equation*}
for any nonzero $M_{\Delta}$ satisfying
\begin{equation} \label{eqn:ass_small_dpidxg}
\begin{aligned}
M_{\Delta} & < \inf_{\compdomwoalb}\Bignorm{\pderiv{\pi}{}{x}(x)}
\\
M_{\Delta} &< \frac{\inf_{\mathcal{S}} \abs{\pderiv{\pi}{}{x}(x) f(x)}}{\sup_{\mathcal{S}} \norm{f(x)}}
\\
M_{\Delta} &< \frac{\inf_{\mathcal{S}^C} \abs{\pderiv{\pi}{}{x}(x) g(x)}}{\sup_{\mathcal{S}^C} \norm{g(x)}}
\end{aligned}
\end{equation}
where
\begin{equation*}
\mathcal{S} \equiv
\Biggl\{
x \in \closedcompdomwoalb \; \Bigg\vert \; \Bigabs{\pderiv{\pi}{}{x}(x) g(x)}
\le 
\sqrt{\frac{1}{2} \inf_{\closedcompdomwoalb} r(x) \inf_{\closedcompdomwoalb} \Bigabs{\lyupderiv{(x)}}}
\Biggr\}.
\end{equation*}

The assumption that $\pi$ is a strict Lyupanov function and the definition of $\mathcal{S}$ implies that both $\sup_{\mathcal{S}} \norm{f(x)}$ and $\sup_{\mathcal{S}^C} \norm{g(x)}$ are strictly greater than zero, so $M_{\Delta}$ is well defined.

\begin{proof}
Let $w = \tpderiv{\pi}{}{x}(x) + \Delta$, if $w \cdot g(x) \ne 0$, then by the formula for $z(w)$ in \eqref{eqn:dpi_dx_ptchy_defns},
\begin{equation} \label{eqn:z_over_w}
\frac{z(w)}{\norm{w}}
=
\frac{w \cdot f(x) + \sqrt{(w \cdot f(x))^2 + 2\frac{q(x)}{r(x)} (w \cdot g(x))^2 }}
{ (w \cdot g(x))^2} r(x)
\end{equation}
Since $(\tpderiv{\pi}{}{x}(x) + \Delta) g(x)$ appears in the denominator of \eqref{eqn:z_over_w}, we split the proof into the two cases where $x$ is in  $\mathcal{S}$ or its complement.

If $x$ is in $\mathcal{S}$, then by the definition of $\mathcal{S}$ and the assumption that $\pi$ is a strict Lyupanov function
\begin{equation*}
\pderiv{\pi}{}{x}(x) f(x)
\le
-\frac{1}{2} \inf_{\closedcompdomwoalb} \Bigabs{\lyupderiv{(x)}}
<
0                  
.
\end{equation*}
Therefore, if $M_{\Delta}$ satisfies the second condition in \eqref{eqn:ass_small_dpidxg}, then there is a strictly negative uniform upper bound on $(\tpderiv{\pi}{}{x}(x) + \Delta)f(x)$ for all $x$ in $\mathcal{S}$ and all $\norm{\Delta} \le M_{\Delta}$.
It follows that for all $x$ in $\mathcal{S}$, the square root term in \eqref{eqn:z_over_w} has the Taylor expansion with remainder
\begin{equation*}
\sqrt{(w \cdot f(x))^2 + 2\frac{q(x)}{r(x)} (w \cdot g(x))^2 }
=
-(w \cdot f(x))\Biggpars{1 + \frac{q(x)}{r(x)} \Bigpars{\frac{w \cdot g(x)}{w \cdot f(x)}}^2 \frac{1}{\sqrt{1 + \xi}}}
\end{equation*}
for some $\xi \ge 0$.
So, for all $x$ in $\mathcal{S}$ and all $\norm{ \Delta } \le M_{\Delta}$
\begin{equation} \label{eqn:z_over_dpidx_bounded_S_bnd}
\frac{\abs{z(\pderiv{\pi}{}{x}(x) + \Delta)}}{\norm{\pderiv{\pi}{}{x}(x) + \Delta)}}
=
\frac{q(x)}{(\pderiv{\pi}{}{x}(x) + \Delta) f(x)} \frac{1}{\sqrt{1 + \xi}}
\le
\frac{q(x)}{(\pderiv{\pi}{}{x}(x) + \Delta) f(x)}
\end{equation}
which also holds in the case where $(\pderiv{\pi}{}{x}(x) + \Delta) g(x) = 0$, and consequently $z(\pderiv{\pi}{}{x}(x) + \Delta)$ solves a linear equation \eqref{eqn:drvtn_sclr_hjb_eqn}.
The supremum of the right-hand side of \eqref{eqn:z_over_dpidx_bounded_S_bnd} over  $\mathcal{S}$ and all $\Delta \le M_{\Delta}$ is finite.
The first two conditions on $M_{\Delta}$ \eqref{eqn:ass_small_dpidxg} guarantee that the infimum of $\norm{(\pderiv{\pi}{}{x}(x) + \Delta) f(x)}$ is strictly greater than zero over $\mathcal{S}$ and all $\Delta$ bounded above by $M_{\Delta}$ in norm.
The function $q$ is continuous on the compact set $\closedcompdomwoalb \supseteq \mathcal{S}$, therefore its supremum over $\mathcal{S}$ is bounded.

We now turn to the case where $x$ is in $\mathcal{S}^C$.
In this case, the only way $z(w)/\norm{w}$ \eqref{eqn:z_over_w} can become unbounded is if $(\tdpidx(x) + \Delta) g(x)$ can become arbitrarily small, but this cannot happen. %
By the definition of $\mathcal{S}$, the infimum of $\abs{\tpderiv{\pi}{}{x}(x) g(x)}$ over $\mathcal{S}^C$ is strictly positive, and the infimum of $\abs{(\tpderiv{\pi}{}{x}(x) + \Delta) g(x)}$ over $\mathcal{S}^C$ and all $\Delta$ bounded above in norm by $M_{\Delta}$ is also strictly positive.

\end{proof}
\end{lemma}

Our final task in establishing Theorem \ref{thm:phi1_lte_bound} is to show that there exists an $M_{\Delta}$ such that the gradient of $z(w)$ exists and is uniformly bounded over $\tdpidx(\closedcompdomwoalb) + \Delta$ for $\norm{\Delta} \le M_{\Delta}$.
This is the content of Lemmas \ref{lem:dzdw_exists} and \ref{lem:dzdw_bounded} .
\begin{lemma} \label{lem:dzdw_exists}
Under the assumptions on the optimal control problem of \S \ref{sec:drvtn:opt_cntrl_ass},
the function $z$ that solves \eqref{eqn:drvtn_sclr_hjb_eqn} is continuously differentiable in a nonempty neighborhood of every point in $\tdpidx(\closedcompdomwoalb)$.
\begin{proof}
By the implicit function theorem \cite[ch.~9]{rudin:real_anly}, a solution to $a x^2 + b x + c = 0$ is a continuously differentiable function of its coefficients in a neighborhood of $a_0$, $b_0$, and $c_0$ if $x_0$ solves the quadratic equation with coefficients $a_0$, $b_0$, and $c_0$, and $2 a_0 x_0 + b_0 \ne 0$.
To conclude that $z(w)$ is continuously differentiable in a neighborhood of each point in $\tdpidx(\closedcompdomwoalb)$, we must verify that $-\tfrac{1}{r(x)}g_n(\tpderiv{\pi}{}{x}(x))^2 z(\tpderiv{\pi}{}{x}(x)) + f_n(\tpderiv{\pi}{}{x}(x))$ is nonzero for every $x$ in $\tdpidx(\closedcompdomwoalb)$.

As a function of $x$, $\tdpidx \dot{x}$ is continuous under the assumptions that $f$, $g$, $r$, and $\tdpidx$ are continuous on $\closedcompdomwoalb$.
Furthermore, $\tdpidx \dot{x}$ is nonzero on $\closedcompdomwoalb$ by the assumption that $\pi$ is a strict Lyapunov function, so there is a nonzero uniform lower bound on both $\abs{\tdpidx(x) \dot{x}}$ and $\norm{\tdpidx}$ on $\closedcompdomwoalb$.
It then follows from the identity
\begin{equation*}
-\frac{1}{r(x)}g_n \Bigpars{ \pderiv{\pi}{}{x}(x)}^2 z\Bigpars{\pderiv{\pi}{}{x}(x)}
+ f_n\Bigpars{\pderiv{\pi}{}{x}(x)}
=
\frac{\pderiv{\pi}{}{x}(x) \dot{x}}{\Bignorm{\pderiv{\pi}{}{x}(x)}}
,
\end{equation*}
that $-\tfrac{1}{r(x)}g_n(\tpderiv{\pi}{}{x}(x))^2 z(\tpderiv{\pi}{}{x}(x)) + f_n(\tpderiv{\pi}{}{x}(x))$ is nonzero for every $x$ in $\tdpidx(\closedcompdomwoalb)$.
\end{proof}
\end{lemma}

We now verify that there exists an $M_{\Delta} > 0$ such that  $\norm{\tpderiv{z}{}{w}(\tdpidx(\closedcompdomwoalb) + \Delta)}$ is uniformly bounded for all $\norm{\Delta} \le M_{\Delta}$.

\begin{lemma}
\label{lem:dzdw_bounded}
If $M_{\Delta}$ satisfies the conditions of Lemma \ref{lem:z_over_dpidx_bounded} in \eqref{eqn:ass_small_dpidxg} in addition to
\begin{multline*}
M_{\Delta}
<
\frac{1}{2}
\sup_{\closedcompdomwoalb}
\Bigpars{
\Bignorm{\dpidx(x)}\bigpars{\norm{f(x)}^2 + 2 \frac{q(x)}{r(x)}\norm{g(x)}^2} + \norm{f(x)}^2 + \norm{g(x)}^2
}
\\
\times
\inf_{\closedcompdomwoalb} \Bigpars{\dpidx(x)\dot{x}}^2
\end{multline*}
then
\begin{equation*}
\sup_{\substack{\overline{X_c \setminus \mathcal{A}}\\ \norm{\Delta} \le M_{\Delta}}}
\Biggnorm{
\pderiv{z}{}{w}\Biggpars{\pderiv{\pi}{}{x}(x) + \Delta}
}
< \infty.
\end{equation*}

\begin{proof}
Let $w= \tdpidx(x) + \Delta$.
By Lemma \ref{lem:dzdw_exists}, $\tpderiv{z}{}{w}$ exists in a neighborhood of $\tdpidx(x)$ for all $x$ in $\closedcompdomwoalb$, and an implicit formula for it is
\begin{multline}
\label{eqn:dzdw_formula}
\pderiv{z}{}{w}(w) = 
\frac{1}{f_n(w) - \frac{1}{r(x)}g_n(w)^2 z(w)} 
\\
\times \Biggbracks{
\Bigpars{\frac{z(w)}{\norm{w}_2}}^2
g_n(w) \bigpars{ \norm{w}_2 g(x) - g_n(w) w }
\\
+
\frac{z(w)}{\norm{w}_2}
\Bigpars{
g_n(w) \frac{w}{\norm{w}_2} - f(x) 
}
}
\end{multline}
Both $f$ and $g$ are uniformly bounded on $\closedcompdomwoalb$.
Since $\Delta \le M_{\Delta} < \inf_{\closedcompdomwoalb} \norm{\tdpidx(x)}$, $w$ is nonzero, and the terms $f_n(w)$ and $g_n(w)$ are both uniformly bounded on $\closedcompdomwoalb$ for all $\Delta \le M_{\Delta}$ by the supremums of $f$ and $g$ on $\closedcompdomwoalb$.
By Lemma \ref{lem:z_over_dpidx_bounded}, the term $z(w)/\norm{w}$ is bounded.

The only remaining term that can potentially become unbounded is $(f_n(w) - \tfrac{1}{r(x)}g_n(w)^2 z(w))^{-1}$.
This quantity can be rewritten as
\begin{align} 
\Bigpars{f_n(w) -  \frac{g_n(w)^2}{r(x)} z(w)}^2 
&= f_n(w)^2 + 2 \frac{q(x)}{r(x)}g_n(w)^2 
\label{eqn:discr_rel1}\\
&= \frac{1}{\norm{\dpidx(x) + \Delta}^2} \Bigpars{ \Bigpars{\dpidx(x)\dot{x}}^2 + E(x, \Delta) }
\label{eqn:discr_rel2}
\end{align}
Equation \eqref{eqn:discr_rel1} follows from the fact that both sides of the equation are formulas for the discriminant of the quadratic equation \eqref{eqn:drvtn_sclr_hjb_eqn}.
If we solve the HJB equation \eqref{eqn:HJB_PDE1} for $q(x)$ in terms of $\tpderiv{\pi}{}{x}(x)$ and the other problem data, and then substitute this formula into \eqref{eqn:discr_rel1}, then we arrive at \eqref{eqn:discr_rel2}.
We defined $M_{\Delta}$ in the lemma statement so that $E(x, \Delta) = \mathcal{O}(\norm{\Delta})$ and the infimum of \eqref{eqn:discr_rel2} is strictly greater than zero over $\closedcompdomwoalb$ and all $\norm{\Delta} \le M_{\Delta}$.
\end{proof}
\end{lemma}

\section{$\phi_1$ satisfies the Lipschitz condition} \label{sec:phi-1_lip}
\begin{theorem}
\label{thm:phi_1_is_lipschitz}
Under the assumptions on the optimal control problem of \S \ref{sec:drvtn:opt_cntrl_ass}, there exists a maximum step size $h_L>0$, maximum multiplier $\Mmax>0$, and Lipschitz constant $L_1$ such that for all $x^i$ and $x^{i+1}$ in $\closedcompdomwoalb$ within a distance of $h_L$ of each other, and for all $M \le \Mmax$,
\begin{gather*}
\norm{\phi_1(C^i) - \phi_1(\hat{C}^i)} \le L_1 M \norm{x^{i+1} - x^i}^{d+1}
\intertext{whenever}
\norm{C_j^i -\hat{C}_j^i} \le M \norm{x^{i+1} - x^i}^{d+2-j}
\text{ for }
0 \le j \le d+1
\end{gather*}
where $C^i$ is the vector of exact partial derivatives up to order $d+1$ of the optimal cost at $x^i$.
The theorem holds for any $\Mmax$, $h_L$ pair that satisfy maximum consecutive patch point distance multiplier condition of Definition \ref{defn:max_step_mult}.
\end{theorem}

\begin{definition}
\label{defn:max_step_mult}
The maximum consecutive patch point distance $h_L$ and the maximum multiplier $\Mmax$ are said to satisfy the maximum consecutive patch point distance multiplier condition if
\begin{align*}
h_L > 0
& &
\Mmax > 0
& &
h_L \le h_T
& &
(\sqrt{n} \Linh \Mmax + T) h_L^{d+1} \le M_{\Delta}
\end{align*}
where $\Linh$ is defined in Lemma \ref{lem:poly_lip}, $h_T$ and $T$ are the maximum consecutive step size and local truncation constant from Corollary \ref{cor:phi1_lte_bound} and $M_{\Delta}$ satisfies the inequalities of Theorem \ref{thm:phi1_lte_bound}.
\end{definition}

\begin{proof}[Proof of Theorem \ref{thm:phi_1_is_lipschitz}]
If $w=\tpderiv{P}{}{x}(C, x^{i+1}-x^i)$ is nonzero and $z$ is defined at $w$, then $\phi_1$ can be expressed as
\begin{equation*}
\phi_1(C)
=
z\Bigpars{\pderiv{P}{}{x}(C,x^{i+1}-x^i)}
\frac{\pderiv{P}{}{x}(C,x^{i+1}-x^i)}{\norm{\pderiv{P}{}{x}(C,x^{i+1}-x^i)}}
\end{equation*}
where $P$ is the polynomial whose coefficients are held in $C$, and is defined in Lemma \ref{lem:poly_lip}.
We will show that there exists $L'$ such that, if the conditions of the theorem are met, then
\begin{gather}
\label{eqn:phi1_lip_proof_ineq1}
\Biggnorm{
\frac{\pderiv{P}{}{x}(C^i,x^{i+1}-x^i)}{\norm{\pderiv{P}{}{x}(C^i,x^{i+1}-x^i)}}
-
\frac{\pderiv{P}{}{x}(\hat{C}^i,x^{i+1}-x^i)}{\norm{\pderiv{P}{}{x}(\hat{C}^i,x^{i+1}-x^i)}}
}
\le
L_1' M \norm{x^{i+1}-x^i}^{d+1}
\intertext{and}
\label{eqn:phi1_lip_proof_ineq2}
\Bigabs{
z \Bigpars{\tpderiv{P}{}{x}(C^i, x^{i+1}-x^i)}
- 
z \Bigpars{\tpderiv{P}{}{x}(\hat{C}^i, x^{i+1}-x^i)}}
\le
L_1' M \norm{x^{i+1}-x^i}^{d+1}
\end{gather}
which implies the conclusion of the theorem.

We start with the inequality in \eqref{eqn:phi1_lip_proof_ineq1}.
$P$ and all its partial derivatives are linear in its first argument.
If
\begin{equation*}
\Bignorm{
\pderiv{P}{}{x}(C^i, x^{i+1}-x^i) - \pderiv{P}{}{x}(\hat{C}^i, x^{i+1}-x^i)
}
<
\Bignorm{
\pderiv{P}{}{x}(C^i, x^{i+1}-x^i)
}
\end{equation*}
then
\begin{multline*}
\Biggnorm{
\frac{\pderiv{P}{}{x}(C^i,x^{i+1}-x^i)}{\norm{\pderiv{P}{}{x}(C^i,x^{i+1}-x^i)}}
-
\frac{\pderiv{P}{}{x}(\hat{C}^i,x^{i+1}-x^i)}{\norm{\pderiv{P}{}{x}(\hat{C}^i,x^{i+1}-x^i)}}
}
\\
\begin{aligned}
&\le
2 \frac{\Bignorm{\pderiv{P}{}{x}(C^i - \hat{C}^i, x^{i+1}-x^i)}}
{\norm{\pderiv{P}{}{x}(C^i, x^{i+1}-x^i)} - 
\norm{\pderiv{P}{}{x}(C^i-\hat{C}^i, x^{i+1}-x^i) }}
\\
&\le
2 \frac{\sqrt{n} \Linh M \norm{x^{i+1}-x^i}^{d+1}}
{\norm{\pderiv{P}{}{x}(C^i, x^{i+1}-x^i)} - 
\norm{\pderiv{P}{}{x}(C^i-\hat{C}^i, x^{i+1}-x^i)}}
\end{aligned}
\end{multline*}
where the final inequality follows from Lemma \ref{lem:poly_lip}.
The inequality in \eqref{eqn:phi1_lip_proof_ineq1} will follow after we establish that for some $\epsilon$,
\begin{equation}
\label{eqn:phi1_lip_proof_eps_bound}
0 < \epsilon 
\le
\Biggabs{
\Bignorm{\pderiv{P}{}{x}(C^i, x^{i+1}-x^i)} - \Bignorm{\pderiv{P}{}{x}(C^i-\hat{C}^i, x^{i+1}-x^i)}
}
\end{equation}
\eqref{eqn:phi1_lip_proof_eps_bound} also guarantees that $\tpderiv{P}{}{x}(\hat{C}^i, x^{i+1}-x^i)$ is nonzero.
$C^i$ holds the exact partial derivatives of the optimal cost at $x^i$, so $P(C^i, x-x^i) = \bar{\pi}(x)$ where $\bar{\pi}$ is the Taylor polynomial centered at $x^i$.
By Corollary \ref{cor:phi1_lte_bound},
\begin{align*}
\Biggabs{
\Bignorm{\pderiv{\pi}{}{x}(x^{i+1})} - \Bignorm{\pderiv{P}{}{x}(C^i,x^{i+1}-x^i)}
}
&\le
\Bignorm{\pderiv{\pi}{}{x}(x^{i+1}) - \pderiv{P}{}{x}(C^i,x^{i+1}-x^i)} \\
&\le
T \norm{ x^{i+1} - x^i }^{d+1}
.
\end{align*}
By Lemma \ref{lem:poly_lip}
\begin{equation*}
\Bignorm{\pderiv{P}{}{x}(C^i,x^{i+1}-x^i) - \pderiv{P}{}{x}(\hat{C}^i,x^{i+1}-x^i)} 
\le
\sqrt{n} \Linh \Mmax \norm{x^{i+1}-x^i}
\end{equation*}
Inequality \ref{eqn:phi1_lip_proof_eps_bound} follows from the final strict inequality in
\begin{multline*}
\Biggabs{
\Bignorm{\pderiv{P}{}{x}(C^i, x^{i+1}-x^i)} - \Bignorm{\pderiv{P}{}{x}(C^i-\hat{C}^i, x^{i+1}-x^i)}
}
\\
\begin{aligned}
&=
\Biggabs{
\Bignorm{\pderiv{\pi}{}{x}(x^{i+1})}
+
\Bignorm{\pderiv{P}{}{x}(C^i, x^{i+1}-x^i)}
\\
&\phantom{= \vert}
- \Bignorm{\pderiv{\pi}{}{x}(x^{i+1})}
- \Bignorm{\pderiv{P}{}{x}(C^i-\hat{C}^i, x^{i+1}-x^i)}
}
\\
&\ge
\inf_{\closedcompdomwoalb} \Bignorm{\pderiv{\pi}{}{x}(x)}
-
(\sqrt{n} \Linh \Mmax + T) h_L^{d+1}
\\
&>
0
\end{aligned}
\end{multline*}
The maximum patch point multiplier condition (Definition \ref{defn:max_step_mult}) was defined so the last inequality is strict.

We now prove that the inequality \eqref{eqn:phi1_lip_proof_ineq2} holds by applying the mean value theorem to $z(w)$.
If
\begin{gather}
\label{eqn:phi1_lip_proof_ineq3}
\Bignorm{
\pderiv{P}{}{x}(C^i, x^{i+1}-x^i)
-
\pderiv{\pi}{}{x}(x^{i+1})
}
\le
M_{\Delta}
\intertext{and}
\label{eqn:phi1_lip_proof_ineq4}
\Bignorm{
\pderiv{P}{}{x}(\hat{C}^i, x^{i+1}-x^i)
-
\pderiv{\pi}{}{x}(x^{i+1})
}
\le
M_{\Delta}
\end{gather}
then for all $\tau$ in $0 \le \tau \le 1$,
\begin{equation*}
\Bignorm{
(1 - \tau) \pderiv{P}{}{x}(C^i, x^{i+1}-x^i) + \tau \pderiv{P}{}{x}(\hat{C}^i, x^{i+1}-x^i)
- 
\pderiv{\pi}{}{x}(x^{i+1})
}
\le
M_{\Delta}
\end{equation*}
meaning that every point on the line segment connecting $\tpderiv{P}{}{x}(C^i, x^{i+1}-x^i)$ and $\tpderiv{P}{}{x}(\hat{C}^i, x^{i+1}-x^i)$ is at most a distance of $M_{\Delta}$ from $\tpderiv{\pi}{}{x}(x^{i+1})$.
Therefore, $z$ is well defined at the endpoints, and by Lemma \ref{lem:dzdw_bounded} there is a uniform bound on the magnitude of $\tpderiv{z}{}{w}$ over all such lines.
Inequality \ref{eqn:phi1_lip_proof_ineq2} will follow once we show that the conditions of the theorem guarantee that inequalities \ref{eqn:phi1_lip_proof_ineq3} and \ref{eqn:phi1_lip_proof_ineq4} hold.
By assumption, $h_L \le h_T$, so by Corollary \ref{cor:phi1_lte_bound} and Definition \ref{defn:max_step_mult}
\begin{equation*}
\Bignorm{\pderiv{\pi}{}{x}(x^{i+1})
-
\pderiv{P}{}{x}(C^i, x^{i+1}-x^i)
}
\le
T h_L^{d+1}
\le
M_{\Delta}
\end{equation*}
and
\begin{multline*}
\Bignorm{
\pderiv{\pi}{}{x}(x^{i+1})
-
\pderiv{P}{}{x}(\hat{C}^i, x^{i+1}-x^i)
}
\\
\begin{aligned}
&\le
\Bignorm{
\pderiv{P}{}{x}(C^i-\hat{C}^i, x^{i+1}-x^i)
}
+
\Bignorm{
\pderiv{\pi}{}{x}(x^{i+1})
-
\pderiv{P}{}{x}(C^i, x^{i+1}-x^i)
}
\\
&\le
(\sqrt{n} \Linh \Mmax + T)h_L^{d+1}
\\
&\le
M_{\Delta}
.
\end{aligned}
\end{multline*}

\end{proof}

\section{Local truncation error for second and higher order partial derivatives }
\begin{theorem} \label{thm:opt_cost_local_trunc_err_higer_order}
Let $C^i$ and $C^{i+1}$ denote vectors of exact partial derivatives of the optimal cost at the consecutive patch points $x^i$ and $x^{i+1}$ in the computational domain $\compdom$.
If the optimal control problem satisfies the assumptions of \S \ref{sec:drvtn:opt_cntrl_ass} , then there exists local truncation constant $T_2 < \infty$ and maximum consecutive patch point distance $h > 0$ such that for all $2 \le j \le d+1$,
\begin{equation*}
\norm{C_j^{i+1} - \phi_j(C^i)} \le T \norm{x^{i+1} - x^i}_2^{d+2-j}
\text{ whenever } \norm{x^{i+1}-x^i} \le h
\end{equation*}
\end{theorem}
We delay the proof of Theorem \ref{thm:opt_cost_local_trunc_err_higer_order} until after we have developed an overview of its main ideas.

Let $\dot{x}^{i+1}$ denote the computed optimal direction at the current patch point $x=x^{i+1}$.
There are two changes of variables associated with the computed and exact solutions at the current patch point.
They are
\begin{align}
\label{eqn:chngs_of_vars}
x = x^{i+1} + \hat{V} \hat{\xi} & & \hat{V}^1 = \frac{\dot{x}^{i+1}}{\norm{\dot{x}^{i+1}}}
&& 
x = x^{i+1} + V \xi & & V^1 = \frac{\dot{x}}{\norm{\dot{x}}}
.
\end{align}
We denote the computed and exact optimal cost and control under the changes of variables as
\begin{align}
\label{eqn:exct_cmptd_optml_cost_cntrl}
\begin{aligned}
\hatpi^{i+1}(\hat{\xi}) 
&\equiv 
\pi^{i+1}(x^{i+1} + \hat{V} \hat{\xi})
\\
\hatkappa^{i+1}(\hat{\xi}) 
&\equiv 
\kappa^{i+1}(x^{i+1} + \hat{V} \hat{\xi})
\end{aligned}
& &
\begin{aligned}
\tildepi(\xi) 
&\equiv 
\pi(x^{i+1} + V \xi)
\\
\tildekappa(\xi) 
&\equiv 
\kappa(x^{i+1} + V \xi)
.
\end{aligned}
\end{align}

We defined the \emph{Kronecker derivative} notation $\tkronderivelip{x}{x}$ in \S \ref{sec:hghr_ordr_prtls_off_alb} and the appendix, and it is merely a bookkeeping mechanism for mixed partial derivatives.
The patchy algorithm computes the $\ith{k}$ order partial derivatives with respect to the change of variables of the optimal cost, then recovers the partial derivatives with respect to the original state space variables by the formula
\begin{equation}
\label{eqn:rcvrd_prtls_cmptd}
\underset{k \text{ times}}{\kronderivelip{x}{x}} [\pi^{i+1}(x)]_{x=x^{i+1}}
= 
\underset{k \text{ times}}{\kronderivelip{\hat{\xi}}{\hat{\xi}}} [\hat{\pi}^{i+1}(\hat{\xi})]_{\xi=0}
\underset{k \text{ times}}{\kronprodelip{\hat{V}^T}{\hat{V}^T}}
.
\end{equation}
The exact $\ith{k}$ order partial derivatives of the optimal cost with respect to the original state space variables are recovered by 
\begin{equation}
\label{eqn:rcvrd_prtls_exct}
\underset{k \text{ times}}{\kronderivelip{x}{x}} [\pi(x)]_{x=x^{i+1}}
= 
\underset{k \text{ times}}{\kronderivelip{\xi}{\xi}} [\tilde{\pi}(\xi)]_{\xi=0}
\underset{k \text{ times}}{\kronprodelip{V^T}{V^T}}
.
\end{equation}

There are two main ideas behind the proof of Theorem \ref{thm:opt_cost_local_trunc_err_higer_order}.
The left hand side of \eqref{eqn:rcvrd_prtls_exct} is the same for any choice of orthogonal matrix $V$ in the change of variables $x = x^{i+1} + V \xi$.
In the case of the computed partial derivatives of the optimal cost, the entries in $\kronderivelip{\hat{\xi}}{\hat{\xi}}$ appearing on the right hand side of \eqref{eqn:rcvrd_prtls_cmptd} are a mixture of characteristic and non characteristic partial derivatives.
Despite this, the computed partial derivatives with respect to the original state space variables of the optimal cost are invariant under any choice of orthonormal vectors $\{\hat{V}^j\}_{j=1}^n$ in the change of variables $x = x^{i+1} + \hat{V} \hat{\xi}$ as long as $\hat{V}^1 = \dot{x}^{i+1} / \norm{\dot{x}^{i+1}}$.
This is the first main idea and is the content of Lemma \ref{lem:char_prtls_chng_crds_rltn}, Corollary \ref{cor:approx_char_prtls_chng_crds_rltn}, and Lemma \ref{lem:approx_prtls_chng_crds_rltn}.

The second idea is that there exist two convenient changes of variables that simplify our task of establishing the desired bound on the difference between the partial derivatives with respect to the original state space variables of the exact and computed optimal cost by bounding the difference between the partial derivatives under the convenient changes of variables.
The advantage of this approach is that we may treat the case of bounding the difference between the computed and exact characteristic partial derivatives independently of the non characteristic partial derivatives.

Once we have established these two facts, we will prove the theorem for the second order partial derivatives of the optimal cost and proceed by induction.

The relevant interpretation of Lemma \ref{lem:char_prtls_chng_crds_rltn} is that, taken as a group, the characteristic partial derivatives that the patchy algorithm computes at the current patch point can be expressed as an orthogonal transformation of a vector that depends only on partial derivatives of the problem data at the current patch point, and on partial derivatives of the optimal cost that are of strictly lower order than the characteristic partial derivatives being computed.
Secondly, all these partial derivatives are with respect to the original state space variables.
\begin{lemma}
\label{lem:char_prtls_chng_crds_rltn}
Suppose $V$ is an orthogonal matrix whose first column is the exact normalized optimal direction, $V^1 = \dot{x}/\norm{\dot{x}}$ at $x=x^{i+1}$.
Let $\tilde{\pi}(\xi)$ denote the optimal cost under the change of variables $x = x^{i+1} + V \xi$.
Then, at  $x=x^{i+1}$, there exist row vectors $w^p \in \mathbb{R}^{n^{p-1} \times 1}$, independent of $V^2, \ldots ,V^n$, such that the exact characteristic partial derivatives of the optimal cost satisfy
\begin{align}
\label{eqn:char_prtls_crds_rltn}
\begin{split}
\pderiv{}{}{\xi}\Bigl[ \pderiv{\tilde{\pi}}{}{\xi_1}(\xi) \Bigr]_{\xi = 0}
&=
\bigpars{w^2(x^{i+1})}^T V \\
\pderiv{}{}{\xi} \otimes \pderiv{}{}{\xi}\Bigl[ \pderiv{\tilde{\pi}}{}{\xi_1}(\xi) \Bigr]_{\xi = 0}
&=
\bigpars{w^3(x^{i+1})}^T V \otimes V \\
& \mspace{11 mu} \vdots \\
\kronderivelip{\xi}{\xi} \Bigl[ \pderiv{\tilde{\pi}}{}{\xi_1}(\xi) \Bigr]_{\xi = 0}
&=
\bigpars{w^p(x^{i+1})}^T \kronprodelip{V}{V} \\
\end{split}
\end{align}
where $w^k(x^{i+1})$ depends only on partial derivatives of the problem data and partial derivatives with respect to the original state space variables of the optimal cost of order $k-1$ and lower.

An analogous statement is true for the computed characteristic partial derivatives of the optimal cost.
\begin{proof}
To reduce the number of terms in some of the ensuing formulas, let $\tilde{F}$ and $\tilde{\ell}$ denote the dynamics and lagrangian as functions of the new coordinates.
\begin{align*}
\tilde{F}(\xi) \equiv \tilde{f}(\xi) + \tilde{g}(\xi) \tilde{\kappa}(\xi)
& &
\tilde{\ell}(\xi) \equiv \tilde{q}(\xi) + \frac{1}{2}\tilde{r}(\xi) \tilde{\kappa}(\xi)^2
\end{align*}
where the tilde functions appearing on the right hand sides are defined by the change of variables $x=x^{i+1} + V \xi$.

We will prove the lemma in the representative case of the third order characteristic partial derivatives of the optimal cost.
It follows from the HJB equations \eqref{eqn:drvtn_trnsfrms_HJBa} that a formula for the exact third order characteristic partial derivatives of the optimal cost is
\begin{multline}
\label{eqn:3rd_ordr_char_prtl}
\frac{\partial^3 \tilde{\pi}}{\partial \xi_{j_1} \partial \xi_{j_2} \partial \xi_1}(0)
=
-\frac{1}{\norm{\dot{x}}}
\Biggl[ 
\ppderiv{\tilde{\pi}}{2}{\xi_{j_1}}{}{\xi}{}(0) \pderiv{\tilde{F}}{}{\xi_{j_2}}(0)
+ \ppderiv{\tilde{\pi}}{2}{\xi_{j_2}}{}{\xi}{}(0) \pderiv{\tilde{F}}{}{\xi_{j_1}}(0)
\\
+ \pderiv{\tilde{\pi}}{}{\xi}(0) \ppderiv{\tilde{F}}{2}{\xi_{j_1}}{}{\xi_{j_2}}{}(0)
+ \ppderiv{\tilde{\ell}}{2}{\xi_{j_1}}{}{\xi_{j_2}}{}(0)
\Biggr]
.
\end{multline}
The right hand side of the formula \eqref{eqn:3rd_ordr_char_prtl} contains three distinctive types of terms characterized by the order of the partial derivatives of $\tilde{\pi}$, $\tilde{F}$, and $\tilde{\ell}$.
The first type is the product of a second order partial derivative of $\tilde{\pi}$ and a first order partial derivative of $\tilde{F}$, the second type is the product of a first order partial of $\tilde{\pi}$ and a second order partial of $\tilde{F}$, and the final type is a second order partial of $\ell$.

We will show that each type of term can be written as the product of a row vector that depends only on the partial derivatives of both $\pi(x)$ and $F(x)$ at $x=x^{i+1}$ and the column vector $V^{j_1} \otimes V^{j_2}$.
We tackle the first type of term first.
A formula for $\tppderiv{\tilde{\pi}}{2}{\xi_{j_1}}{}{\xi}{}(0) \tpderiv{\tilde{F}}{}{\xi_{j_2}}(0)$ is
\begin{equation} \label{eqn:3rd_ordr_char_prtl_term_1}
\ppderiv{\tilde{\pi}}{2}{\xi_{j_1}}{}{\xi}{}(0) \pderiv{\tilde{F}}{}{\xi_{j_2}}(0)
= 
\Biggl[
\sum_{\sigma=1}^n
\Bigpars{
\pderiv{}{}{x} \Bigl[ \pderiv{\pi}{}{x_{\sigma}}(x) \Bigr]_{x=x^{i+1}}
}
\otimes
\Bigpars{
\pderiv{}{}{x}\Bigl[ F_{\sigma}(x) \Bigr]_{x=x^{i+1}}
}
\Biggr] 
V^{j_1} \otimes V^{j_2}
.
\end{equation}
The quantity on the right hand side of \eqref{eqn:3rd_ordr_char_prtl_term_1} depends only on second order and lower partial derivatives of $\pi(x)$ and $F(x)$ evaluated at $x^{i+1}$.
To derive a formula for $\tppderiv{\tilde{\pi}}{2}{\xi_{j_2}}{}{\xi}{}(0) \tpderiv{\tilde{F}}{}{\xi_{j_1}}(0)$ with the desired factor $V^{j_1} \otimes V^{j_2}$, we use the fact that for every permutation $(i_1, i_2)$ of $(j_1,j_2)$, there exists a permutation matrix $P$, that depends only on the permutation and not the choice of basis vectors, such that $ V^{i_1} \otimes V^{i_2} = P ( V^{j_1} \otimes V^{j_2} )$.
Therefore, a relevant formula is%
\begin{multline*} %
\ppderiv{\tilde{\pi}}{2}{\xi_{j_2}}{}{\xi}{}(0) \pderiv{\tilde{F}}{}{\xi_{j_1}}(0) = \\
\Biggl[
\sum_{\sigma=1}^n
\Bigpars{
\pderiv{}{}{x} \Bigl[ \pderiv{\pi}{}{x_{\sigma}}(x) \Bigr]_{x=x^{i+1}}
}
\otimes
\Bigpars{
\pderiv{}{}{x}\Bigl[ F_{\sigma}(x) \Bigr]_{x=x^{i+1}}
}
P^T
\Biggr] 
V^{j_1} \otimes V^{j_2}
\end{multline*}

Although tedious, we can derive formulas for the other terms in \eqref{eqn:3rd_ordr_char_prtl} that are the product of a row vector that depends only on the second order and lower partial derivatives of the optimal cost and problem data with respect to the original state space variables, and the column vector $V^{j_1}$ and $V^{j_2}$.
It follows that there exists $w^3(x^{i+1}) \in \mathbb{R}^{n^2 \times 1}$ such that
\begin{equation*}
\frac{\partial \tilde{\pi}^3}{\partial \xi_{j_1} \partial \xi_{j_2} \partial \xi_1}(0)
=
\bigpars{w^3(x^{i+1})}^T V^{j_1} \otimes V^{j_2}
\end{equation*}
so the formula for $\tpderiv{}{}{\xi} \otimes \pderiv{}{}{\xi}\Bigl[ \pderiv{\tilde{\pi}}{}{\xi_1}(\xi) \Bigr]_{\xi = 0}$ in the theorem statement follows.

The proof for the computed characteristic partial derivatives is analogous.
\end{proof}
\end{lemma}

The relevant interpretation of Corollary \ref{cor:approx_char_prtls_chng_crds_rltn} is that, under two changes of variables that share the  normalized computed optimal direction at the current patch point as their first basis vector, the two sets of computed characteristic partial derivatives of the optimal cost at the current patch point are related to each other by an orthogonal transformation.
\begin{corollary}
\label{cor:approx_char_prtls_chng_crds_rltn}
If $\tilde{V}$ and $\hat{V}$ are are orthogonal matrices whose first columns are the normalized computed optimal direction at $x=x^{i+1}$, that is $\tilde{V}^1 = \hat{V}^1 = \dot{x}^{i+1} / \norm{\dot{x}^{i+1}}$, then the approximate second and higher order approximate characteristic partial derivatives of the optimal cost computed under the changes of variables
\begin{align*}
x = x^{i+1} + \tilde{V} \tilde{\xi} 
& &
x = x^{i+1} + \hat{V} \hat{\xi}
\end{align*}
are related by the orthogonal transformation
\begin{equation*}
\pderiv{}{}{\tilde{\xi}} \otimes \cdots \otimes \pderiv{}{}{\tilde{\xi}}
\Bigbracks{
\pderiv{\tilde{\pi}^{i+1}}{}{\tilde{\xi}_1}(\tilde{\xi})
}_{\tilde{\xi}=0}
= %
\pderiv{}{}{\hat{\xi}} \otimes \cdots \otimes \pderiv{}{}{\hat{\xi}}
\Bigbracks{
\pderiv{\hat{\pi}^{i+1}}{}{\hat{\xi}_1}(\hat{\xi})
}_{\hat{\xi}=0}
\hat{V}^T \tilde{V} \otimes \cdots \otimes \hat{V}^T \tilde{V}
.
\end{equation*}
\begin{proof}
The proof for the case of the approximate third order partials of the optimal cost is representative.
By Lemma \ref{lem:char_prtls_chng_crds_rltn}, there exists a single vector $w^{i+1}$ that depends only on the previously  computed lower order partial derivatives of $\pi^{i+1}(x)$ at $x=x^{i+1}$ such that the characteristic partials under both coordinate transformations are given by
\begin{align*}
\pderiv{}{}{\tilde{\xi}} \otimes \pderiv{}{}{\tilde{\xi}}
\Bigbracks{
\pderiv{\tilde{\pi}^{i+1}}{}{\tilde{\xi}_1}(\tilde{\xi})
}_{\tilde{\xi}=0}
&=
\bigpars{w^{i+1}}^T \tilde{V} \otimes \tilde{V}
\intertext{and}
\pderiv{}{}{\hat{\xi}} \otimes \pderiv{}{}{\hat{\xi}}
\Bigbracks{
\pderiv{\hat{\pi}^{i+1}}{}{\hat{\xi}_1}(\hat{\xi})
}_{\hat{\xi}=0}
&=
\bigpars{w^{i+1}}^T \hat{V} \otimes \hat{V}
.
\end{align*}
The conclusion of the corollary immediately follows from the identity
\begin{equation*}
\bigpars{w^{i+1}}^T \bigpars{ \kronprodelip{\tilde{V}}{\tilde{V}} }
=
\bigpars{w^{i+1}}^T 
\bigpars{\kronprodelip{\hat{V}}{\hat{V}} } \bigpars{ \kronprodelip{\hat{V}^T \tilde{V}}{\hat{V}^T \tilde{V}} }
.
\end{equation*}
\end{proof}
\end{corollary}

The content of Lemma \ref{lem:approx_prtls_chng_crds_rltn} is that, under two orthogonal changes of variables that share the normalized computed optimal direction at the current patch point as their first basis vector, the two sets of computed characteristic and non characteristic partial derivatives are related by an orthogonal transformation.
Thus, it follows immediately from \eqref{eqn:rcvrd_prtls_cmptd} that the computed partial derivatives of the optimal cost with respect to the original state space variables at the current patch point are the same under two different orthogonal changes of variables as long as the changes of variables share the normalized computed optimal direction as their first basis vector.
\begin{lemma}
\label{lem:approx_prtls_chng_crds_rltn}
If $\tilde{V}$ and $\hat{V}$ are orthogonal matrices whose first columns are the normalized optimal direction at $x=x^{i+1}$, that is $\tilde{V}^1 = \hat{V}^1 = \dot{x}^{i+1} / \norm{\dot{x}^{i+1}}$, then the approximate second and higher order approximate partial derivatives of the optimal cost computed under the changes of variables
\begin{align*}
x = x^{i+1} + \tilde{V} \tilde{\xi} 
& &
x = x^{i+1} + \hat{V} \hat{\xi}
\end{align*}
are related to each other by
\begin{multline*}
\pderiv{}{}{\tilde{\xi}} \otimes \cdots \otimes \pderiv{}{}{\tilde{\xi}}
\bigbracks{\tilde{\pi}^{i+1}(\tilde{\xi})}_{\tilde{\xi}=0}
\bigpars{\tilde{V}^T \otimes \cdots \otimes \tilde{V}^T}
= \\
\pderiv{}{}{\hat{\xi}} \otimes \cdots \otimes \pderiv{}{}{\hat{\xi}}
\bigbracks{\hat{\pi}^{i+1}(\hat{\xi})}_{\hat{\xi}=0}
\bigpars{\hat{V}^T \otimes \cdots \otimes \hat{V}^T}
.
\end{multline*}
\begin{proof}
The conclusion of the theorem will follow once we establish that the quantity
\begin{multline}
\label{eqn:prtls_st_sp_diff}
\pderiv{}{}{\tilde{\xi}} \otimes \cdots \otimes \pderiv{}{}{\tilde{\xi}}
\bigbracks{\tilde{\pi}^{i+1}(\tilde{\xi})}_{\tilde{\xi}=0}
\\
-
\pderiv{}{}{\hat{\xi}} \otimes \cdots \otimes \pderiv{}{}{\hat{\xi}}
\bigbracks{\hat{\pi}^{i+1}(\hat{\xi})}_{\hat{\xi}=0}
\bigpars{\hat{V}^T \tilde{V} \otimes \cdots \otimes \hat{V}^T \tilde{V}}
\end{multline}
is zero.
We will prove the lemma for the third order partial derivatives of the computed optimal cost, which is a representative case.

It follows immediately from the definition of the $\tpderiv{}{}{x} \otimes \cdots \otimes \tpderiv{}{}{x}$ notation that each entry in the term appearing on the left of \eqref{eqn:prtls_st_sp_diff} can be extracted by multiplying by one of the standard basis vectors.
Any standard basis vector in $\mathbb{R}^{n^3}$ can be expressed as a Kronecker product of three standard basis vectors in $\mathbb{R}^n$, so
\begin{equation*}
\pppderiv{\tilde{\pi}^{i+1}}{3}{\tilde{\xi}_{j_1}}{}{\tilde{\xi}_{j_2}}{}{\tilde{\xi}_{j_3}}{}(0)
=
\pderiv{}{}{\tilde{\xi}} \otimes \pderiv{}{}{\tilde{\xi}} \otimes \pderiv{}{}{\tilde{\xi}}
\bigbracks{\tilde{\pi}^{i+1}(\tilde{\xi})}_{\tilde{\xi}=0}
e_{j_1} \otimes e_{j_2} \otimes e_{j_3},
\end{equation*}
where $e_j$ denotes the $\ith{j}$ column of the identity matrix.
We will establish that the difference in \eqref{eqn:prtls_st_sp_diff} is zero by multiplying it by $e_{j_1} \otimes e_{j_2} \otimes e_{j_3}$ and showing that
\begin{equation*}
\pppderiv{\tilde{\pi}^{i+1}}{3}{\tilde{\xi}_{j_1}}{}{\tilde{\xi}_{j_2}}{}{\tilde{\xi}_{j_3}}{}(0)
-
\pderiv{}{}{\hat{\xi}} \otimes \pderiv{}{}{\hat{\xi}} \otimes \pderiv{}{}{\hat{\xi}}
\bigbracks{\hat{\pi}^{i+1}(\hat{\xi})}_{\hat{\xi}=0}
\bigpars{\hat{V}^T \tilde{V}^{j_1} \otimes \hat{V}^T\tilde{V}^{j_2} \otimes \hat{V}^T \tilde{V}^{j_3}}
\end{equation*}
is zero for any $1 \le j_1,j_2,j_3 \le n$, which is equivalent to showing that
\begin{multline} \label{eqn:prtls_st_sp_diff_sclr}
\pppderiv{\tilde{\pi}^{i+1}}{3}{\tilde{\xi}_{j_1}}{}{\tilde{\xi}_{j_2}}{}{\tilde{\xi}_{j_3}}{}(0)
\\
-
\sum_{\sigma_1=1}^n \sum_{\sigma_2=1}^n \sum_{\sigma_3=1}^n
\Bigpars{
\pppderiv{\hat{\pi}^{i+1}}{3}{\hat{\xi}_{\sigma_{j_1}}}{}{\hat{\xi}_{\sigma_{j_2}}}{}{\hat{\xi}_{\sigma_{j_3}}}{}(0)
\bigpars{\hat{V}^{\sigma_1} \cdot \tilde{V}^{j_1}} 
\bigpars{\hat{V}^{\sigma_2} \cdot \tilde{V}^{j_2}} 
\bigpars{\hat{V}^{\sigma_3} \cdot \tilde{V}^{j_3}}
}
\end{multline}
is zero.

The partial derivatives of $\tilde{\pi}^{i+1}(\tilde{\xi})$ and $\hat{\pi}^{i+1}(\hat{\xi})$ at $\tilde{\xi}=\hat{\xi}=0$ are computed by two different methods depending on whether they are characteristic or non characteristic partial derivatives, so the proof that the quantity in \eqref{eqn:prtls_st_sp_diff_sclr} is zero naturally breaks up into two cases.
In the characteristic case, at least one of the coordinate indices $j_1$, $j_2$, or $j_3$ is one, and we may assume that $j_3 = 1$ without a meaningful loss of generality.
The columns of $\tilde{V}$ and $\hat{V}$ are both orthonormal sets, and $\tilde{V}^1 = \hat{V}^1$ by assumption, so $\hat{V}^{\sigma_3} \cdot\tilde{V}^{1}= \delta_{1, \sigma_3}$, where $\delta_{1, \sigma_3}$ is the Kronecker delta.
Therefore for any $j_1,j_2$ in $1 \le j_1,j_2 \le n$, the difference in \eqref{eqn:prtls_st_sp_diff_sclr} reduces to a quantity that depends only on characteristic partial derivatives of the optimal cost
\begin{multline} \label{eqn:char_prtl_deriv_sum}
\pppderiv{\tilde{\pi}^{i+1}}{3}{\tilde{\xi}_{j_1}}{}{\tilde{\xi}_{j_2}}{}{\tilde{\xi}_1}{}(0)
-
\sum_{\sigma_1=1}^n \sum_{\sigma_2=1}^n
\Bigpars{
\pppderiv{\hat{\pi}^{i+1}}{3}{\hat{\xi}_{\sigma_{j_1}}}{}{\hat{\xi}_{\sigma_{j_2}}}{}{\hat{\xi}_1}{}(0)
\bigpars{\hat{V}^{\sigma_1} \cdot \tilde{V}^{j_1}} \bigpars{\hat{V}^{\sigma_2} \cdot \tilde{V}^{j_2}} 
} 
\\
\begin{aligned}
&=
\pppderiv{\tilde{\pi}^{i+1}}{3}{\tilde{\xi}_{j_1}}{}{\tilde{\xi}_{j_2}}{}{\tilde{\xi}_1}{}(0)
-
\pderiv{}{}{\hat{\xi}} \otimes \pderiv{}{}{\hat{\xi}} 
\Bigbracks{\pderiv{\hat{\pi}^{i+1}}{}{\hat{\xi}_1}(\hat{\xi}) }_{\hat{\xi}=0}
\bigpars{ \hat{V}^T \tilde{V}^{j_1} \otimes \hat{V}^T \tilde{V}^{j_2}} 
\\
&=
0
.
\end{aligned}
\end{multline}
The first equality follows from the definition of the $\tkronderivelip{x}{x}$, and the second equality follows from Corollary \ref{cor:approx_char_prtls_chng_crds_rltn}.
Therefore the quantity in \eqref{eqn:prtls_st_sp_diff_sclr} is zero whenever  $\pppderiv{\tilde{\pi}^{i+1}}{3}{\tilde{\xi}_{j_1}}{}{\tilde{\xi}_{j_2}}{}{\tilde{\xi}_{j_3}}{}(0)$ is a characteristic partial derivative.

We now show that the difference in \eqref{eqn:prtls_st_sp_diff_sclr} is zero in the non characteristic partial derivative case, that is, when none of the coordinate indices $j_1$, $j_2$, or $j_3$ are one.
By a standard argument
\begin{multline*}
\pppderiv{\tilde{\pi}^{i+1}}{3}{\tilde{\xi}_{j_1}}{}{\tilde{\xi}_{j_2}}{}{\tilde{\xi}_{j_3}}{} (0)
=\\
\Bigbracks{
\pderiv{}{}{x} \otimes \pderiv{}{}{x} \otimes \pderiv{}{}{x}
\bigbracks{\pi^i(x)}_{x=x^{i+1}}
\bigpars{\hat{V} \otimes \hat{V} \otimes \hat{V}}
}
\Bigbracks{
\hat{V}^T \tilde{V}^{j_1} \otimes \hat{V}^T \tilde{V}^{j_2} \otimes \hat{V}^T \tilde{V}^{j_3}
}
\end{multline*}
The right hand side of the previous equation is the product of a row and column vector. %
Since none of $j_1$, $j_2$, or $j_3$ is one, the previous equation can be rewritten as
\begin{multline} \label{eqn:nonchar_prtl_deriv_sum}
\pppderiv{}{3}{\tilde{\xi}_{j_1}}{}{\tilde{\xi}_{j_2}}{}{\tilde{\xi}_{j_3}}{} 
\bigbracks{\tilde{\pi}^{i+1}(\tilde{\xi})}_{\tilde{\xi}=0}
= \\
\sum_{\sigma_1=2}^n \sum_{\sigma_2=2}^n \sum_{\sigma_3=2}^n
\Bigpars{
\pderiv{}{}{x} \otimes \pderiv{}{}{x} \otimes \pderiv{}{}{x}
\bigbracks{\pi^i(x)}_{x=x^{i+1}}
\bigpars{\hat{V}^{\sigma_1} \otimes \hat{V}^{\sigma_2} \otimes \hat{V}^{\sigma_3}}
}
\\
\times
\Bigpars{
\bigpars{\hat{V}^{\sigma_1} \cdot \tilde{V}^{j_1}}
\bigpars{\hat{V}^{\sigma_2} \cdot \tilde{V}^{j_2}}
\bigpars{\hat{V}^{\sigma_3} \cdot \tilde{V}^{j_3}}
}
\end{multline}
where the summations all start at two because $\hat{V}^1 \cdot \tilde{V}^{j_1}=\hat{V}^1 \cdot \tilde{V}^{j_2}=\hat{V}^1 \cdot \tilde{V}^{j_3}=0$. 
Thus, the sum on the right hand side of \eqref{eqn:nonchar_prtl_deriv_sum} depends only on non characteristic partial derivatives of $\hat{\pi}^{i+1}(\hat{\xi})$ at $\hat{\xi}=0$.
Therefore when $j_1, j_2, j_3 \ne 1$, $\tpppderiv{\tilde{\pi}^{i+1}}{3}{\tilde{\xi}_j}{}{\tilde{\xi}_k}{}{\tilde{\xi}_{\ell}}{} (0)$ can be expressed as a linear transformation of $\kronprodelip{\tpderiv{}{}{\hat{\xi}}}{\tpderiv{}{}{\hat{\xi}}}[\hat{\pi}^{i+1}(\hat{\xi})]_{\hat{\xi}=0}$, whose entries are composed of both characteristic and non characteristic partial derivatives.
The conclusion of the theorem follows from the fact that \eqref{eqn:char_prtl_deriv_sum} and \eqref{eqn:nonchar_prtl_deriv_sum} cover all of \eqref{eqn:prtls_st_sp_diff_sclr}.
\end{proof}
\end{lemma}

We now turn to the task of constructing two convenient changes of variables, meaning we can pick $V$ and $\hat{V}$ in \eqref{eqn:chngs_of_vars} such that $V$ and $\hat{V}^T$ are almost inverses of each other. 
This is the content of Lemma \ref{lem:cnvnt_basis_id_pert_bnd}, which in turn relies on Lemma \ref{lem:canonical_basis_exists}, and Corollaries \ref{cor:canonical_mat_prod_iden_diff_ip_bnd} and \ref{cor:canonical_mat_prod_iden_diff_v_bnd}.%

\begin{lemma}
\label{lem:cnvnt_basis_id_pert_bnd}
Let $\dot{x}^{i+1}$ denote the optimal direction at $x=x^{i+1}$ computed by the patchy algorithm from the exact partial derivatives of the optimal cost at the previous patch point.
If the optimal cost $\pi(x)$ is a strict Lyapunov function and the distance between the current and previous patch point is less than the maximum consecutive patch point distance $h>0$, then there exists a constant $B < \infty$ and orthogonal matrices $V$ and $\hat{V}$ such that
\begin{align*}
V^1 = \frac{\dot{x}}{\norm{\dot{x}}} & & \hat{V}^1 = \frac{\dot{x}^{i+1}}{\norm{\dot{x}^{i+1}}}
\end{align*}
and
\begin{equation*}
\norm{\hat{V}^T V - I } \le B \norm{x^{i+1} - x^i}^{d+1}
\end{equation*}
where the degree of the approximate optimal cost polynomial $\pi^{i+1}$ is $d+1$. 

\begin{proof}
By Theorem \ref{thm:phi1_lte_bound}, there is a constant $T < \infty$ and a maximum consecutive patch point distance $h>0$ such that 
\begin{equation*}
\Bignorm{
\pderiv{\pi}{}{x}(x^{i+1}) - \pderiv{\pi^{i+1}}{}{x}(x^{i+1})
}
\le T \norm{x^{i+1} - x^i}^{d+1}
.
\end{equation*}
Therefore, there exists an $\epsilon > 0$ such that, for any $x^{i+1}$ in $\closedcompdomwoalb$ within a radius of $h$ of the previous patch point
\begin{align*}
\norm{\dot{x} - \dot{x}^{i+1}}
=
\mathcal{O} \bigpars{\norm{x^{i+1} - x^i}^{d+1}}
& &\text{and}& &
0 < \epsilon \le \norm{\dot{x}} - \norm{\dot{x} - \dot{x}^{i+1}}  
\end{align*}
The result of the lemma follows from Lemma \ref{lem:canonical_basis_exists} where $V$ and $\hat{V}$ are constructed, and Corollaries \ref{cor:canonical_mat_prod_iden_diff_ip_bnd} and \ref{cor:canonical_mat_prod_iden_diff_v_bnd} where the bound on $\norm{\hat{V}^T V - I}$ in the lemma statement is established.
\end{proof}
\end{lemma}

We construct orthogonal matrices $V$ and $\hat{V}$ associated with the changes of variables for the exact and computed optimal cost such that $\hat{V}^T$ and $V$ are almost inverses of each other.
This is the content of Lemma \ref{lem:canonical_basis_exists}, which is mostly a restatement of \cite[pg.~73]{stewart:mat_alg} in terms relevant to our problem.
\begin{lemma}
\label{lem:canonical_basis_exists}
Suppose $v$ and $\hat{v}$ are nonzero vectors in $\mathbb{R}^n$ such that $v \cdot\hat{v} > 0$, then there exist orthonormal bases $\{ V^1, \ldots , V^n \}$ and $\{\hat{V}^1 , \ldots , \hat{V}^n \}$ with the properties
\begin{enumerate}
\item
$V^1 = v/\norm{v}$ and $\hat{V}^1 = \hat{v}/\norm{\hat{v}}$
\item
$
\sum_{j=2}^n \bigpars{V^1 \cdot \hat{V}^j}^2 
= 
\sum_{j=2}^n \bigpars{\hat{V}^1 \cdot V^j}^2
=
1 - \bigpars{V^1 \cdot \hat{V}^1}^2
$

\item
The sets $\{ V^2, \ldots , V^n \}$ and $\{\hat{V}^2 , \ldots , \hat{V}^n \}$ satisfy the orthogonality condition $\bigpars{V^j \cdot \hat{V}^k} = 0$ whenever $j \ne k$ and $V^1 \cdot \hat{V}^1 \le V^j \cdot \hat{V}^j\le 1$ for $2 \le j \le n$
\end{enumerate}

\begin{proof}
The first property from the lemma statement is trivially achieved by setting $V^1$ and $\hat{V}^1$ to the unit vectors in the direction of $v$ and $\hat{v}$.

Let $W$ be any $n \times n-1$ matrix whose columns are orthogonal to $V^1$, and define $\hat{W}$ analogously.
Let $M$ and $\hat{M}$ be the orthogonal matrices
\begin{align*}
M \equiv \begin{bmatrix} V^1 & W \end{bmatrix}
& &
\hat{M} \equiv \begin{bmatrix} \hat{V}^1 & \hat{W} \end{bmatrix}
\end{align*}
then $\hat{M}^T M$ is orthogonal.
The second statement of the lemma is just a restatement of the fact that the first row and column of $\hat{M}^TM$ both have unit two-norm.

The third statement in the lemma statement requires a special choice of bases for the subspaces orthogonal to $V^1$ and $\hat{V}^1$.
Let $\hat{W}^T W$ have the singular value decomposition 
\begin{equation*}
\hat{W}^T W = \hat{U} \Sigma U^T
\end{equation*}
where $\Sigma$ is an $n-1 \times n-1$ diagonal matrix.
If we take the sets $\{ V^j \}_{j=2}^n$ and $\{ \hat{V}^j \}_{j=2}^n$  to be the columns of $WU$ and $\hat{W} \hat{U}$, then $\{V^j\}_{j=2}^n$ and $\{\hat{V}^j\}_{j=2}^n$ are orthonormal bases for the span of the spaces orthogonal to $V^1$ and $\hat{V}^1$.
Furthermore, it follows from the fact that $(\hat{W} \hat{U})^T W U = \Sigma$ that $\{V^j\}_{j=2}^n$ and $\{\hat{V}^j\}_{j=2}^n$ are biorthogonal and $V^j \cdot \hat{V}^j$ is a singular value of $\hat{W}^T W$ when $2 \le j \le n$.

Finally, we must verify that $V^1 \cdot \hat{V}^1 \le V^j \cdot \hat{V}^j \le 1$ for $2 \le j \le n$.
Let $N$ and $\hat{N}$ be the orthogonal matrices
\begin{align*}
N \equiv \begin{bmatrix} V^1 & W U \end{bmatrix}
& &
\hat{N} \equiv \begin{bmatrix} \hat{V}^1 & \hat{W} \hat{U} \end{bmatrix}
\end{align*}
Their product has the form
\begin{equation*}
\hat{N}^T N
=
\begin{bmatrix}
\innerprod{\hat{V}^1}{V^1} & (\hat{V}^1)^T (W U) \\
(\hat{W} \hat{U})^T V^1 & \Sigma
\end{bmatrix}
.
\end{equation*}
Each column of of $\hat{N}^T N$ has unit norm, in particular the last column contains $\sigma_{\text{min}}(\hat{W}^T W)$, the smallest singular value of $\hat{W}^T W$, and thus satisfies
\begin{align*}
1 
&\le
\sigma_{\text{min}}(\hat{W}^T W)^2 + \norm{(\hat{V}^1)^T W U}^2 \\
&=
\min_j \bigpars{V^j \cdot \hat{V}^j}^2 + \norm{(\hat{V}^1)^T W U}^2 \\
&\le
\min_j \bigpars{V^j \cdot \hat{V}_j}^2 + 1 - \bigpars{V^1 \cdot \hat{V}^1}^2
.
\end{align*}
It then follows from the final inequality and the assumption that $\bigpars{V^1 \cdot \hat{V}^1} > 0$ that $V^j \cdot \hat{V}^j \ge V^1 \cdot \hat{V}^1$ for $2 \le j \le n$.
\end{proof}
\end{lemma}

Corollaries \ref{cor:canonical_mat_prod_iden_diff_ip_bnd} and \ref{cor:canonical_mat_prod_iden_diff_v_bnd} establish a bound on $\norm{\hat{V}^T V - I}$ in terms of the difference between the true and computed optimal directions at the current patch point.
\begin{corollary}
\label{cor:canonical_mat_prod_iden_diff_ip_bnd}
Let $V$ and $\hat{V}$ denote the orthogonal matrices whose columns are $\{ V^j \}_{j=1}^n$ and $\{\hat{V}^j\}_{j=1}^n$ from Lemma \ref{lem:canonical_basis_exists}, then 
\begin{equation*}
\norm{I - \hat{V}^T V }_2 \le 2 \sqrt{2} \sqrt{1 - V_1 \cdot \hat{V}_1} 
\end{equation*}

\begin{proof}
Partition the matrices $V$ and $\hat{V}$ as
\begin{align*}
V = \begin{bmatrix} V_1 & W \end{bmatrix}
& &
\hat{V} = \begin{bmatrix} \hat{V}_1 & \hat{W} \end{bmatrix}
.
\end{align*}
$W$ and $\hat{W}$ were constructed so that their columns have unit norm and satisfy $W^j \cdot \hat{W}^k = 0$ whenever $j \ne k$.
Let $\Sigma = \hat{W}^T W$, where $\Sigma$ is diagonal and its diagonal entries are $\hat{W}_j \cdot W_j$.
$I - \hat{V}V$ can be expressed as
\begin{equation}
\label{eqn:VhatTV_decomp}
I - \hat{V}^T V
=
\begin{bmatrix}
1 - V_1 \cdot \hat{V}_1 & 0 \\
0 & I - \Sigma
\end{bmatrix}
+
\begin{bmatrix}
0 & \hat{V}_1^T W \\
\hat{W}^T V_1 & 0
\end{bmatrix}
.
\end{equation}
By Lemma \ref{lem:canonical_basis_exists}, $\Sigma_{j,j} = \hat{W}_j \cdot W_j \ge \hat{V}_1 \cdot V_1$, so the norm of the diagonal matrix appearing in the right hand side of \eqref{eqn:VhatTV_decomp} is bounded by
\begin{equation}
\label{eqn:VhatTV_decomp_bound1}
\left\|
\begin{bmatrix}
1 - V_1 \cdot \hat{V}_1 & 0 \\
0 & I - \Sigma
\end{bmatrix}
\right\|_2
\le
1 - \hat{V}_1 \cdot V_1
\end{equation}
Let $x$ be a unit column vector partitioned as $x = [ x_1 \; x_2 ]^T$, then the following bound on the norm of the second matrix in \eqref{eqn:VhatTV_decomp} follows from the fact that $\norm{\hat{V}_1^T W}^2 = \norm{\hat{W}^T V_1}^2 = 1 - \hat{V}_1 \cdot V_1^2$
\begin{align}
\label{eqn:VhatTV_decomp_bound2}
\begin{split}
\left\|
\begin{bmatrix}
0 & \hat{V}_1^T W \\
\hat{W}^T V_1 & 0
\end{bmatrix}
\begin{bmatrix}
x_1 \\ x_2
\end{bmatrix}
\right\|_2^2
&=
\norm{\hat{V}_1^T W x_2}^2 + \norm{\hat{W}^T V_1 x_1}^2 \\
&\le
\norm{\hat{V}_1^T W}^2 \norm{x_2}^2 + \norm{\hat{W}^T V_1}^2(1-\norm{x_2}^2) \\
&\le
2\bigpars{1 - \hat{V}_1 \cdot V_1}
\end{split}
\end{align}
The conclusion of the lemma follows from \eqref{eqn:VhatTV_decomp_bound1} and \eqref{eqn:VhatTV_decomp_bound2}.
\end{proof}
\end{corollary}

\begin{corollary}
\label{cor:canonical_mat_prod_iden_diff_v_bnd}
Suppose $v$ is a nonzero vector, and $\norm{v - \hat{v}}_2 < \norm{v}_2$, then there exist orthogonal matrices $V$ and $\hat{V}$ such that their first columns are
\begin{align*}
V^1 &= \frac{v}{\norm{v}} & \hat{V}^1 &= \frac{\hat{v}}{\norm{\hat{v}}}
\end{align*}
and
\begin{equation*}
\norm{I - \hat{V}^T V} \le 4 \frac{\norm{v - \hat{v}}}{\norm{v} - \norm{v - \hat{v}}}
\end{equation*}

\begin{proof}
The corollary follows from Corollary \ref{cor:canonical_mat_prod_iden_diff_ip_bnd} and the fact that
\begin{equation*}
\Bignorm{\frac{v}{\norm{v}} - \frac{\hat{v}}{\norm{\hat{v}}}}
\le
2 \frac{\norm{v - \hat{v} }}{\norm{v} - \norm{v - \hat{v}}}
\text{ whenever }
\norm{v - \hat{v}} < \norm{v}.
\end{equation*}
\end{proof}
\end{corollary}

The relevant interpretation of Lemma \ref{lem:kron_prd_id_diff} for the purpose of proving Theorem \ref{thm:opt_cost_local_trunc_err_higer_order} is that $\hat{V}^T V$ and the repeated Kronecker product $\hat{V}^T V$ of the orthogonal matrices from the convenient changes of coordinates are both perturbations of the identity of the same order.
This fact is useful for bounding the difference between partial derivatives of the same function under the two changes of coordinates.
\begin{lemma} \label{lem:kron_prd_id_diff}
If $A$ and $B$ are matrices such that $AB$ is in $\mathbb{R}^{n \times n}$ and $\norm{AB-I_n}_2 \le 1$, then
\begin{equation*}
\norm{\bigpars{AB \otimes \cdots \otimes AB} - I_{n^k}}_2 \le (2^k - 1) \norm{AB - I_n}_2
\end{equation*}
where the Kronecker product on the left hand side of the inequality is repeated $k$ times, and $I_n$ and $I_{n^k}$ denote the identity matrices in $\mathbb{R}^{n \times n}$ and $\mathbb{R}^{n^p \times n^p}$.

\begin{proof}
The proof of the lemma rests on the fact that for any two matrices $C$ and $D$, $\norm{C \otimes D}_2 = \norm{C}_2 \norm{D}_2$.
A standard induction argument establishes the conclusion of the lemma.
\end{proof}
\end{lemma}

We now turn back to the task of bounding the difference between an exact and computed partial derivative of the optimal cost with respect to the original state space variables at the current patch point.
This difference is given by the formula
\begin{multline}
\label{eqn:rcvrd_opt_cost_prtl_deriv_diff}
\pderivelip{\pi}{k}{x_{j_1}}{}{x_{j_k}}{}(x^{i+1})
-
\pderivelip{\pi^{i+1}}{p}{x_{j_1}}{}{x_{j_k}}{}(x^{i+1})
=\\
\Biggbracks{
\kronderivelip{\xi}{\xi} [\tilde{\pi}(\xi)]_{\xi=0}
-
\kronderivelip{\hat{\xi}}{\hat{\xi}} [\hat{\pi}^{i+1}(\hat{\xi})]_{\hat{\xi}=0}
\\
-
\kronderivelip{\xi}{\xi} [\tilde{\pi}(\xi)]_{\xi=0}
\bigpars{ \kronprodelip{V^T \hat{V}}{V^T \hat{V}} - I }
}
\\
\kronprodelip{\hat{V}^T e_{j_1}}{\hat{V}^T e_{j_k}}
.
\end{multline}
There are a few things to note about the right hand side of \eqref{eqn:rcvrd_opt_cost_prtl_deriv_diff}.
The first is that it is the product of a row vector on the left and column vector on the right, where the column vector $\kronprodelip{\hat{V}^T e_{j_1}}{\hat{V}^T e_{j_k}}$ has unit norm since it is the column of an orthogonal matrix.
It follows from the Cauchy-Schwarz inequality that the the right hand side of \eqref{eqn:rcvrd_opt_cost_prtl_deriv_diff} is bounded by the norm of the row vector.
This row vector has two terms, the first is the difference between the exact and computed partial derivatives of the optimal cost with respect to the changes of variables.
The second is involves a vector composed of partial derivatives of the exact solution times a small perturbation of the identity.
It follows from assumption that the optimal cost is smooth on the compact computational domain and by Lemma \ref{lem:cnvnt_basis_id_pert_bnd} that there exist $T' < \infty$ such that, whenever the current patch point is within a radius of the maximum consecutive patch point distance of the previous patch point, this term is bounded by
\begin{multline}
\label{eqn:higher_order_trunc_err_term1_bnd}
\Bignorm{
\kronderivelip{\xi}{\xi} [ \tilde{\pi}(\xi) ]_{\xi=0}
\bigpars{ \kronprodelip{V^T \hat{V}}{V^T \hat{V}} - I}
}
\\
\begin{aligned}
&\le
\bigpars{2^k-1} 
\sup_{\closedcompdomwoalb} \Bignorm{\kronderivelip{x}{x}[\pi(x)]} \norm{V^T \hat{V} - I} \\
&\le
T' \norm{x^{i+1} - x^i}^{d+1}
.
\end{aligned} 
\end{multline}
This bound is independent of $p$, the order of the partial derivative, when $p \ge 2$.
The remaining work is in establishing the existence of some $T'' < \infty$, such that under the convenient changes of variables
\begin{multline}
\label{eqn:rcvrd_opt_cost_prtl_deriv_diff_just_prtls_bnd}
\Bignorm{
\kronderivelip{\xi}{\xi} [\tilde{\pi}(\xi)]_{\xi=0}
-
\kronderivelip{\hat{\xi}}{\hat{\xi}} [\hat{\pi}^{i+1}(\hat{\xi})]_{\hat{\xi}=0}
}_2
\\
\le
T'' \norm{x^{i+1}-x^i}^{d+2-k}
\end{multline}
whenever the distance between the current and previous patch point is less than the maximum consecutive patch point distance.
We will bound the left hand side of \eqref{eqn:rcvrd_opt_cost_prtl_deriv_diff_just_prtls_bnd} by bounding each of its entries by considering two cases, depending on whether the entry is the difference of exact and computed characteristic or non characteristic partial derivatives of the optimal cost.

With this approach in mind, we will need a bound on the partial derivatives of the problem data with respect to two changes of variables.
This is the content of Lemma \ref{lem:prtl_drv_chng_crds_diff}.
\begin{lemma}
\label{lem:prtl_drv_chng_crds_diff}
Suppose all partial derivatives of $s : \mathbb{R}^n \rightarrow \mathbb{R}$ up to order $k$ exist and are continuous in a neighborhood of $\bar{x}$ in $\mathbb{R}^n$.
If $\tilde{V}$ and $\hat{V}$ are orthogonal matrices and $\tilde{s}$ and $\hat{s}$ are defined as
\begin{align*}
\tilde{s}(\tilde{\xi}) \equiv s(\bar{x} + \tilde{V} \tilde{\xi} )
& &
\hat{s}(\hat{\xi}) \equiv s(\bar{x} + \hat{V} \hat{\xi})
\end{align*}
then
\begin{multline*}
\Bigabs{
\pderivelip{\tilde{s}} {k} {\tilde{\xi}_{j_1}} {} {\tilde{\xi}_{j_k}} {}(0)
-
\pderivelip{\hat{s}} {k} {\hat{\xi}_{j_1}} {} {\hat{\xi}_{j_k}} {}(0)
}
\le \\
\bigpars{2^k-1} \Bignorm{\kronderivelip{x}{x}[s(x)]_{x=\bar{x}}}_2 \norm{\hat{V}^T \tilde{V} - I}_2
\end{multline*}

If the components of $G : \mathbb{R}^n \rightarrow \mathbb{R}^n$ all have partial derivatives up to order $p$ in a neighborhood of $\bar{x}$ in $\mathbb{R}^n$, and $\tilde{G}$ and $\hat{G}$ are defined as
\begin{align*}
\tilde{G}(\tilde{\xi}) \equiv \tilde{V}^T G(\bar{x} + \tilde{V} \tilde{\xi} )
& &
\hat{G}(\hat{\xi}) \equiv \hat{V}^T G(\bar{x} + \hat{V} \hat{\xi})
\end{align*}
then
\begin{multline*}
\Bignorm{
\pderivelip{\tilde{G}} {k} {\tilde{\xi}_{j_1}} {} {\tilde{\xi}_{j_k}} {}(0)
-
\pderivelip{\hat{G}} {k} {\hat{\xi}_{j_1}} {} {\hat{\xi}_{j_k}} {}(0)
}
\le \\
2^k n^{1/2} 
\max_{\ell} \Bignorm{\kronderivelip{x}{x}[G_{\ell}(x)]_{x=\bar{x}}}_2 \norm{\hat{V}^T \tilde{V} - I}_2
.
\end{multline*}

\begin{proof}
A $\ith{k}$ order partial derivative of $\tilde{s}$ can be expressed in Kronecker notation as
\begin{equation*}
\pderivelip{\tilde{s}}{k}{\tilde{\xi}_{j_1}}{}{\tilde{\xi}_{j_k}}{}(0)
= 
\kronderivelip{x}{x} [s(x)]_{x=\bar{x}} 
(\kronprodelip{\tilde{V}}{\tilde{V}})
(\kronprodelip{e_{j_1}} {e_{j_k}})
\end{equation*}
so after a sequence of algebraic manipulations, the difference between the partial derivatives of $\tilde{s}$ and $\hat{s}$ is
\begin{multline*}
\Bigabs{
\pderivelip{\tilde{s}}{k}{\tilde{\xi}_{j_1}}{}{\tilde{\xi}_{j_k}}{}(0)
-
\pderivelip{\hat{s}}{k}{\hat{\xi}_{j_1}}{}{\hat{\xi}_{j_k}}{}(0)
}
\\
\begin{aligned}
&=\Bigabs{
\kronderivelip{x}{x}[s]_{x=\bar{x}}
\bigpars{\kronprodelip{V \hat{V}^T}{V \hat{V}^T} - I}
\bigpars{ \kronprodelip{\hat{V}^{j_1}}{\hat{V}^{j_k}} }
}
\\
&\le
\bigpars{2^k-1} \Bignorm{\kronderivelip{x}{x} [s(x)]_{x=\bar{x}}}_2  \norm{\tilde{V}\hat{V}^T-I}_2
.
\end{aligned}
\end{multline*}
The last inequality follows from the Cauchy-Schwarz inequality, Lemma \ref{lem:kron_prd_id_diff}, and the facts that $\kronprodelip{\hat{V}^{j_1}}{\hat{V}^{j_k}}$ has unit norm and $\norm{\tilde{V}\hat{V}^T-I}_2 = \norm{\hat{V}^T\tilde{V}-I}_2$.

The difference between the partial derivatives of $G$ under the two changes of variables is bounded by
\begin{multline*}
\Bignorm{
\pderivelip{}{k}{\tilde{\xi}_{j_1}}{}{\tilde{\xi}_{j_k}}{}[G(\bar{x} + \tilde{V} \tilde{\xi})]_{\tilde{\xi}=0}
-
\pderivelip{}{k}{\hat{\xi}_{j_1}}{}{\hat{\xi}_{j_k}}{}[G(\bar{x} + \hat{V} \hat{\xi})]_{\hat{\xi}=0}
}_2
\\
\le
\bigpars{2^k-1} n^{1/2} \max_{\ell} \Bignorm{\kronderivelip{x}{x} [G_{\ell}(x)]_{x=\bar{x}}}_2  \norm{\tilde{V}\hat{V}^T-I}_2
.
\end{multline*}
This bound is derived by using the result of the first half of the lemma to bound the infinity norm of the difference between the partial derivatives.
The second result of the lemma follows from the previous bound and the identity
\begin{equation*}
\tilde{G}(\tilde{\xi}) - \hat{G}(\hat{\xi})
=
\hat{V}^T 
\bigbracks{
G(\bar{x} + \tilde{V} \tilde{\xi}) - G(\bar{x} + \hat{V} \hat{\xi})
+ (\hat{V} \tilde{V}^T - I) G(\bar{x} + \tilde{V} \tilde{\xi})
}
.
\end{equation*}
\end{proof}
\end{lemma}

The next corollary differs from the previous lemma in that it bounds the difference in partial derivatives of two different functions with respect to two changes of variables.
\begin{corollary}
\label{cor:prtl_drv_chng_crds_diff2}
Suppose all partial derivatives of $s,\bar{s} : \mathbb{R}^n \rightarrow \mathbb{R}$ up to order $k$ exist and are continuous in a neighborhood of $\bar{x}$ in $\mathbb{R}^n$.
If $\tilde{V}$ and $\hat{V}$ are orthogonal matrices and $\tilde{s}$ and $\hat{s}$ are defined as
\begin{align*}
\tilde{s}(\tilde{\xi}) \equiv s(\bar{x} + \tilde{V} \tilde{\xi} )
& &
\hat{s}(\hat{\xi}) \equiv \bar{s}(\bar{x} + \hat{V} \hat{\xi})
\end{align*}
then
\begin{multline*}
\Bigabs{
\pderivelip{\tilde{s}} {k} {\tilde{\xi}_{j_1}} {} {\tilde{\xi}_{j_k}} {}(0)
-
\pderivelip{\hat{s}} {k} {\hat{\xi}_{j_1}} {} {\hat{\xi}_{j_k}} {}(0)
}
\le \\
\Bignorm{\kronderivelip{x}{x}[s(x) - \bar{s}(x)]_{x=\bar{x}}}_2  \\
+ \bigpars{2^k-1} \Bignorm{\kronderivelip{x}{x}[s(x)]_{x=\bar{x}}}_2 \norm{\hat{V}^T \tilde{V} - I}_2
\end{multline*}

\begin{proof}
\begin{multline*}
\Bigabs{
\pderivelip{\tilde{s}} {k} {\tilde{\xi}_{j_1}} {} {\tilde{\xi}_{j_k}} {}(0)
-
\pderivelip{\hat{s}} {k} {\hat{\xi}_{j_1}} {} {\hat{\xi}_{j_k}} {}(0)
}
\\
\begin{aligned}
&=
\Bigabs{
\Bigpars{
\kronderivelip{x}{x} [s(x)-\bar{s}(x)]_{x=\bar{x}}} \bigpars{\kronprodelip{\hat{V}^{j_1}}{\hat{V}^{j_k}}}  \\
&\phantom{\le \bigl\|}
+
\kronderivelip{x}{x} [s(x)]_{x=\bar{x}}
\Bigpars{\kronprodelip{\hat{V}^T\tilde{V}}{\hat{V}^T\tilde{V}}}
(\kronprodelip{e_{j_1}}{e_{j_k}})
}
\\
&\le
\Bignorm{\kronderivelip{x}{x}[s(x) - \bar{s}(x)]_{x=\bar{x}}}_2  \\
&\phantom{\le \bigl\|}
+ \bigpars{2^k-1} \Bignorm{\kronderivelip{x}{x}[s(x)]_{x=\bar{x}}}_2 \norm{\hat{V}^T \tilde{V} - I}_2
\end{aligned}
\end{multline*}
\end{proof}

\end{corollary}

We now have all the ingredients necessary to prove Theorem \ref{thm:opt_cost_local_trunc_err_higer_order}.
\begin{proof}[Proof of Theorem \ref{thm:opt_cost_local_trunc_err_higer_order}]
We denote the problem data under the two changes of variables in \eqref{eqn:chngs_of_vars} as
\begin{align}
\label{eqn:prob_data_chng_vars}
\begin{aligned}
\tilde{f}(\xi) &\equiv V^T f(x^{i+1} + V \xi) \\
\tilde{g}(\xi) &\equiv V^T g(x^{i+1} + V \xi) \\
\tilde{q}(\xi) &\equiv q(x^{i+1} + V \xi) \\
\tilde{r}(\xi) &\equiv r(x^{i+1} + V \xi) 
\end{aligned}
& &
\begin{aligned}
\hat{f}(\hat{\xi}) &\equiv \hat{V}^T f(x^{i+1} + \hat{V} \hat{\xi}) \\
\hat{g}(\hat{\xi}) &\equiv \hat{V}^T g(x^{i+1} + \hat{V} \hat{\xi}) \\
\hat{q}(\hat{\xi}) &\equiv \hat{q}(x^{i+1} + \hat{V} \hat{\xi}) \\
\hat{r}(\hat{\xi}) &\equiv \hat{r}(x^{i+1} + \hat{V} \hat{\xi}) 
\end{aligned}
\end{align}
We first consider the case of the second order partial derivatives of the optimal cost, so we must verify that the left hand side of \eqref{eqn:rcvrd_opt_cost_prtl_deriv_diff_just_prtls_bnd} is $\mathcal{O}(\norm{x^{i+1}-x^i}^d)$.
The characteristic second order partial derivatives of the optimal cost are given by \eqref{eqn:ddpiddxi_cmptd} and the computed optimal control is given by \eqref{eqn:dkappa_dxi_cmptd}.
The exact second order partial derivatives of the optimal const and the optimal control satisfy the equations
\begin{multline} %
\ppderiv{\tilde{\pi}}{2}{\xi_j}{}{\xi_1}{}(0)
=
-\frac{1}{\norm{\dot{x}}}
\Biggbracks{
\pderiv{\tildepi}{}{\xi}(0)
\Bigpars{
\pderiv{\tildef}{}{\xi_j}(0) + \pderiv{\tildeg}{}{\xi_j}(0) \tildekappa(0)
\\
+
\pderiv{\tildeq}{}{\xi_j}(0) + \frac{1}{2} \pderiv{\tilder}{}{\xi_j}(0) \tildekappa(0)^2
}
}
\end{multline}
and
\begin{equation} %
\tildekappa(0) = -\frac{1}{\tilder(0)}
\Bigpars{
\pderiv{\tildepi}{}{\xi}(0) \tildeg(0)
}
\end{equation}
The difference between $\tppderiv{\hat{\pi}^{i+1}}{2}{\hat{\xi}_j}{}{\hat{\xi}_1}{}(0)$ and $\ppderiv{\tilde{\pi}}{2}{\xi_j}{}{\xi_1}{}(0)$ can be bounded in terms of two types of differences.
The first type involves terms that are the difference between the partial derivatives with respect to the two changes of variables of the problem data.
\begin{align}
\label{eqn:prob_data_chng_vars_diff}
\begin{aligned}
\pderiv{\tildef}{}{\xi}(0) &- \pderiv{\hatf}{}{\hat{\xi}}(0)
\\
\pderiv{\tildeq}{}{\xi}(0) &- \pderiv{\hatq}{}{\hat{\xi}}(0)
\end{aligned}
& &
\begin{aligned}
\pderiv{\tildeg}{}{\xi}(0) &- \pderiv{\hatg}{}{\hat{\xi}}(0) 
\\
\pderiv{\tilder}{}{\xi}(0) &- \pderiv{\hatr}{}{\hat{\xi}}(0) 
\end{aligned}
\end{align}
If the distance between the current and previous patch point is less than the maximum consecutive patch point distance $h$, then by Lemmas \ref{lem:cnvnt_basis_id_pert_bnd} and \ref{lem:prtl_drv_chng_crds_diff}, there is a single bound that is proportional to $\norm{x^{i+1}-x^i}^{d+1}$ that bounds each difference in \eqref{eqn:prob_data_chng_vars_diff}.

The second type of difference involves terms that are the difference between exact and computed values of derivatives with respect to the changes of coordinates of the optimal cost.
These terms are
\begin{align}
\label{eqn:soln_chng_vars_diff}
\pderiv{\tildepi}{}{\xi}(0) - \pderiv{\hatpi^{i+1}}{}{\hat{\xi}}(0)
& &
\tildekappa(0) - \hatkappa^{i+1}(0)
& &
\frac{1}{\norm{\dot{x}^{i+1}}} - \frac{1}{\norm{\dot{x}}} 
\end{align}
By Theorem \ref{thm:phi1_lte_bound}, the difference between the computed and exact first order partial derivatives of the optimal cost is bounded by
\begin{equation}
\label{eqn:1st_ordr_opt_cost_diff_r_bnd}
\Bignorm{ \pderiv{\pi}{}{x}(x^{i+1}) - \pderiv{\pi^{i+1}}{}{x}(x^{i+1}) } \le T' \norm{x^{i+1} - x^i}^{d+1}
\end{equation}
for some $T' < \infty$ and by Lemma \ref{lem:cnvnt_basis_id_pert_bnd}, there exists $B < \infty$ such that
\begin{equation}
\label{eqn:can_basis_id_pert_r_bnd}
\norm{\hat{V}^T V - I} \le B \norm{x^{i+1} - x^i}^{d+1}
.
\end{equation}

Each difference in \eqref{eqn:soln_chng_vars_diff} shares a single upper bound that is a linear combination of the right hand sides of 
\eqref{eqn:1st_ordr_opt_cost_diff_r_bnd} and \eqref{eqn:can_basis_id_pert_r_bnd}.
The first difference is bounded by
\begin{equation*}
\begin{aligned}
\Bignorm{\pderiv{\tildepi}{}{\xi}(0) - \pderiv{\hatpi^{i+1}}{}{\hat{\xi}}(0) }
&=
\Bignorm{\Bigbracks{\pderiv{\pi}{}{x}(x^{i+1}) - \pderiv{\pi^{i+1}}{}{x}(x^{i+1}) + \pderiv{\pi}{}{x}(x^{i+1})(V \hat{V}^T - I)} \hat{V}} \\
&\le
\Bigpars{T' + B \sup_{\closedcompdomwoalb} \Bignorm{\pderiv{\pi}{}{x}(x)}}\norm{x^{i+1} - x^i}^{d+1}
\end{aligned}
\end{equation*}
and the second difference obeys
\begin{equation*}
\norm{\tildekappa(0) - \hatkappa^{i+1}(0) }
\le
\Bigpars{T' \sup_{\closedcompdomwoalb} \frac{\norm{g(x)}}{r(x)}} \norm{x^{i+1}-x^i}^{d+1}.
\end{equation*}
By assumption, the distance between $x^i$ and $x^{i+1}$ is less than the maximum consecutive patch point distance $h$ of Corollary \ref{cor:phi1_lte_bound} , so
\begin{equation}
\sup_{\substack{\closedcompdomwoalb \\ \norm{x^{i+1}-x^i} \le h}}
  \bigpars{ \norm{\dot{x}}( \norm{\dot{x}} - \norm{\dot{x} - \dot{x}^{i+1}} )}^{-1} < \infty
\end{equation}
and the third difference in \eqref{eqn:soln_chng_vars_diff} is bounded by
\begin{align*}
\Bigabs{\frac{1}{\norm{\dot{x}}} - \frac{1}{\norm{\dot{x}^{i+1}}}}
&\le
\Bigpars{
T' \sup_{\substack{\closedcompdomwoalb \\ \norm{x^{i+1}-x^i} \le h}}
\frac{\norm{g(x)}^2}{r(x) \bigpars{\norm{\dot{x}}( \norm{\dot{x}} - \norm{\dot{x} - \dot{x}^{i+1}} )}}
}
\norm{x^{i+1} - x^i}^{d+1}
\end{align*}

Our next task is to demonstrate the contribution from the second order non characteristic partial derivatives to \eqref{eqn:rcvrd_opt_cost_prtl_deriv_diff_just_prtls_bnd} is $\mathcal{O}(\norm{x^{i+1}-x^i}^d)$.
Let $\hat{\pi}^{i+1}(\hat{\xi}) \equiv \bar{\pi}(x^{i+1} + \hat{V} \hat{\xi})$ where $\bar{\pi}$ denotes the Taylor polynomial of the optimal cost centered at $x=x^i$.
If neither $j_1$ nor $j_2$ are one, then by Corollary \ref{cor:prtl_drv_chng_crds_diff2},
the difference between the non characteristic partial derivatives is bounded by
\begin{multline}
\label{eqn:nonchar_prtl_deriv_diff_bound}
\Bigabs{
\ppderiv{\tildepi}{2}{\xi_{j_1}}{}{\xi_{j_2}}{}(0) - 
\ppderiv{\hatpi^{i+1}}{2}{\hat{\xi}_{j_1}}{}{\hat{\xi}_{j_2}}{}(0)
 }
\le
\bigpars{2^2-1}\Bignorm{\pderiv{}{}{x} \otimes \pderiv{}{}{x} [\pi(x)]_{x=x^{i+1}}} \norm{V^T \hat{V} - I}
\\
+ \bigpars{2^2-1}\Bignorm{\pderiv{}{}{x} \otimes \pderiv{}{}{x} [\pi(x) - \bar{\pi}(x)]_{x=x^{i+1}}} 
\end{multline}
The first term on the right hand side of the inequality \eqref{eqn:nonchar_prtl_deriv_diff_bound} is $\mathcal{O}(\norm{x^{i+1}-x^i}^{d+1})$ by the assumption that $\pi$ is smooth on the compact computational domain, and \eqref{eqn:can_basis_id_pert_r_bnd}.
Since $\bar{\pi}$ is the Taylor polynomial of $\pi$ centered at $x^i$, then each entry in $\tpderiv{}{}{x} \otimes \tpderiv{}{}{x} [\pi(x) - \bar{\pi}(x)]_{x=x^{i+1}}$ is bounded by $R_T \norm{x^{i+1} - x^i}^d$, so the second term on the right hand side of the inequality in \eqref{eqn:nonchar_prtl_deriv_diff_bound} is $\mathcal{O}(\norm{x^{i+1}-x^i}^d)$.

We can conclude that the bound in \eqref{eqn:rcvrd_opt_cost_prtl_deriv_diff_just_prtls_bnd} holds for $k=2$, so the conclusion of the theorem holds in the case of the second order partial derivatives of the optimal cost. 

By the formula for $\tpderiv{\hatkappa^{i+1}}{}{\hat{\xi}_j}(0)$ in \eqref{eqn:dkappa_dxi_cmptd} and the same reasoning we just employed, we can conclude that the difference between $\tpderiv{\hatkappa^{i+1}}{}{\hat{\xi}_j}(0)$ and $\tpderiv{\tildekappa}{}{\hat{\xi}_j}(0)$ is $\mathcal{O}(\norm{x^{i+1}-x^i}^d)$.

The proof that the higher order partial derivatives satisfy the inequality in \eqref{eqn:rcvrd_opt_cost_prtl_deriv_diff_just_prtls_bnd} is analogous to the second order case.
We outline the proof of the third order case, which is representative of the higher order cases.
The computed third order characteristic partial derivatives with respect to the change of variables of the optimal cost is calculated from the formula in \S \ref{sec:deriv_formulas},
and the exact characteristic partial derivatives have an analogous formula.
As in the case of the second order characteristic partial derivatives of the optimal cost, the bound on the difference between the exact and computed third order characteristic partial derivatives depends linearly on two types of differences.
The first type of difference is the difference between the partial derivatives of problem data functions with respect to the two changes of variables, and each one of these differences is $\mathcal{O}(\norm{x^{i+1} - x^i}^{d-1})$.
The other type of difference is the difference between the exact and computed partial derivatives with respect to the two changes of variables of the solution.
These differences are
\begin{align}
\begin{gathered}
\frac{1}{\norm{\dot{x}}} - \frac{1}{\norm{\dot{x}^{i+1}}}
\\
\pderiv{\tildepi}{}{\xi}(0) - \pderiv{\hatpi^{i+1}}{}{\hat{\xi}}(0)
\\
\ppderiv{\tildepi}{2}{\xi_j}{}{\xi}{}(0) - \ppderiv{\hatpi^{i+1}}{2}{\hat{\xi}_j}{}{\hat{\xi}}{}(0)
\end{gathered}
& &
\begin{gathered}
\tildekappa(0) - \hatkappa^{i+1}(0)
\\
\pderiv{\tildekappa}{}{\xi_j}(0) - \pderiv{\hatkappa^{i+1}}{}{\hat{\xi}_j}(0)
\end{gathered}
\end{align}
Each one of these differences was treated in the second order case and is $\mathcal{O}(\norm{x^{i+1}-x^i}^{d-1})$.

In the case of bounding the non characteristic partial derivatives, if none of the coordinate indices $j_1$, $j_2$, nor $j_3$ are one, then we may apply Corollary \ref{cor:prtl_drv_chng_crds_diff2} to the difference between the non characteristic partial derivatives
\begin{multline}
\label{eqn:nonchar_prtl_deriv_diff_bound3}
\Bigabs{
\pppderiv{\tildepi}{3}{\xi_{j_1}}{}{\xi_{j_2}}{}{\xi_{j_3}}{}(0) - 
\pppderiv{\hatpi^{i+1}}{2}{\hat{\xi}_{j_1}}{}{\hat{\xi}_{j_2}}{}{\hat{\xi}_{j_3}}{}(0)
 }
 \\
\begin{aligned}
&\le
\bigpars{2^3-1}\Bignorm{\pderiv{}{}{x} \otimes \pderiv{}{}{x} \otimes \pderiv{}{}{x} [\pi(x)]_{x=x^{i+1}}} \norm{V^T \hat{V} - I}
\\
&\phantom{\le}
+ \bigpars{2^3-1}\Bignorm{\pderiv{}{}{x} \otimes \pderiv{}{}{x} \otimes \pderiv{}{}{x} [\pi(x) - \bar{\pi}(x)]_{x=x^{i+1}}} 
\end{aligned}
\end{multline}
The first term on the right hand side of the inequality in \eqref{eqn:nonchar_prtl_deriv_diff_bound3} is $\mathcal{O}(\norm{x^{i+1}-x^i}^{d+1})$ by the assumption that $\pi$ is smooth on the compact computational domain, and \eqref{eqn:can_basis_id_pert_r_bnd}.
Since $\bar{\pi}$ is the Taylor polynomial of $\pi$ centered at $x^i$, then each entry in $\tpderiv{}{}{x} \otimes \tpderiv{}{}{x} \otimes \tpderiv{}{}{x}[\pi(x) - \bar{\pi}(x)]_{x=x^{i+1}}$ is bounded by $R_T \norm{x^{i+1} - x^i}^{d-1}$, so the second term on the right hand side of the inequality in \eqref{eqn:nonchar_prtl_deriv_diff_bound3} is $\mathcal{O}(\norm{x^{i+1}-x^i}^{d-1})$.

\end{proof}

\section{Lipschitz condition for second and higher order partial derivatives} \label{sec:lip_k}
In this section we prove that $\phi_j$ satisfies the Lipschitz condition when $j \ge 2$ under the conditions on the optimal control problem in \S \ref{sec:drvtn:opt_cntrl_ass}, most importantly that the optimal cost is a strict Lyupanov function.
\begin{theorem}
\label{thm:phi_2_is_lipschitz} 
Under the conditions on the optimal control problem in \S \ref{sec:drvtn:opt_cntrl_ass}, there exists a maximum step size $h_L>0$, maximum multiplier $\Mmax>0$, and Lipschitz constant $L_2$ such that for all $x^i$ and $x^{i+1}$ in $\closedcompdomwoalb$ within a distance of $h_L$ of each other, and for all $M \le \Mmax$,
\begin{gather}
\label{eqn:lip2_main_diff_bound}
\norm{\phi_k(C^i) - \phi_k(\hat{C}^i)} \le L_2 M \norm{x^{i+1} - x^i}^{d+2-k}
\text{ for } 2 \le k \le d+1
\intertext{whenever} \notag
\norm{C_j^i - \hat{C}_j^i} \le M \norm{x^{i+1} - x^i}^{d+2-j}
\text{ for }
0 \le j \le d+1
\end{gather}
where $C^i$ is the vector of exact partial derivatives up to order $d+1$ of the optimal cost at $x^i$.
\end{theorem}

The proof mirrors the proof of the local truncation error result of Theorem \ref{thm:opt_cost_local_trunc_err_higer_order}.
We will delay the proof of the theorem until after we have introduced some notation and established some relevant bounds.
We denote the two approximate optimal directions at $x^{i+1}$ computed from $C^i$ and $\hat{C}^i$ as
\begin{align*}
\xdot{c}{i+1}
&\equiv
f(x^{i+1}) - \frac{\phi_1(C^i) \cdot g(x^{i+1})}{r(x^{i+1})} g(x^{i+1})
\intertext{and}
\xdot{\hat{c}}{i+1}
&\equiv
f(x^{i+1}) - \frac{\phi_1(\hat{C}^i) \cdot g(x^{i+1})}{r(x^{i+1})} g(x^{i+1})
\end{align*}

As in the case of the local truncation error result, we will work with two changes of variables that we choose to be convenient.
These changes of coordinates are
\begin{align}
\label{eqn:lip_tilde_chng_var}
x = x^{i+1} + \tilde{V} \tildexi
& &
\tilde{V}^1 = \frac{\xdot{c}{i+1}}{\norm{\xdot{c}{i+1}}}
\intertext{and}
\label{eqn:lip_hat_chng_var}
x = x^{i+1} + \hat{V} \hatxi
& &
\hat{V}^1 = \frac{\xdot{\hat{c}}{i+1}}{\norm{\xdot{\hat{c}}{i+1}}}
\end{align}
We denote the approximate optimal cost and optimal control polynomials under the two changes of variables as
\begin{align*}
\tildepi^{i+1}(\tildexi)
& &
\hatpi^{i+1}(\hatxi)
\\
\tildekappa^{i+1}(\tildexi)
& &
\hatkappa^{i+1}(\hatxi)
\end{align*}
and we denote the problem data functions under the two changes of variables as
\begin{align}
\label{eqn:lip_prob_data_chng_vars}
\begin{aligned}
\tilde{f}(\tildexi) &\equiv \tilde{V}^T f(x^{i+1} + \tilde{V} \tildexi) \\
\tilde{g}(\tildexi) &\equiv \tilde{V}^T g(x^{i+1} + \tilde{V} \tildexi) \\
\tilde{q}(\tildexi) &\equiv q(x^{i+1} + \tilde{V} \tildexi) \\
\tilde{r}(\tildexi) &\equiv r(x^{i+1} + \tilde{V} \tildexi) 
\end{aligned}
& &
\begin{aligned}
\hat{f}(\hat{\xi}) &\equiv \hat{V}^T f(x^{i+1} + \hat{V} \hat{\xi}) \\
\hat{g}(\hat{\xi}) &\equiv \hat{V}^T g(x^{i+1} + \hat{V} \hat{\xi}) \\
\hat{q}(\hat{\xi}) &\equiv q(x^{i+1} + \hat{V} \hat{\xi}) \\
\hat{r}(\hat{\xi}) &\equiv r(x^{i+1} + \hat{V} \hat{\xi}) 
\end{aligned}
\end{align}

In the course of proving Theorem \ref{thm:phi_2_is_lipschitz}, we will need a guarantee that certain quantities depending on the computed solution are finite.
Lemma \ref{lem:phi_2_lip_basic_infs} provides this guarantee.
\begin{lemma}
\label{lem:phi_2_lip_basic_infs}
If the optimal control problem satisfies the assumptions of \S \ref{sec:drvtn:opt_cntrl_ass}, then there exists an $\Mmax > 0$ and a maximum consecutive patch point distance $h_2>0$ such that, for all $M \le \Mmax$ and for all $x^i$ and $x^{i+1}$ in $\closedcompdomwoalb$, 
\begin{gather}
\begin{gathered}
 \inf \norm{\xdot{c}{i+1}} > 0
\\
 \inf \bigpars{\norm{\xdot{c}{i+1}} - \norm{\xdot{c}{i+1} - \xdot{\hat{c}}{i+1}} } > 0
\end{gathered}
\intertext{whenever}
\notag
\norm{x^{i+1} - x^i} \le h_2
\text{ and } 
\norm{C_j^i - \hat{C}_j^i} \le M \norm{x^{i+1} - x^i}^{d+2-j}
\text{ for } 0 \le j \le d+1
\end{gather}
where the infimums are taken over all $M \le \Mmax$ and all $x^{i+1}$ in $\closedcompdomwoalb$ such that $x^{i+1} - x^i \le h_2$. 
$C^i$ holds the exact partial derivatives of $\pi$ at $x^i \in \closedcompdomwoalb$.

\begin{proof}
The exact optimal direction at $x^{i+1}$ is $\dot{x} = f(x^{i+1}) + g(x^{i+1}) \kappa(x^{i+1})$.
$C^i$ holds the exact partial derivatives of the optimal cost at $x^i$ up to order $d+1$, 
so by Theorem \ref{thm:phi1_lte_bound}, 
\begin{equation}
\label{eqn:phi_2_lip_basic_infs_lim1}
\lim_{x^{i+1} \rightarrow x^i} \bigpars{(f(x^{i+1}) + g(x^{i+1}) \kappa(x^{i+1})) - \xdot{c}{i+1} }
= 
0
.
\end{equation}
The assumption that $\pi$ is a strict Lyapunov function guarantees that the infimum over $\closedcompdomwoalb$ of $\norm{\dot{x}}$ is strictly greater than zero. 
The limit in \eqref{eqn:phi_2_lip_basic_infs_lim1} is uniform, so for all $\epsilon_1$ in $0 < \epsilon_1 < 1$, there exists a $\delta_1 > 0$ independent of $x^i$ and $x^{i+1}$ such that
\begin{equation*}
\norm{f(x^{i+1}) + g(x^{i+1}) \kappa(x^{i+1}) - \xdot{c}{i+1} }
\le 
\epsilon_1 \inf_{\closedcompdomwoalb} \norm{\dot{x}}
\end{equation*}
Therefore
\begin{align*}
\norm{\xdot{c}{i+1}}
&\ge
\norm{f(x^{i+1}) + g(x^{i+1}) \kappa(x^{i+1})} 
  - \norm{f(x^{i+1}) + g(x^{i+1}) \kappa(x^{i+1}) -\xdot{c}{i+1} } \\
&\ge
(1- \epsilon_1 )\inf_{\closedcompdomwoalb} \norm{\dot{x}} \\
&>
0
\end{align*}
Thus we have established the first inequality of the lemma statement.

We will assume for the rest of the proof that $M \le \Mmax$.
By Theorem \ref{thm:phi_1_is_lipschitz}, if $\norm{C_j^i - \hat{C}_j^i} \le M \norm{x^{i+1} - x^i}^{d+2-j}$ for $0 \le j \le d+1$, then
\begin{equation}
\label{eqn:phi_2_lip_basic_infs_lim2}
\lim_{x^{i+1} \rightarrow x^i} \xdot{c}{i+1} - \xdot{\hat{c}}{i+1} 
= 
0
.
\end{equation}
The limit in \eqref{eqn:phi_2_lip_basic_infs_lim2} is uniform, so for all $\epsilon_2$ in $0 < \epsilon_2 < 1$, there exists a $\delta_2 > 0$ independent of $x^i$ and $x^{i+1}$ such that
\begin{gather*}
\norm{\xdot{c}{i+1} - \xdot{\hat{c}}{i+1} }
\le 
\epsilon_2 \inf_{\closedcompdomwoalb} \norm{\dot{x}}
\intertext{whenever}
\norm{x^{i+1} - x^{i}} \le \delta_2
\text{ and }
\norm{C_j^i - \hat{C}_j^i } \le M \norm{x^{i+1}-x^i}^{d+2-j}
\text{ for } 0 \le j \le d+1
.
\end{gather*}
Therefore, if $\norm{\xdot{c}{i+1}} \ge (1-\epsilon_1) \inf_{\closedcompdomwoalb}\norm{\dot{x}}$ where $0 < \epsilon_1 < 1$, then
\begin{gather*}
\begin{aligned}
\norm{\xdot{c}{i+1}} - \norm{\xdot{c}{i+1} - \xdot{\hat{c}}{i+1}}
&\ge
(1- \epsilon_1 )\inf_{\closedcompdomwoalb} \norm{\dot{x}}
- \epsilon_2 (1- \epsilon_1 )\inf_{\closedcompdomwoalb} \norm{\dot{x}} \\
&= (1- \epsilon_1 ) (1- \epsilon_2 ) \inf_{\closedcompdomwoalb} \norm{\dot{x}} \\
&>
0
\end{aligned}
\intertext{whenever}
\norm{x^{i+1} - x^{i}} \le \delta_1
\text{ and }
\norm{C_j^i - \hat{C}_j^i } \le M \norm{x^{i+1}-x^i}^{d+2-j}
\text{ for } 0 \le j \le d+1
\end{gather*}
so the second infimum of the lemma statement has been established.
We take $h_2$ to be the smaller of $\delta_1$ and $\delta_2$.
\end{proof}
\end{lemma}

Corollary \ref{cor:lip_nrmlzd_opt_drctn_bnd} provides a guarantee that any term depending on the difference between the normalized optimal directions at $x^{i+1}$ computed from $C^i$ and $\hat{C}^i$ does not affect the order of the main bound \eqref{eqn:lip2_main_diff_bound} in Theorem \ref{thm:phi_2_is_lipschitz}.
\begin{corollary}
\label{cor:lip_nrmlzd_opt_drctn_bnd}
If the optimal cost $\pi$ is a strict Lyapunov function on the computational domain $\compdom$, then there exists a Lipschitz constant $L'$, maximum multiplier $\Mmax > 0$ and a maximum consecutive patch point distance $h_2>0$ such that for all $x^i$ and $x^{i+1}$ in $\closedcompdomwoalb$ within a distance of $h_2$ of each other, and for all $M \le \Mmax$,
\begin{gather*}
\Biggnorm{
\frac{\xdot{c}{i+1}}{\norm{\xdot{c}{i+1}}} - \frac{\xdot{\hat{c}}{i+1}}{\norm{\xdot{\hat{c}}{i+1}}}
}
\le
L' M \norm{x^{i+1} - x^i}^{d+1}
\intertext{whenever}
\notag
\norm{C_j^i - \hat{C}_j^i} \le M \norm{x^{i+1} - x^i}^{d+2-j}
\text{ for } 0 \le j \le d+1
\end{gather*}
where $C^i$ holds the exact partial derivatives of $\pi$ at $x^i \in \closedcompdomwoalb$.
\begin{proof}
The bound in the corollary statement follows from the Lemma \ref{lem:phi_2_lip_basic_infs} and the inequality
\begin{equation*}
\Biggnorm{
\frac{\xdot{c}{i+1}}{\norm{\xdot{c}{i+1}}} - \frac{\xdot{\hat{c}}{i+1}}{\norm{\xdot{\hat{c}}{i+1}}}
}
\le
\frac{ \norm{\xdot{c}{i+1} - \xdot{\hat{c}}{i+1}} } { \norm{\xdot{c}{i+1}} - \norm{\xdot{c}{i+1} - \xdot{\hat{c}}{i+1}} }
.
\end{equation*}

\end{proof}
\end{corollary}

As in the case of proving the local truncation error result of Theorem \ref{thm:opt_cost_local_trunc_err_higer_order}, we will pick the changes of coordinates in \eqref{eqn:lip_tilde_chng_var} and \eqref{eqn:lip_hat_chng_var} so 
that the orthogonal matrices $\tilde{V}$ and $\hat{V}$ are almost inverses of each other.
This is the content of Corollary \ref{cor:lip_cnvtn_mat_id_diff_bnd}.
\begin{corollary} 
\label{cor:lip_cnvtn_mat_id_diff_bnd}
If the optimal cost $\pi$ is a strict Lyapunov function on the computational domain $\compdom$, then there exists a Lipschitz constant $L'$, maximum multiplier $\Mmax > 0$ and a maximum consecutive patch point distance $h_2>0$, and orthogonal matrices $\hat{V}$ and $\tilde{V}$ such that their first columns are
\begin{align*}
\tilde{V}^1 = \frac{\xdot{c}{i+1}}{\norm{\xdot{c}{i+1}}}
& &
\hat{V}^1 = \frac{\xdot{\hat{c}}{i+1}}{\norm{\xdot{\hat{c}}{i+1}}}
\end{align*}
so that for all $x^i$ and $x^{i+1}$ in $\closedcompdomwoalb$ within a distance of $h_2$ of each other and all $M \le \Mmax$,
\begin{gather*}
\norm{\hat{V}^T \tilde{V} - I }
\le
L' M \norm{x^{i+1} - x^i}^{d+1}
\intertext{whenever}
\notag
\norm{C_j^i - \hat{C}_j^i} \le M \norm{x^{i+1} - x^i}^{d+2-j}
\text{ for } 0 \le j \le d+1
\end{gather*}
where $C^i$ holds the exact partial derivatives of $\pi$ at $x^i \in \closedcompdomwoalb$, then
\begin{proof}
For a sufficiently small maximum consecutive patch point distance $h_2$,  Lemma \ref{lem:phi_2_lip_basic_infs} guarantees $\xdot{c}{i+1}$ and $\xdot{\hat{c}}{i+1}$ are close enough so that Lemma \ref{lem:canonical_basis_exists}, and Corollaries \ref{cor:canonical_mat_prod_iden_diff_ip_bnd} and \ref{cor:canonical_mat_prod_iden_diff_v_bnd} imply the bound on $\hat{V}^T \tilde{V}-I$ of the corollary statement.
\end{proof}
\end{corollary}

As in the proof of Theorem \ref{thm:opt_cost_local_trunc_err_higer_order}, we bound the differences between the computed partial derivatives of the optimal cost with respect to the two changes of coordinates, and then show that the order of this bound does not change after recovering the two sets of partial derivatives with respect to the original state space coordinates and compare their differences.
Lemma \ref{lem:lip_nonchar_prtl_deriv_diff_bnd} provides a bound on the difference between the characteristic partial derivatives, which are computed by inheritance.
\begin{lemma}
\label{lem:lip_nonchar_prtl_deriv_diff_bnd}
Suppose the orthogonal matrices $\tilde{V}$ and $\hat{V}$ from the changes of coordinates in \eqref{eqn:lip_tilde_chng_var} and \eqref{eqn:lip_hat_chng_var} are the ones from Corollary \ref{cor:lip_cnvtn_mat_id_diff_bnd}. 
If the optimal cost $\pi$ is a strict Lyapunov function on the computational domain $\compdom$, then there exists a Lipschitz constant $\Linh' < \infty$,  $\Mmax > 0$ and a maximum consecutive patch point distance $h_2>0$ such that the partial derivatives computed by inheritance satisfy
\begin{gather*}
\Bigabs{
\pderivelip{\tildepi^{i+1}}{k}{\tildexi_{j_1}}{}{\tildexi_{j_k}}{}(0)
-
\pderivelip{\hatpi^{i+1}}{k}{\hatxi_{j_1}}{}{\hatxi_{j_k}}{}(0)
}
\le
\Linh' \norm{x^{i+1}-x^i}^{d+2-k}
\intertext{whenever}
\notag
\norm{x^{i+1} - x^i} \le h_2
\text{ and } 
\norm{C_j^i - \hat{C}_j^i} \le M \norm{x^{i+1} - x^i}^{d+2-j}
\text{ for } 0 \le j \le d+1
\end{gather*}

\begin{proof}
$P$ is linear in its first argument, and $\kronderivelip{x}{x}$ is linear, so by Lemma \ref{lem:poly_lip}
\begin{multline*}
\Bignorm{
\kronderivelip{x}{x} [P(C^i,x-x^i)]_{x=x^{i+1}}
\\
- \kronderivelip{x}{x} [P(\hat{C}^i,x-x^i)]_{x=x^{i+1}} 
} \\
\begin{aligned}
&=
\Bignorm{ \kronderivelip{x}{x} [P(C^i-\hat{C}^i,x)]_{x=x^{i+1}-x^i} } \\
&\le
n^{k/2} \Linh M \norm{x^{i+1} - x^i}^{d+2-k}
\end{aligned}
\end{multline*}
where $\kronderivelip{x}{x}$ is repeated $k$ times.

It follows from the local truncation error result of Theorems \ref{thm:phi1_lte_bound} and \ref{thm:opt_cost_local_trunc_err_higer_order} that there is some $T < \infty$ such that 
\begin{multline*}
\Bignorm{\kronderivelip{x}{x} [P(C^i,x-x^i) - \pi(x) + \pi(x)]_{x=x^{i+1}}} \\
\le  T \norm{x^{i+1}-x^i}^{d+2-k} + \sup_{\closedcompdomwoalb} \Bignorm{\kronderivelip{x}{x} [\pi(x)]_{x=x^{i+1}}}
\end{multline*}
where the supremum is finite since $\pi$ is smooth by assumption on the compact computational domain.
It follows from Corollary \ref{cor:lip_cnvtn_mat_id_diff_bnd} and Lemma \ref{lem:kron_prd_id_diff} that
\begin{equation*}
\Bignorm{\kronprodelip{\tilde{V} \hat{V}^T}{\tilde{V} \hat{V}^T} - I} \le (2^k-1) L_1' M \norm{x^{i+1} -x^i}^{d+1}
\end{equation*}
The conclusion of the lemma follows by applying the previous three inequalities to the following bound on the difference between the two computed characteristic partial derivatives of the optimal cost.
\begin{multline*}
\Bigabs{
\pderivelip{\tildepi^{i+1}}{k}{\tildexi_{j_1}}{}{\tildexi_{j_k}}{}(0)
-
\pderivelip{\hatpi^{i+1}}{k}{\hatxi_{j_1}}{}{\hatxi_{j_k}}{}(0)
} \\
\begin{aligned}
&=
\Bigabs{
\kronderivelip{x}{x} [P(C^i,x)]_{x=x^{i+1}} \kronprodelip{\tilde{V}^{j_1}}{\tilde{V}^{j_k}}
\\
&\phantom{= \|}
- 
\kronderivelip{x}{x} [P(\hat{C}^i,x)]_{x=x^{i+1}} \kronprodelip{\hat{V}^{j_1}}{\hat{V}^{j_k}} 
} \\
&\le
\Bignorm{
\kronderivelip{x}{x} [P(C^i - \hat{C}^i,x)]_{x=x^{i+1}} }\\
&\phantom{= \|}
+\Bignorm{\kronderivelip{x}{x} [P(C^i,x)]_{x=x^{i+1}} \bigpars{\kronprodelip{\tilde{V} \hat{V}^T}{\tilde{V} \hat{V}^T} - I}
} \\
\end{aligned}
\end{multline*}
\end{proof}
\end{lemma}

We are now in a position to prove Theorem \ref{thm:phi_2_is_lipschitz}.
\begin{proof}[Proof of Theorem \ref{thm:phi_2_is_lipschitz}]
The proof of the theorem mirrors the proof of the local truncation error result in  Theorem \ref{thm:opt_cost_local_trunc_err_higer_order}. 
We may assume that the orthogonal matrices $\tilde{V}$ and $\hat{V}$ appearing in the changes of coordinates of \eqref{eqn:lip_tilde_chng_var} and \eqref{eqn:lip_hat_chng_var} are convenient, that is they are taken from Corollary \ref{cor:lip_cnvtn_mat_id_diff_bnd}.
Starting with the computed second order partial derivatives, we will show that the main bound \eqref{eqn:lip2_main_diff_bound} of the theorem holds, by bounding the individual entries of $\phi_k(C^i)-\phi_k(\hat{C}^i)$.
We do this by considering the cases of the computed characteristic and non characteristic partial derivatives of the optimal cost independently, and then proceed by induction for the higher order partial derivatives.

We denote the computed optimal cost polynomials computed from the coefficient sets $C^i$ and $\hat{C}^i$ and  centered at $x^{i+1}$  as $\pi_{c}^{i+1}$ and $\pi_{\hat{c}}^{i+1}$.
They are given by the formulas
\begin{align*}
\pi_{c}^{i+1}(x) \equiv P(\phi(C^i), x-x^{i+1})
& &
\pi_{\hat{c}}^{i+1}(x) \equiv P(\phi(\hat{C}^i), x-x^{i+1})
\end{align*}

At $x^{i+1}$, the difference between a $\ith{p}$ order partial derivative of the two computed optimal cost polynomials is bounded by
\begin{multline}
\label{eqn:lip2_pth_rcvrd_prtl_diff_bnd}
\Bigabs{
\pderivelip{\pi_{c}^{i+1}}{p}{x_{j_1}}{}{x_{j_p}}{}(x^{i+1})
-
\pderivelip{\pi_{\hat{c}}^{i+1}}{p}{x_{j_1}}{}{x_{j_p}}{}(x^{i+1})
}\\
\begin{aligned}
&\le
\Bignorm{
\kronderivelip{\tildexi}{\tildexi} [\tildepi^{i+1}(\tildexi)]_{\tildexi=0} 
}
\,
\Bignorm{
\kronprodelip{\tilde{V}\hat{V}^T}{\tilde{V}\hat{V}^T} - I
}
\\
&\phantom{\le \|}
+\Bignorm{
\kronderivelip{\tildexi}{\tildexi} [\tildepi^{i+1}(\tildexi)]_{\tildexi=0} 
-
\kronderivelip{\hatxi}{\hatxi} [\hatpi^{i+1}(\hatxi)]_{\hatxi=0} 
}
\end{aligned}
\end{multline}
The following two inequalities will imply that the first term on the right hand side of \eqref{eqn:lip2_pth_rcvrd_prtl_diff_bnd} has an upper bound that is proportional to $M \norm{x^{i+1} - x^i}^{d+2-k}$ for any $k \ge 2$.
Since $\tildepi^{i+1}$ is computed from the set of coefficients that are exact at the previous patch point, it follows from Theorem \ref{thm:opt_cost_local_trunc_err_higer_order} that there is some $T < \infty$ such that
\begin{equation}
\label{eqn:lip_kron_deriv_c_opt_cost_bnd}
\Bignorm{
\kronderivelip{\tildexi}{\tildexi} [\tildepi^{i+1}(\tildexi)]_{\tildexi=0}
}
\le 
\sup_{\closedcompdomwoalb}\Bignorm{\kronderivelip{x}{x} [\pi(x)]}
+ T \norm{x^{i+1}-x^i}^{d+2-k}
.
\end{equation}
By Corollary \ref{cor:lip_cnvtn_mat_id_diff_bnd} and Lemma \ref{lem:kron_prd_id_diff},
\begin{equation}
\label{eqn:lip_kron_tildev_hat_v_id_diff_bnd}
\Bignorm{
\kronprodelip{\tilde{V}\hat{V}^T} {\tilde{V}\hat{V}^T} - I
}
\le
(2^k-1) L_1' M \norm{x^{i+1}-x^i}^{d+1}
.
\end{equation}

We now prove that the second term on the right hand side of \eqref{eqn:lip2_pth_rcvrd_prtl_diff_bnd} has an upper bound that is proportional to $M \norm{x^{i+1}-x^i}^{d+1}$.
Each entry in
\begin{equation*}
\kronderivelip{\tildexi}{\tildexi} [\tildepi^{i+1}(\tildexi)]_{\tildexi=0} 
-
\kronderivelip{\hatxi}{\hatxi} [\hatpi^{i+1}(\hatxi)]_{\hatxi=0} 
\end{equation*}
is either the difference between characteristic or non characteristic partial derivatives.
In the case of the non characteristic partial derivatives, Lemma \ref{lem:lip_nonchar_prtl_deriv_diff_bnd} guarantees that the difference is at most
\begin{equation*}
\Bigabs{
\pderivelip{\tildepi^{i+1}}{k}{\tildexi_{j_1}}{}{\tildexi_{j_k}}{}(0)
-
\pderivelip{\hatpi^{i+1}}{k}{\hatxi_{j_1}}{}{\hatxi_{j_k}}{}(0)
}
\le
\Linh' M \norm{x^{i+1}-x^i}^{d+2-k}
\end{equation*}
so the non characteristic partials make a contribution to the right hand side of \eqref{eqn:lip2_pth_rcvrd_prtl_diff_bnd} that is at most $\Linh' M \norm{x^{i+1}-x^i}^{d+2-k}$.
We now consider the contribution to \eqref{eqn:lip2_pth_rcvrd_prtl_diff_bnd} from the characteristic partial derivatives.
We start with the second order partial derivatives and proceed by induction.
At $x^{i+1}$, a characteristic partial derivative of the optimal cost is calculated from the coefficients held in $C^i$ from the formula
\begin{multline}
\label{eqn:lip_2nd_ord_c_opt_cost_prtl}
\ppderiv{\tilde{\pi}}{2}{\xi_j}{}{\xi_1}{}(0)
=
-\frac{1}{\norm{\xdot{c}{i+1}}}
\Biggbracks{
\pderiv{\tildepi^{i+1}}{}{\tildexi}(0)
\Bigpars{
\pderiv{\tildef}{}{\tildexi_j}(0) + \pderiv{\tildeg}{}{\tildexi_j}(0) \tildekappa^{i+1}(0)
\\
+
\pderiv{\tildeq}{}{\tildexi_j}(0) + \frac{1}{2} \pderiv{\tilder}{}{\tildexi_j}(0) \tildekappa^{i+1}(0)^2
}
}
\end{multline}
and the corresponding characteristic partial derivative computed from the coefficients in $\hat{C}^i$ are computed by an analogous formula.
The task of showing that the difference between the second order computed characteristic partial derivatives makes no more than an $\mathcal{O}(M \norm{x^{i+1}-x^i}^d)$ contribution to the right hand side of \eqref{eqn:lip2_pth_rcvrd_prtl_diff_bnd} reduces to showing that the difference between the partial derivatives of the problem data functions with respect to the two changes of variables
\begin{align}
\label{eqn:lip_prob_data_chng_vars_diff}
\begin{aligned}
\pderiv{\tildef}{}{\tildexi}(0) &- \pderiv{\hatf}{}{\hat{\xi}}(0)
\\
\pderiv{\tildeq}{}{\tildexi}(0) &- \pderiv{\hatq}{}{\hat{\xi}}(0)
\end{aligned}
& &
\begin{aligned}
\pderiv{\tildeg}{}{\tildexi}(0) &- \pderiv{\hatg}{}{\hat{\xi}}(0) 
\\
\pderiv{\tilder}{}{\tildexi}(0) &- \pderiv{\hatr}{}{\hat{\xi}}(0) 
\end{aligned}
\end{align}
and the difference between each of the computed quantities
\begin{align}
\label{eqn:lip_soln_chng_vars_diff}
\pderiv{\tildepi^{i+1}}{}{\tildexi}(0) - \pderiv{\hatpi^{i+1}}{}{\hat{\xi}}(0)
& &
\tildekappa^{i+1}(0) - \hatkappa^{i+1}(0)
& &
\frac{1}{\norm{\xdot{c}{i+1}}} - \frac{1}{\norm{\xdot{\hat{c}}{i+1}}}
\end{align}
are all proportional to $M \norm{x^{i+1}-x^i}^{d}$ or higher.
In fact, they are all proportional to $M \norm{x^{i+1}-x^i}^{d+1}$.
Lemma \ref{lem:prtl_drv_chng_crds_diff} bounds the difference between the partial derivatives of the problem data functions in terms of $\hat{V}^T \tilde{V} - I$.
Corollary \ref{cor:lip_cnvtn_mat_id_diff_bnd} then implies that for some Lipschitz constant $L < \infty$
\begin{gather*}
\Bignorm{
\pderiv{\tildeg}{}{\tildexi}(0) - \pderiv{\hatg}{}{\hat{\xi}}(0) 
}
\le
L M \norm{x^{i+1} - x^i}^{d+1}
\end{gather*}
and and the same bound bounds the other differences in \eqref{eqn:lip_prob_data_chng_vars_diff}.

We now show that there is some Lipschitz constant $L>0$ such that each difference in \eqref{eqn:lip_soln_chng_vars_diff} is bounded by $L M \norm{x^{i+1}-x^i}^{d+1}$.
By Theorem \ref{thm:phi_1_is_lipschitz}, there exists an $L$ such that
\begin{equation*}
\Bignorm{
\pderiv{\pi_c^{i+1}}{}{x}(x^{i+1})
-
\pderiv{\pi_{\hat{c}}^{i+1}}{}{x}(x^{i+1})
}
\le
L M \norm{x^{i+1} - x^i}^{d+1}
\end{equation*}
and by the inequalities \eqref{eqn:lip_kron_deriv_c_opt_cost_bnd} and \eqref{eqn:lip_kron_tildev_hat_v_id_diff_bnd}, there is some other $L$
\begin{equation*}
\Bignorm{
\pderiv{\pi_c^{i+1}}{}{x}(x^{i+1}) ( \tilde{V} \hat{V}^T - I )
}
\le
L M \norm{x^{i+1} - x^i}^{d+1}
.
\end{equation*}
It then follows from
\begin{align*}
\Bignorm{
\pderiv{\tildepi^{i+1}}{}{\tildexi}(0)
-
\pderiv{\hatpi^{i+1}}{}{\hatxi}(0)
}
&=
\Bignorm{
\pderiv{\pi_c^{i+1}}{}{x}(x^{i+1}) \tilde{V}
-
\pderiv{\pi_{\hat{c}}^{i+1}}{}{x}(x^{i+1}) \hat{V}
} \\
&\le
\Bignorm{
\pderiv{\pi_c^{i+1}}{}{x}(x^{i+1})
-
\pderiv{\pi_{\hat{c}}^{i+1}}{}{x}(x^{i+1})
}
\\
&\phantom{\le \|}
+
\Bignorm{
\pderiv{\pi_c^{i+1}}{}{x}(x^{i+1}) ( \tilde{V} \hat{V}^T - I )
}
\end{align*}
that for some $L < \infty$,
\begin{equation}
\label{eqn:lip_1st_prtls_opt_cost_diff_bnd}
\Bignorm{
\pderiv{\tildepi^{i+1}}{}{\tildexi}(0)
-
\pderiv{\hatpi^{i+1}}{}{\hatxi}(0)
}
\le
L M \norm{x^{i+1} - x^i}^{d+1}
.
\end{equation}
The differences between the computed optimal controls is at most
\begin{align}
\begin{split}
\abs{
\tildekappa^{i+1}(0) - \hatkappa^{i+1}(0)
}
&=
\frac{1}{r(x^{i+1})}
\Bigabs{
\Bigpars{
\pderiv{\pi_c^{i+1}}{}{x}(x^{i+1}) -\pderiv{\pi_{\hat{c}}^{i+1}}{}{x}(x^{i+1})
}
g(x^{i+1})
} \\
&\le
\sup_{\closedcompdomwoalb} \frac{\norm{g(x)}}{r(x)}
\Bignorm{ \pderiv{\pi_c^{i+1}}{}{x}(x^{i+1}) -\pderiv{\pi_{\hat{c}}^{i+1}}{}{x}(x^{i+1}) } \\
&\le
\sup_{\closedcompdomwoalb} \frac{\norm{g(x)}}{r(x)} L M \norm{x^{i+1} - x^i}^{d+1}
\end{split}
\end{align}
Lemma \ref{lem:phi_2_lip_basic_infs} guarantees that for all $x^i$ and $x^{i+1}$ in $\closedcompdomwoalb$,
\begin{gather}
\inf_{\norm{x^{i+1}-x^i} \le h_2}
 \norm{\xdot{c}{i+1}}( \norm{\xdot{c}{i+1}} - \norm{\xdot{c}{i+1} - \xdot{\hat{c}}{i+1}} ) > 0
\intertext{whenever} \notag 
\norm{x^{i+1} - x^i} \le h_2
\text{ and } 
\norm{C_k^i - \hat{C}_k^i} \le M \norm{x^{i+1} - x^i}^{d+2-k}
\text{ for } 0 \le k \le d+1
\end{gather}
Therefore, for some $L$, the final difference in \eqref{eqn:lip_soln_chng_vars_diff} is bounded by
\begin{multline}
\Bigabs{\frac{1}{\norm{\xdot{c}{i+1}}} - \frac{1}{\norm{\xdot{\hat{c}}{i+1}}}}
\\
\begin{aligned}
&\le
\frac{\sup \norm{g(x)}^2/r(x)}
{\inf \bigpars{ \norm{\xdot{c}{i+1}}( \norm{\xdot{c}{i+1}} - \norm{\xdot{c}{i+1} - \xdot{\hat{c}}{i+1}} )}}
 \Bignorm{
\pderiv{\pi_c^{i+1}}{}{x}(x^{i+1}) -\pderiv{\pi_{\hat{c}}^{i+1}}{}{x}(x^{i+1})
} \\
&\le
L M \norm{x^{i+1} - x^i}^{d+1}
\end{aligned}
\end{multline}
where the supremum is taken over $\closedcompdomwoalb$, and the infimum is taken over all $x^i$ and $x^{i+1}$ in $\closedcompdomwoalb$ within a distance of $h_2$ of each other.
Their ratio is then absorbed into $L$.

The proof that the difference between every corresponding characteristic partial derivatives of order $k>2$ shares a single upper bound that is proportional to $M \norm{x^{i+1}-x^i}^{d+2-k}$ proceeds by induction and follows the same form as the $k=2$ case.

\end{proof}

\section{Numerical results}
\label{chap:nmrcl_rslts}
We tested the patchy algorithm on the following nonlinear test problem.
\begin{gather*}
\intertext{Find}
\pi(x^0) = \min_u \frac{1}{2} \int_0^{\infty} 
  \Bigl( 
     \sin^2(x_1) + \Bigl( x_2 - \frac{1}{3} x_1^3 \Bigr)^2 + u^2
  \Bigr) dt
\intertext{subject to the dynamics}
\begin{aligned}
\dot{x}_1 &= \Bigl( x_2 - \frac{1}{3} x_1^3 \Bigr) \sec(x_1) \\
\dot{x}_2 &= \Bigl( x_1^2 x_2 - \frac{1}{3} x_1^5 \Bigr) \sec(x_1) + u
\end{aligned}
\end{gather*}
For any integer $k$, the optimal dynamics have a rest point at $(x_1, x_2) = (k \pi, \tfrac{1}{3}(k \pi)^3)$, and the optimal cost is zero, and the presence of $\sec(x_1)$ makes $\dot{x}$ blow up near $x_1 = -\pi/2$ and $x_1 = \pi/2$.

Under a change of coordinates, the test problem is the linear-quadratic regulator problem
\begin{gather*}
\intertext{Find}
\pi(y^0) = \min_u \frac{1}{2} \int_0^{\infty} \bigl( y^T y + u^2 \bigr) dt
\intertext{subject to the dynamics}
\begin{aligned}
\dot{y}_1 &= y_2 \\
\dot{y}_2 &= u
\end{aligned}
\end{gather*}
and the change of variables is
\begin{align}
\label{eqn:test_prob_chng_vars}
\begin{split}
y_1
&= 
\sin(x_1)
\\
y_2 
&= 
x_2 - \frac{1}{3} x_1^3
.
\end{split}
\end{align}

We chose this nonlinear test problem because we can calculate the exact solution to the linear-quadratic regulator problem exactly by solving the associated algebraic Riccati matrix equation \cite[p. 252]{krener_navasca_patchy}.
We can then calculate the exact solution to the nonlinear test problem by applying the change of coordinates to the solution of the linear-quadratic regulator problem.
Doing so gives us the solution to the nonlinear optimal control problem
\begin{equation*}
\pi(x)
= 
\frac{\sqrt{3}}{2}
\Bigbracks{
\Bigpars{ \sin(x_1) + \bigpars{x_2 - \frac{1}{3}x_1^3}}^2 + \frac{2}{3} \bigpars{x_2 - \frac{1}{3}x_1^3}^2
}
\end{equation*}
The problem data functions $f$, $g$, $q$, and $r$ and the solution $\pi$ are all smooth on $\{ (x_1,x_2) \mspace{4 mu} \bigl\vert \mspace{4 mu} -\pi/2 < x_1 < \pi/2 , -\infty < x_2 < \infty \}$.
Furthermore, $\pi$ is a strict Lyapunov function on this set.
If we denote the change of variables \eqref{eqn:test_prob_chng_vars} $y = T(x)$, and if the matrix $\tpderiv{T}{}{x}(x)$ is invertible at $x$, then $\tpderiv{\pi}{}{x}(x) \dot{x}$ reduces to the following quadratic form in $T(x)$
\begin{equation*}
\pderiv{\pi}{}{x}(x) \dot{x}
=
T(x)^T 
\begin{bmatrix}
-1 & 0 \\
-\sqrt{3} & -2
\end{bmatrix}
T(x)
,
\end{equation*}
which is negative definite.
Therefore, the computed solutions $\pi^i$ obey the error bound of Theorem \ref{thm:cmptd_exct_err_bnd}.

We computed the patchy solution on 73 patches including the Al'brekht patch.
The maximum consecutive patch point distance was $h \approx .54$.
The degree of the computed optimal cost and optimal control polynomials was four and three.
To compute the absolute error, we created a $100 \times 100$ grid of equally spaced points in the square $-1 \le x_1, x_2 \le 1$, and then computed the difference between the exact and computed optimal cost at each grid point that fell inside a patch.
The maximum absolute error at a grid point was approximately $5 \times 10^{-3}$.
The theoretical absolute error bound grows along a sequence of consecutive patch points as $K(\tfrac{L^i-1}{L-1}+1)h^5 \approx K(\tfrac{L^i-1}{L-1}+1)(4.6 \times 10^{-2})$.
The computed and exact solutions are displayed in Figure \ref{fig:cmptd_exct_soln}.
The absolute error is displayed in Figure \ref{fig:abs_err}.

To verify the validity of the analytical error bound, we computed the patchy solution on five concentric level sets of the computed cost.
We doubled the number of patch points on each level set, starting with eight points on the boundary of the Al'brekht patch.
We then computed the absolute error of optimal cost at each patch point and then found the sequence of patch points that terminates with the greatest absolute error.
The level sets of the optimal cost behave badly near the singularities at $x_1 = \pm \pi/2$,
which limited the size of the computational domain for this test problem.
In this case, $h \approx .5$, so the absolute error bound grows as $K(\tfrac{L^i-1}{L-1}+1)h^5 \approx K(\tfrac{L^i-1}{L-1}+1)(3.1 \times 10^{-2})$.
The absolute error and the sequence of patch points yielding the worst case error are shown in Figure \ref{fig:err_growth}.

\begin{figure}[hb]
\fbox{\subfloat[Patchy Optimal Cost]{\includegraphics[scale=.5, angle=270]{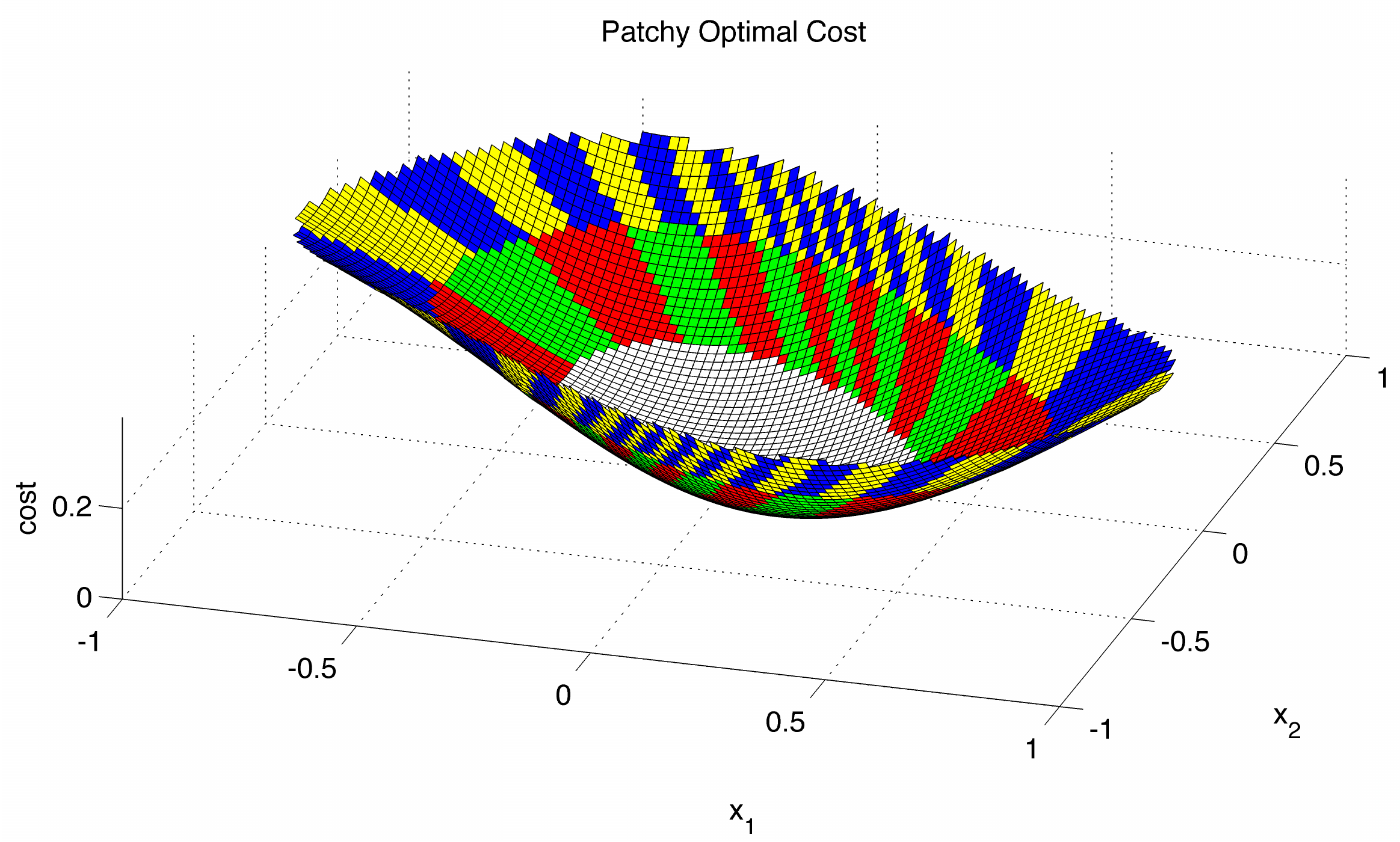}}}
\fbox{\subfloat[Exact Optimal Cost]{\includegraphics[scale=.5, angle=270]{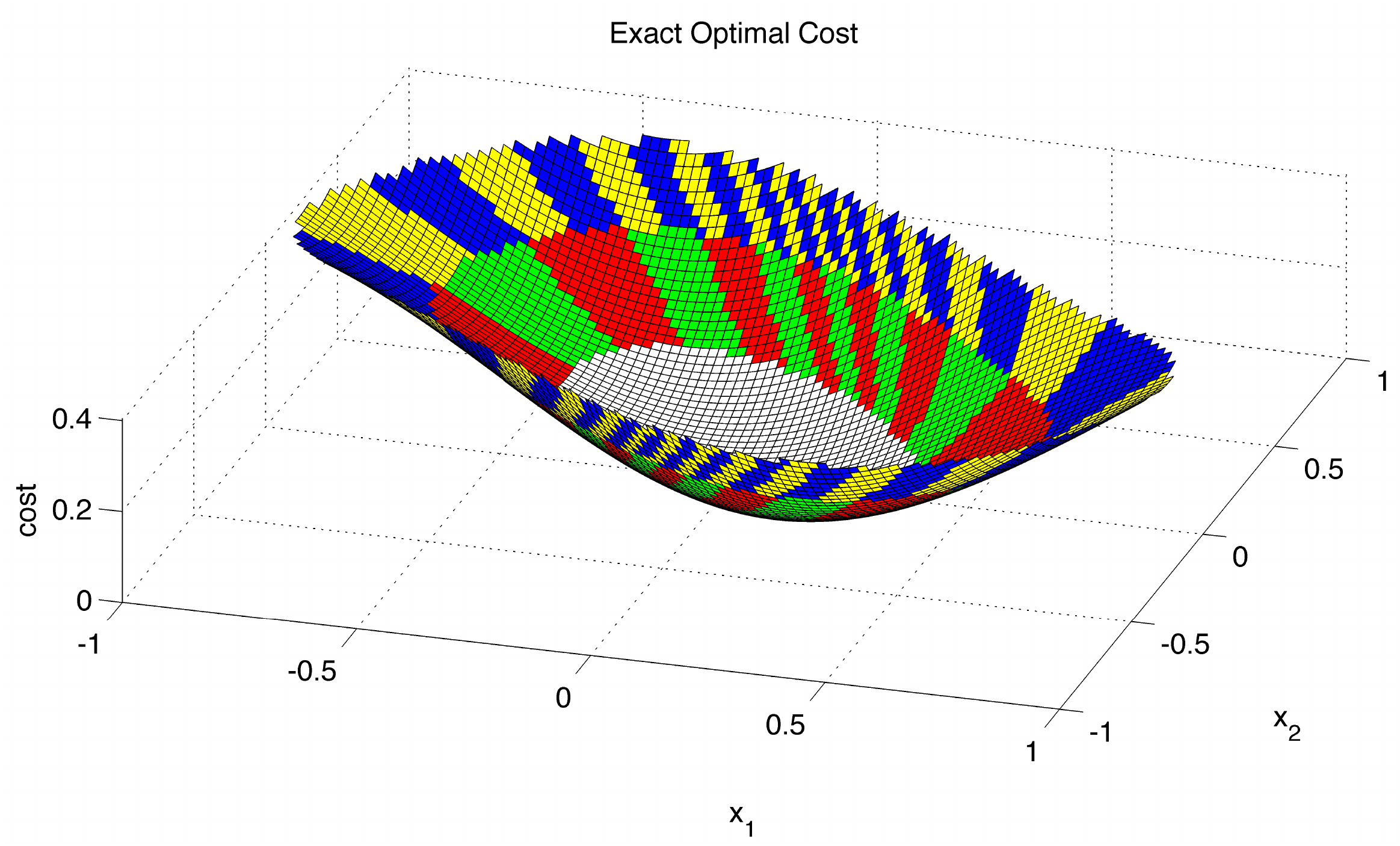}}}
\caption{Patchy and exact optimal costs}
\label{fig:cmptd_exct_soln}
\end{figure}

\begin{figure}[ht]
\centering
\fbox{\includegraphics[scale=.7, angle=270]{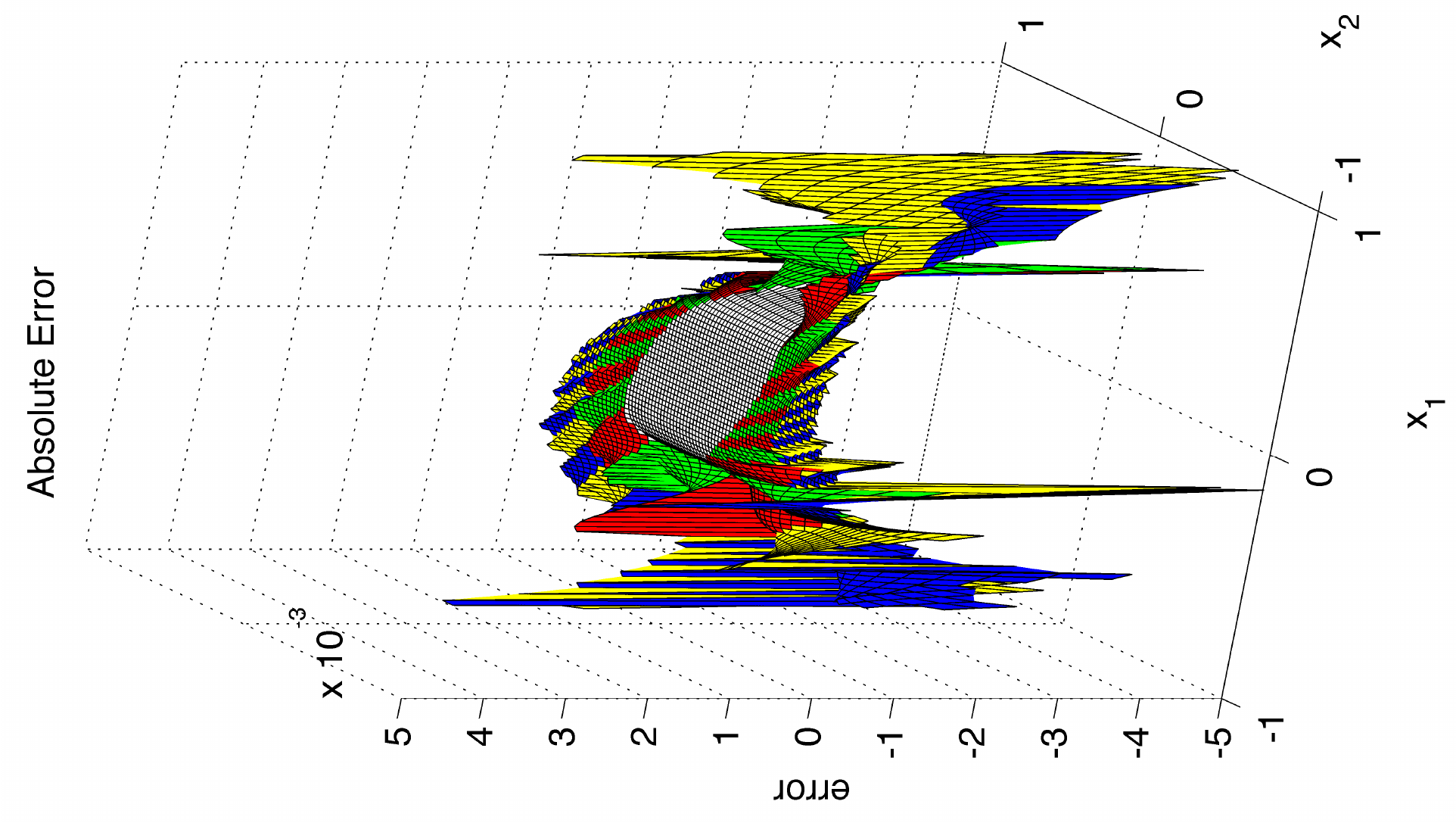}}
\caption{Absolute Error $\pi(x) - \pi^i(x)$}
\label{fig:abs_err}
\end{figure}

\begin{figure}[hb]
\fbox{\subfloat[Optimal cost absolute error at sequence of consecutive patch points.]{\includegraphics[width=4.5in]{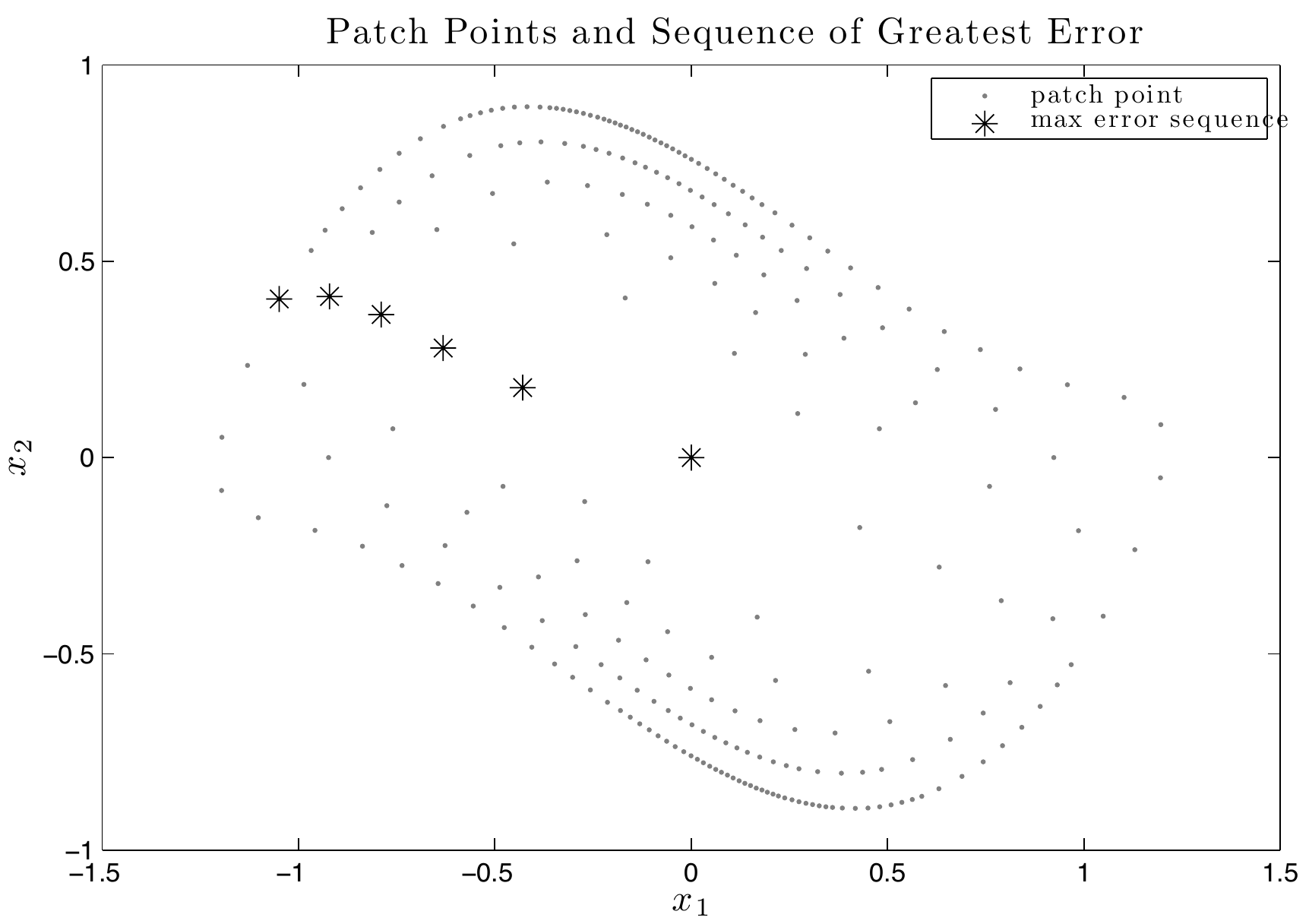}}}
\\
\fbox{\subfloat[Patch points and patch point sequence of greatest error]{\includegraphics[width=4.5in]{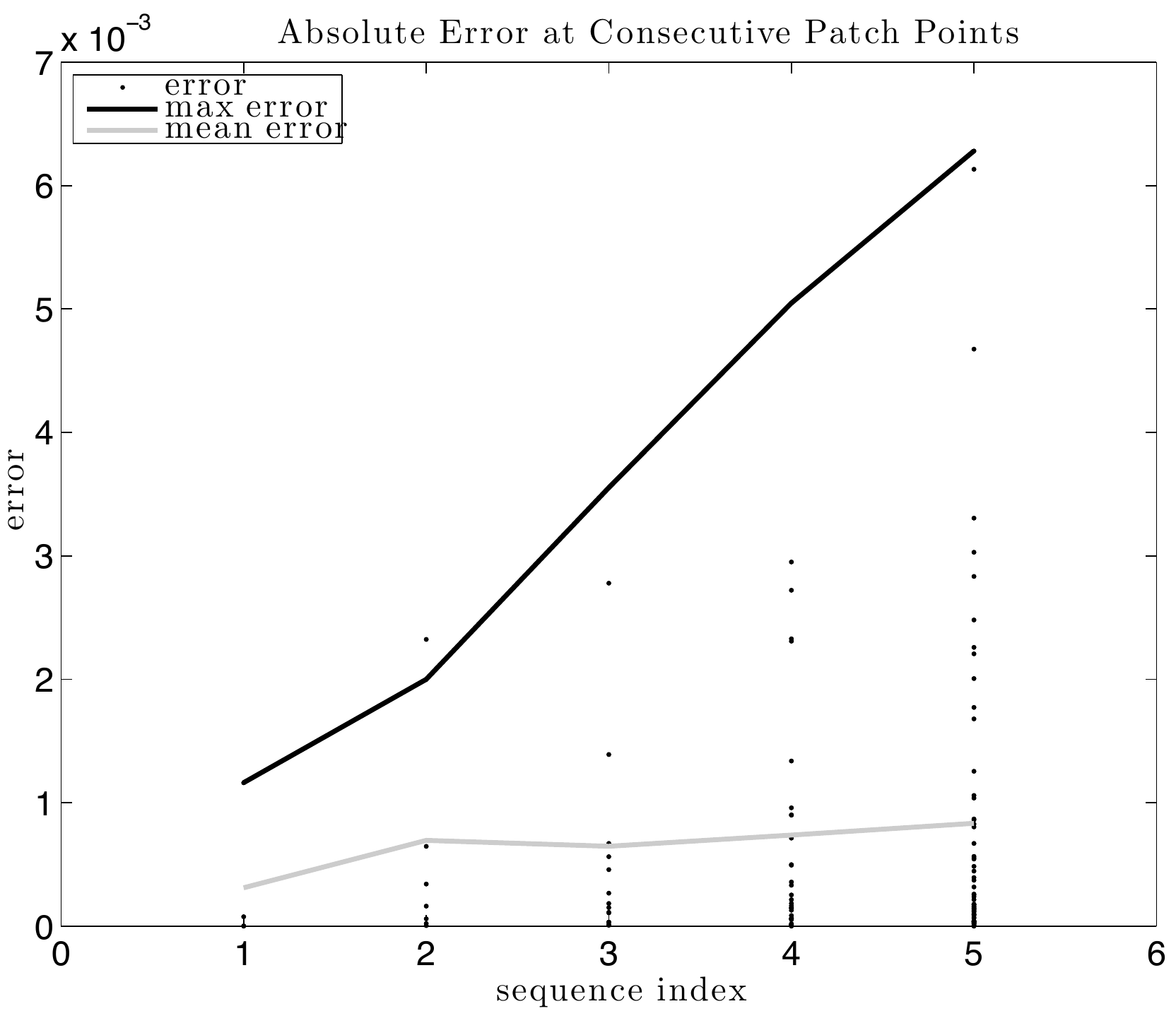}}}
\caption{Optimal cost absolute error at sequences of patch points. $h \approx .5$}
\label{fig:err_growth}
\end{figure}

\clearpage

\section{Appendix}

\subsection{Kronecker derivative notation}
\label{sec:kron_diff_op}
\begin{definition}
\label{defn:kron_diff_op}
Let $s : \mathbb{R}^n \rightarrow \mathbb{R}^{1 \times m}$ %
\begin{equation*}
s(x)
=
\begin{bmatrix}
s_1(x) & s_2(x) & \cdots & s_m(x)
\end{bmatrix}
,
\end{equation*}
The differential operator is $\tpderiv{}{}{x} \otimes s(x)$ is shorthand for the $1 \times mn$ row vector
\begin{equation*}
\pderiv{}{}{x} \otimes s(x)
\equiv 
\begin{bmatrix}
\pderiv{s}{}{x_1}(x) & \pderiv{s}{}{x_2}(x) & 
\cdots 
& \pderiv{s}{}{x_n}(x)
\end{bmatrix}
\end{equation*}
The higher order operator $\tpderiv{}{}{x} \otimes \cdots \otimes \tpderiv{}{}{x}$ is defined recursively in the natural way.
In the special case where $s : \mathbb{R}^2 \rightarrow \mathbb{R}$, then $\tpderiv{s}{}{x}  : \mathbb{R}^2 \rightarrow \mathbb{R}^{1 \times 2} $, and
\begin{align*}
 \pderiv{}{}{x} [s(x)]
&=
\begin{bmatrix}
\pderiv{s}{}{x_1}(x) & 
\pderiv{s}{}{x_2}(x) 
\end{bmatrix}
\\
\pderiv{}{}{x} \otimes \pderiv{}{}{x} [s(x)]
&\equiv 
\begin{bmatrix}
\pderiv{s}{2}{x_1}(x) & 
\ppderiv{s}{2}{x_1}{}{x_2}{}(x) & 
\ppderiv{s}{2}{x_2}{}{x_1}{}(x) &
\pderiv{s}{2}{x_2}(x) 
\end{bmatrix}
\\
\pderiv{}{}{x} \otimes \pderiv{}{}{x} \otimes \pderiv{}{}{x} [s(x)]
&\equiv 
\begin{matrix}
\Bigl[ 
\pderiv{s}{3}{x_1}(x) & 
\ppderiv{s}{3}{x_1}{2}{x_2}{}(x) & 
\ppderiv{s}{3}{x_1}{2}{x_2}{}(x) & 
\ppderiv{s}{3}{x_1}{}{x_2}{2}(x) & 
\cdots \\
\mspace{5 mu}\ppderiv{s}{3}{x_1}{2}{x_2}{}(x) & 
\ppderiv{s}{3}{x_1}{}{x_2}{2}(x) & 
\ppderiv{s}{3}{x_1}{}{x_2}{2}(x) & 
\ppderiv{s}{3}{x_2}{}{x_1}{}(x) &
\cdots \\
\pderiv{s}{3}{x_2}(x) 
\Bigr]
\end{matrix}
\\
&\mspace{11 mu} \vdots
\end{align*}
\end{definition}

\subsection{Partial derivative formulas for the optimal control and characteristic partial derivatives of the optimal cost} \label{sec:deriv_formulas}
\begin{align*}
\pppderiv{\hat{\pi}^{i+1}}{3}{\hatxi_{j_1}}{}{\hatxi_{j_2}}{}{\hatxi_1}{}(0)
=
-\frac{1}{\norm{\dot{x}^{i+1}}}
\Biggl[
&
\ppderiv{\hat{\pi}^{i+1}}{2}{\hatxi_{j_1}}{}{\hatxi}{}(0)
\Bigpars{
\pderiv{\hat{f}}{}{\hatxi_{j_2}}(0) + \pderiv{\hat{g}}{}{\hatxi_{j_2}}(0) \hat{\kappa}^{i+1}(0)}
\\
&
\begin{aligned}
&%
+ \ppderiv{\hat{\pi}^{i+1}}{2}{\hatxi_{j_2}}{}{\hatxi}{}(0)
\Bigpars{
\pderiv{\hat{f}}{}{\hatxi_{j_1}}(0) + \pderiv{\hat{g}}{}{\hatxi_{j_1}}(0) \hat{\kappa}^{i+1}(0)
}
\\
&
+ \pderiv{\hatpi^{i+1}}{}{\hatxi}(0)
\Bigpars{
\ppderiv{\hatf}{2}{\hatxi_{j_1}}{}{\hatxi_{j_2}}{}(0) + \ppderiv{\hatg}{2}{\hatxi_{j_1}}{}{\hatxi_{j_2}}{}(0) \hatkappa^{i+1}(0)
}
\\
&
+ \ppderiv{\hatq}{2}{\hatxi_{j_1}}{}{\hatxi_{j_2}}{}(0) 
+ \frac{1}{2} \ppderiv{\hatr}{2}{\hatxi_{j_1}}{}{\hatxi_{j_2}}{}(0) \hatkappa^{i+1}(0)^2
\\
&
- \pderiv{\hatkappa}{}{\hatxi_{j_1}}(0) \pderiv{\hatkappa}{}{\hatxi_{j_2}}(0)
\Biggr]
\end{aligned}
\end{align*}

\begin{align*}
\ppderiv{\hatkappa^{i+1}}{2}{\hatxi_{j_1}}{}{\hatxi_{j_2}}{}(0)
&=
-\frac{1}{\hatr(0)}
\Biggbracks{
\pppderiv{\hatpi^{i+1}}{3}{\hatxi_{j_1}}{}{\hatxi_{j_2}}{}{\hatxi}{}(0) \hatg(0)
+ \ppderiv{\hatpi^{i+1}}{2}{\hatxi_{j_1}}{}{\hatxi}{}(0) \pderiv{\hatg}{}{\hatxi_{j_2}}(0)
\\
&\phantom{=-\frac{1}{\norm{\dot{x}^{i+1}}} \Biggl[ \;} 
+ \ppderiv{\hatpi^{i+1}}{2}{\hatxi_{j_2}}{}{\hatxi}{}(0) \pderiv{\hatg}{}{\hatxi_{j_1}}(0)
\\
&\phantom{=-\frac{1}{\norm{\dot{x}^{i+1}}} \Biggl[ \;} 
+ \pderiv{\hatpi^{i+1}}{}{\hatxi}(0) \ppderiv{\hatg}{2}{\hatxi_{j_1}}{}{\hatxi_{j_2}}{}(0)
+ \ppderiv{r}{2}{\hatxi_{j_1}}{}{\hatxi_{j_2}}{}(0) \hatkappa^{i+1}(0)
\\
&\phantom{=-\frac{1}{\norm{\dot{x}^{i+1}}} \Biggl[ \;} 
+ \pderiv{\hatr}{}{\hatxi_{j_1}}(0) \hatkappa^{i+1}(0)
\\
&\phantom{=-\frac{1}{\norm{\dot{x}^{i+1}}} \Biggl[ \;} 
+ \pderiv{\hatr}{}{\hatxi_{j_1}}(0) \pderiv{\hatkappa^{i+1}}{}{\hatxi_{j_2}}(0)
+ \pderiv{\hatr}{}{\hatxi_{j_2}}(0) \pderiv{\hatkappa^{i+1}}{}{\hatxi_{j_1}}(0)
}
\end{align*}
\begin{multline*}
\frac{\partial^4 \hat{\pi}^{i+1}}{\partial \hatxi_{j_1} \partial \hatxi_{j_2} \partial \hatxi_{j_3} \partial \hatxi_1}(0)
= \\
\begin{aligned}
-\frac{1}{\norm{\dot{x}^{i+1}}}
\Biggl[
&
\pppderiv{\hat{\pi}^{i+1}}{3}{\hatxi_{j_1}}{}{\hatxi_{j_2}}{}{\hatxi}{}(0)
\Bigpars{
\pderiv{\hatf}{}{\hatxi_{j_3}}(0) + \pderiv{\hatg}{}{\hatxi_{j_3}}(0) \hatkappa^{i+1}(0)
} \\
&
+
\pppderiv{\hat{\pi}^{i+1}}{3}{\hatxi_{j_1}}{}{\hatxi_{j_2}}{}{\hatxi}{}(0)
\Bigpars{
\pderiv{\hatf}{}{\hatxi_{j_3}}(0) + \pderiv{\hatg}{}{\hatxi_{j_3}}(0) \hatkappa^{i+1}(0)
+\hatg(0) \pderiv{\hatkappa^{i+1}}{}{\hatxi_{j_2}}(0)
} \\
&
+
\ppderiv{\hat{\pi}^{i+1}}{2}{\hatxi_{j_1}}{}{\hatxi}{}(0)
\Bigpars{
\ppderiv{\hatf}{2}{\hatxi_{j_2}}{}{\hatxi_{j_3}}{}(0)
+\ppderiv{\hatg}{2}{\hatxi_{j_2}}{}{\hatxi_{j_3}}{}(0) \hatkappa^{i+1}(0)
+\pderiv{\hatg}{}{\hatxi_{j_3}}{}(0) \pderiv{\hatkappa^{i+1}}{}{\hatxi_{j_2}}{}(0) \\
& \phantom{\ppderiv{\hat{\pi}^{i+1}}{2}{\hatxi_{j_1}}{}{\hatxi}{}(0) \Bigl( \mspace{20 mu}}
+\hatg(0) \ppderiv{\hatkappa^{i+1}}{2}{\hatxi_{j_2}}{}{\hatxi_{j_3}}{}(0)
} \\
&
+
\ppderiv{\hat{\pi}^{i+1}}{2}{\hatxi_{j_3}}{}{\hatxi}{}(0)
\Bigpars{
\ppderiv{\hatf}{2}{\hatxi_{j_1}}{}{\hatxi_{j_2}}{}(0)
+\ppderiv{\hatg}{2}{\hatxi_{j_1}}{}{\hatxi_{j_2}}{}(0) \hatkappa^{i+1}(0)
} \\
&
+
\ppderiv{\hat{\pi}^{i+1}}{2}{\hatxi_{j_2}}{}{\hatxi}{}(0)
\Bigpars{
\ppderiv{\hatf}{2}{\hatxi_{j_1}}{}{\hatxi_{j_3}}{}(0)
+\ppderiv{\hatg}{2}{\hatxi_{j_1}}{}{\hatxi_{j_3}}{}(0) \hatkappa^{i+1}(0)
} \\
&
+
\pderiv{\hat{\pi}^{i+1}}{}{\hatxi}{}(0)
\Bigpars{
\pppderiv{\hatf}{3}{\hatxi_{j_1}}{}{\hatxi_{j_2}}{}{\hatxi_{j_3}}{}(0)
+\pppderiv{\hatg}{3}{\hatxi_{j_1}}{}{\hatxi_{j_2}}{}{\hatxi_{j_2}}{}(0) \hatkappa^{i+1}(0)
} \\
&
+ 
\pppderiv{\hatq}{3}{\hatxi_{j_1}}{}{\hatxi_{j_2}}{}{\hatxi_{j_3}}{}(0)
+ \frac{1}{2} \pppderiv{\hatr}{3}{\hatxi_{j_1}}{}{\hatxi_{j_2}}{}{\hatxi_{j_3}}{}(0) \bigpars{\hatkappa^{i+1}(0)}^2
\\
&
+
\ppderiv{\hatr}{2}{\hatxi_{j_1}}{}{\hatxi_{j_2}}{}(0) \hatkappa^{i+1}(0) \pderiv{\hatkappa^{i+1}}{}{\hatxi_{j_3}}{}(0)
+ \pderiv{\hatr}{}{\hatxi_{j_3}}{}(0) \pderiv{\hatkappa^{i+1}}{}{\hatxi_{j_1}}{}(0) \pderiv{\hatkappa^{i+1}}{}{\hatxi_{j_2}}{}(0) \\
&
-
\hatr(0) \ppderiv{\hatkappa^{i+1}}{2}{\hatxi_{j_1}}{}{\hatxi_{j_3}}{}(0) \pderiv{\hatkappa^{i+1}}{}{\hatxi_{j_2}}{}(0)
- \hatr(0) \pderiv{\hatkappa^{i+1}}{}{\hatxi_{j_1}}{}(0) \ppderiv{\hatkappa^{i+1}}{2}{\hatxi_{j_2}}{}{\hatxi_{j_3}}{}(0) \\
&
-
\Bigpars{
\pderiv{\hatr}{}{\hatxi_{j_2}}{}(0) \pderiv{\hatkappa^{i+1}}{}{\hatxi_{j_1}}{}(0)
+ \hatr(0) \ppderiv{\hatkappa^{i+1}}{2}{\hatxi_{j_1}}{}{\hatxi_{j_2}}{}(0) \pderiv{\kappa^{i+1}}{}{\hatxi_{j_3}}{}(0)
}
\Biggr]
\end{aligned}
\end{multline*}
\begin{multline*}
\pppderiv{\hatkappa^{i+1}}{3}{\hatxi_{j_1}}{}{\hatxi_{j_2}}{}{\hatxi_{j_3}}{}(0)
=
\\
\begin{aligned}
-\frac{1}{\hatr(0)}
\Biggl[
&
\frac{\partial^4 \hat{\pi}^{i+1}}{\partial \hatxi_{j_1} \partial \hatxi_{j_2} \partial \hatxi_{j_3} \partial \hatxi_{j_4}}(0) \hatg(0)
\\
&
+ 
\pppderiv{\hatpi^{i+1}}{3}{\hatxi_{j_1}}{}{\hatxi_{j_2}}{}{\hatxi}{}(0) \pderiv{\hatg}{}{\hatxi_{j_3}}(0)
+ \pppderiv{\hatpi^{i+1}}{3}{\hatxi_{j_1}}{}{\hatxi_{j_3}}{}{\hatxi}{}(0) \pderiv{\hatg}{}{\hatxi_{j_2}}(0)
\\
&
+ \ppderiv{\hatpi^{i+1}}{2}{\hatxi_{j_1}}{}{\hatxi}{}(0) \ppderiv{\hatg}{2}{\hatxi_{j_2}}{}{\hatxi_{j_3}}{}(0)
+ \ppderiv{\hatpi^{i+1}}{2}{\hatxi_{j_3}}{}{\hatxi}{}(0) \ppderiv{\hatg}{2}{\hatxi_{j_1}}{}{\hatxi_{j_2}}{}(0)
\\
&
+ \pderiv{\hatpi^{i+1}}{}{\hatxi}(0) \pppderiv{\hatg}{3}{\hatxi_{j_1}}{}{\hatxi_{j_2}}{}{\hatxi_{j_3}}{}(0) \\
&
+
\ppderiv{\hatr}{2}{\hatxi_{j_2}}{}{\hatxi_{j_3}}{}(0) \pderiv{\hatkappa^{i+1}}{}{\hatxi_{j_1}}(0)
+ \pderiv{\hatr}{}{\hatxi_{j_2}}(0) \ppderiv{\hatkappa^{i+1}}{2}{\hatxi_{j_2}}{}{\hatxi_{j_3}}{}(0) 
+ \pderiv{\hatr}{}{\hatxi_{j_3}}(0) \ppderiv{\hatkappa^{i+1}}{2}{\hatxi_{j_1}}{}{\hatxi_{j_2}}{}(0) 
\\
&
+ \hatr(0) \pppderiv{\hatkappa^{i+1}}{3}{\hatxi_{j_1}}{}{\hatxi_{j_2}}{}{\hatxi_{j_3}}{}(0)
\Biggr]
\end{aligned}
\end{multline*}

\clearpage

\bibliographystyle{siam}
\bibliography{continuous_time_patchy_references}

\begin{thebibliography}{10}

\bibitem{albrekht}
{\sc E.~G. Al'brekht}, {\em On the optimal stabilization of nonlinear systems},
  Journal of Applied Mathematics and Mechanics, 25 (1961), pp.~836 -- 844.

\bibitem{Ancona2007279}
{\sc Fabio Ancona and Alberto Bressan}, {\em Nearly time optimal stabilizing
  patchy feedbacks}, Annales de l'Institut Henri Poincare (C) Non Linear
  Analysis, 24 (2007), pp.~279 -- 310.

\bibitem{golub_van_loan:mat_comps}
{\sc Gene~H. Golub and Charles F.~Van Loan}, {\em Matrix Computations}, The
  John Hopkins University Press, 1996.

\bibitem{pde_text2}
{\sc Fritz John}, {\em Partial Differential Equations}, Springer-Verlag,
  4th~ed., 1982.

\bibitem{art_code}
{\sc Arthur~J. Krener}, {\em Nonlinear systems toolbox v. 1.0}.
\newblock Available by request from ajkrener@nps.edu.

\bibitem{doi:10.1137/0307007}
{\sc D.~L. Lukes}, {\em Optimal regulation of nonlinear dynamical systems},
  SIAM Journal on Control, 7 (1969), pp.~75--100.

\bibitem{krener_navasca_patchy}
{\sc Carmeliza Navasca and Arthur~J. Krener}, {\em Modeling, Estimation and
  Control: Festschrift in Honor of Giorgio Picci on the Occasion of his
  Sixty-Fifth Birthday}, Springer-Verlag, 2007.

\bibitem{rudin:real_anly}
{\sc Walter Rudin}, {\em Principles of Mathematical Analysis}, McGraw-Hill,
  Inc., third~ed., 1976.

\bibitem{sethian_text}
{\sc J.A. Sethian}, {\em Level Set Methods and Fast Marching Methods},
  Cambridge University Press, 1999.

\bibitem{stewart:mat_alg}
{\sc G.W. Stewart}, {\em Matrix Algorithms Volume 1: Basic Decompositions},
  SIAM, 1998.

\bibitem{stoer_bulirsch:nmrcl_anly}
{\sc Josef Stoer and Roland Bulirsch}, {\em Introduction to Numerical
  Analysis}, Springer, second~ed., 1992.

\bibitem{szpiro_dupuis}
{\sc Adam Szpiro and Paul Dupuis}, {\em Second order numerical methods for
  first order hamilton-jacobi equations}, SIAM Journal on Numerical Analysis,
  40 (2003), pp.~pp. 1136--1183.

\bibitem{412624}
{\sc J.N. Tsitsiklis}, {\em Efficient algorithms for globally optimal
  trajectories}, Automatic Control, IEEE Transactions on, 40 (1995), pp.~1528
  --1538.

\end{thebibliography}

\end{document}